\documentclass[11pt,a4paper]{amsart} 
\usepackage{amssymb,amsmath,bbm,mathrsfs}
\usepackage{epsfig}
\usepackage[usenames,dvipsnames]{xcolor}
\usepackage[colorlinks=true,linkcolor=NavyBlue,urlcolor=RoyalBlue,citecolor=PineGreen,bookmarks=true,bookmarksopen=true,bookmarksopenlevel=2,unicode=true,linktocpage]{hyperref}
\usepackage{tikz}

\newtheorem{theoremintro}{Theorem} 
\newtheorem{theorem}{Theorem}[section] 
\newtheorem{lemma}[theorem]{Lemma}
\newtheorem{proposition}[theorem]{Proposition}

\newtheorem{corollary}[theorem]{Corollary}

\theoremstyle{definition}
\newtheorem{definition}[theorem]{Definition}

\definecolor{bleu}{RGB}{45,75,135}

\def\1{{\mathbbm 1}}
\def\A{{\mathcal A}}
\def\amp{{\sf amp}}
\def\Aut{{\sf Aut}}
\def\b{{\sf b}}
\def\build#1_#2^#3{\mathrel{\mathop{\kern 0pt#1}\limits_{#2}^{#3}}}
\def\C{{\mathbb C}}
\def\cE{\mathcal{E}}

\def\d{{{\rm d}}}
\def\DL{{\mathsf{DL}}}
\def\DP{{\mathsf{DP}}}
\def\CL{{\mathsf{L}}}
\def\CLc{{\mathsf{L}^{\I}}}
\def\CLcp{{\mathsf{L}^{\I'}}}

\def\CP{{\mathsf{P}}}
\def\CPc{{\mathsf{P}^{\I}}}
\def\CPcp{{\mathsf{P}^{\I'}}}
\def\CPv{{\mathsf{P}^{\T}}}
\def\E{{\sf E}}
\def\EE{{\mathbf E}}

\def\Ec{{\E^{\circ}}}

\def\End{{\rm End}}
\def\epsilon{\varepsilon}
\def\Ex{{\mathbb E}}
\def\F{{\sf F}}
\def\G{{\sf G}}

\def\Gc{{\G^{\circ}}}

\def\H{{\mathcal H}}
\newcommand{\bH}{\mathbb{H}}

\def\HE{{\sf H}}
\def\Ht{{{\sf H}_{t}}}
\def\Hs{{{\sf H}_{s}}}
\def\hol{{\sf hol}}
\def\Hom{{\rm Hom}}
\def\I{{\mathbb{I}}}
\def\T{{\mathbb{T}}}
\def\Id{{\rm Id}}

\def\K{{\mathbb K}}
\def\L{{\mathcal L}}

\def\Leb{{\rm Leb}}

\def\mup{{\mu^{\begin{tikzpicture}  \filldraw (0,0) circle (0.02); \draw [draw=white] (0,-0.01) circle (0.05);\end{tikzpicture}}}}
\def\mul{{\mu^{\begin{tikzpicture} \draw (0,-0.01) circle (0.055); \end{tikzpicture}}}}
\def\mupl{\mu^{\begin{tikzpicture} \filldraw (0,-0.01) circle (0.02); \draw (0,-0.01) circle (0.055); \end{tikzpicture}}}
\def\muplc{\mu^{\begin{tikzpicture} \filldraw (0,-0.01) circle (0.02); \draw (0,-0.01) circle (0.055); \end{tikzpicture},\I}}
\def\muplp{\mu^{\begin{tikzpicture} \filldraw (0,-0.01) circle (0.02); \draw (0,-0.01) circle (0.055); \end{tikzpicture},\I'}}
\def\muplcp{\mu^{\begin{tikzpicture} \filldraw (0,-0.01) circle (0.02); \draw (0,-0.01) circle (0.055); \end{tikzpicture},\I,+}}
\def\muplcm{\mu^{\begin{tikzpicture} \filldraw (0,-0.01) circle (0.02); \draw (0,-0.01) circle (0.055); \end{tikzpicture},\I,-}}
\def\muplpp{\mu^{\begin{tikzpicture} \filldraw (0,-0.01) circle (0.02); \draw (0,-0.01) circle (0.055); \end{tikzpicture},\I',+}}
\def\muplpm{\mu^{\begin{tikzpicture} \filldraw (0,-0.01) circle (0.02); \draw (0,-0.01) circle (0.055); \end{tikzpicture},\I',-}}
\def\muplv{\mu^{\begin{tikzpicture} \filldraw (0,-0.01) circle (0.02); \draw (0,-0.01) circle (0.055); \end{tikzpicture},\T}}

\def\new#1{{\textcolor{black}{\em #1}}}
\def\ol{\overline}
\def\P{{\mathbb P}}
\def\PP{{\mathbf P}}

\def\phi{\varphi}
\def\Q{{\mathbb Q}}
\def\R{{\mathbb R}}
\def\S{{\sf S}}
\def\Tr{\mathrm{Tr}}
\def\tr{{\mathrm{tr}}}
\def\Un{{\mathrm{U}}}
\def\U{{\sf U}}
\def\ul{\underline}
\def\V{{\sf V}}

\def\W{{\sf W}}

\newcommand{\old}[1]{}

\title{Covariant Symanzik identities}

\author{Adrien Kassel}
\address{Adrien Kassel -- CNRS -- UMPA, ENS de Lyon}
\email{adrien.kassel@ens-lyon.fr}

\author{Thierry L\'evy}
\address{Thierry L\'evy -- LPSM, Sorbonne Universit\'e, Paris}
\email{thierry.levy@sorbonne-universite.fr}


\date{\today}

\begin{document}

\maketitle

\begin{abstract}
Classical isomorphism theorems due to Dynkin, Eisenbaum, Le Jan, and Sznitman establish equalities between the correlation functions or distributions of occupation times of random paths or ensembles of paths and Markovian fields, such as the discrete Gaussian free field. We extend these results to the case of real, complex, or quaternionic vector bundles of arbitrary rank over graphs endowed with a connection, by providing distributional identities between functionals of the Gaussian free vector field and holonomies of random paths. As an application, we give a formula for computing moments of a large class of random, in general non-Gaussian, fields\old{sections} in terms of holonomies of random paths with respect to an annealed random gauge field, in the spirit of Symanzik's foundational work on the subject.
\end{abstract}

\setcounter{tocdepth}{1}
{\footnotesize \tableofcontents}

\section*{Introduction}

Around 1980, following the pioneering work of Symanzik \cite{Symanzik} where he established, for the needs of Euclidean quantum field theory, a heuristic correspondence between random fields and random paths, Dynkin \cite{Dynkin-originel} initiated the mathematical study of the deep relations that exist between a Markovian field and a Markov process when they share the same generator. 

The joint study of random fields over a space and random paths in the same space may be seen as a probabilistic instance of the general principle that the geometry of a space can be understood by looking at good functions that live on it as well as by looking at its points.

The essence of the situation that gives rise to the correspondence investigated by Symanzik, Dynkin and many others after them, is the existence, on a topological or geometrical space, of a differential operator (say a Laplacian) which determines on one hand a random field on the space (the Gaussian free field) and on the other hand a random motion within the space (the Brownian motion). Both random objects express something of the geometry of the underlying space, and the general problem is to understand how the properties of the one mirror the properties of the other. 

A fairly general framework where a Laplacian operator is defined is that of a Riemannian manifold. A natural extension of this framework is that of a Euclidean or Hermitian vector bundle over a Riemannian manifold, endowed with a connection. In this case, the Laplacian acts on sections of the vector bundle, and the role of the random field will be played by a random \new{vector} field.

The goal of this paper is to investigate some of the relations between random vector fields and random paths in a discrete version of the setting that we just described, namely on a vector bundle with a connection \new{over a graph}. There, the discrete nature of space keeps the analysis simple, making for instance any procedure of renormalisation irrelevant, but the presence of the vector bundle lets interesting geometry enter the picture. In this framework, we will give appropriate formulations and complete self-contained proofs of the natural analogues of four classical theorems, sometimes called `isomorphism theorems', due to Dynkin, Eisenbaum, Le Jan, and Sznitman, and of a fifth result in the original spirit of Symanzik. We insist that these classical statements are not used in our proofs, and that the proofs that we propose also provide a geometrical insight into these classical statements. In their original form, these theorems relate occupation times of various random paths and functions of the square of the Gaussian free field. An interesting feature of the extensions of these results that we propose is the prominent role played not only by random paths on the graph, but also by the \new{parallel transport} determined along them by the connection.

Let us illustrate, without too much formal preparation, and on the example of an identity that is fundamental to our study, the role played by the parallel transport. Suppose $x$ and $y$ are vertices of our graph. Let $\P_{x}$ denote the distribution of the natural continuous time random walk issued from $x$. Let $t$ be a positive real. If $\gamma$ is a path in the graph, we denote by $\gamma_{t}$ its position at time $t$, and by $\hol(\gamma_{|[0,t)})$ the holonomy, or parallel transport, along the restriction to $[0,t)$ of $\gamma$. If $\gamma_{t}=y$, and provided $t$ is not a jump time of $\gamma$, this is a linear isometry from the fibre over $y$ to the fibre over $x$. Then, in a matricial expression of the operator $e^{-t\Delta}$ acting on the space of vector fields over the graph, that is, on the direct sum of all the fibres of the bundle, the block corresponding to a linear map from the fibre over $y$ to the fibre over $x$ can be expressed as
\begin{equation}
\big(e^{-t\Delta}\big)_{x,y} = \int \hol(\gamma_{|[0,t)})^{-1} \1_{\{\gamma_{t}=y\}}\; \d\P_{x}(\gamma).
\end{equation}
This equation is indeed the simplest case of a (covariant) Feynman--Kac formula that will be one of our main tools in proving our five main results.

\subsection*{Overview of the results}

In order to give a more precise idea of at least some of these results, it is necessary to be a little more specific about our framework. We consider a real, complex or quaternionic vector bundle of arbitrary rank over a graph endowed with a unitary connection. In this introduction, for the sake of simplicity, we will focus on the real case. The connection, in this case, consists in an orthogonal matrix over each oriented edge of the graph, called the holonomy along this edge. The holonomies along an edge and the same edge with the opposite orientation are inverse of each other (see Figure \ref{fig:setup} below). We denote the connection by $h$. The holonomy $\hol(\gamma)$ along a path $\gamma$ is the ordered product of the holonomies along the edges traversed by the path. It is an isometry from the fibre over the starting point of the path to the fibre over its finishing point.

\begin{figure}[h!]
\begin{center}
\includegraphics{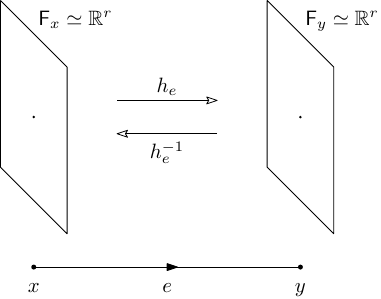}
\end{center}
\caption{\small A discrete real vector bundle $\F$ of rank $r$ over a graph with an orthogonal connection $h$.}\label{fig:setup}
\end{figure}

The connection $h$ gives rise to a quadratic form  
\[\mathcal{E}(f)=\frac{1}{2}\sum_{e=xy} \| f_x -h_{e}^{-1} f_y\|^2\]
on the space of vector fields over the graph, the sum being extended to all oriented edges of the graph. To this quadratic form is associated a centred Gaussian random vector field $\Phi$, called the Gaussian free vector field (see Figure \ref{fig:imageGFVF}). Let us stress that except for very specific choices of the connection, there do not exist bases of the fibres in which the components of the field would be independent Gaussian free fields. 

The covariance of the Gaussian free vector field has an expression in terms of random walk. Indeed, if $x$ and $y$ are vertices, and if $\xi$ and $\eta$ are respectively vectors of the fibres over $x$ and $y$, then
\begin{equation}\label{eq:covariance}
\Ex[\langle\xi,\Phi_x\rangle \langle \Phi_y,\eta\rangle ]=\int \langle \xi,\hol(\gamma)^{-1} (\eta)\rangle\; \d\nu_{x,y}(\gamma), 
\end{equation}
where $\nu_{x,y}$ is a finite measure on the set of paths from $x$ to $y$.

By virtue of Wick's formula, one can turn \eqref{eq:covariance} into an expression for the moments of the field~$\Phi$, that is, for quantities of the form $\Ex\left[ \left(f_1,\Phi\right) \cdots \left(f_k,\Phi\right)\right]$, where $f_{1},\ldots,f_{k}$ are vector fields. This expression is an integral of products of matrix coefficients of holonomies along paths with respect to a multi-bridge measure.

All the results presented in this paper will follow more or less directly from identities slightly more general than, but similar to \eqref{eq:covariance}. For instance, a special case of our analogue of Dynkin's theorem (Theorem \ref{thm:Dynkin}) goes as follows. We denote by $\ell_{z}(\gamma)$ the time spent by a path $\gamma$ at a vertex $z$.

\begin{theoremintro}\label{thmintro:Dynkin} Let $\xi$ and $\eta$ be vectors of the fibres over two vertices $x$ and $y$. For all vertex $z$, let~$\beta_{z}$ be a real parameter. Then 
\[
\Ex\left[ \int e^{-\frac{1}{2}\sum_{z} \beta_{z}( \|\Phi_{z}\|^{2}+\ell_{z}(\gamma))}\langle \xi,\hol(\gamma)^{-1}(\eta)\rangle\;\d\nu_{x,y}(\gamma)\right]=\Ex\left[e^{-\frac{1}{2}\sum_{z} \beta_{z} \|\Phi_{z}\|^{2}} \langle\xi,\Phi_x\rangle \langle \Phi_y,\eta\rangle\right].
\]
\end{theoremintro}

By specialising this result to the case of the trivial bundle of rank $1$, one recovers Dynkin's original statement. On the other hand, our full result is more general than Theorem \ref{thmintro:Dynkin} in that the parameter $\beta_{z}$, which can be understood as a potential at the vertex $z$, can be replaced by a non-negative self-adjoint operator on the fibre over $z$. The definition of the holonomy must be twisted to fit this generalisation, see Definition \ref{def:twisted-holo}.

\begin{figure}[h!]
\begin{center}
\includegraphics[width=15cm]{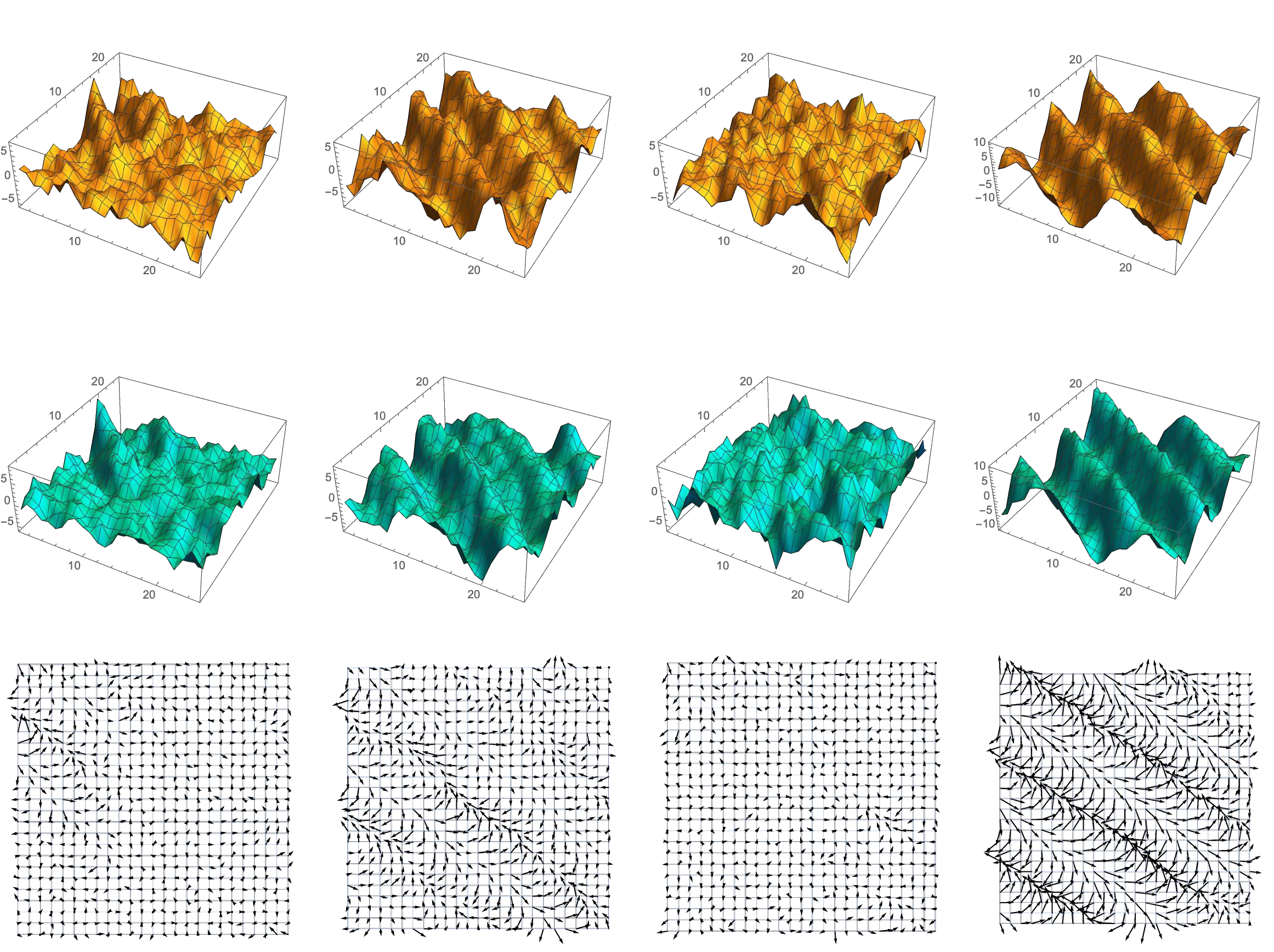}
\caption{\small \label{fig:imageGFVF} This picture shows four samples of the rank $2$ Gaussian free vector field $\Phi$ on the square grid $\{1,\ldots,25\}^{2}$, pinned at the top right corner. The parallel transport along each horizontal edge from left to right is the rotation of angle $\pi/7$, and the parallel transport along each vertical edge from top to bottom is the rotation of angle $\pi/5$.  For each sample, the top picture shows the first component of $\Phi$, the middle picture its second component, and the bottom picture shows the vector field itself.
The picture illustrates the fact that the two components of the field are correlated. Indeed, in certain samples such as the fourth, the field is visibly structured by the geometry of the connection, whereas in other samples, this structure is less apparent. The point is that the two coordinates of the field seem to consistently exhibit a comparable level of structuration. 
}
\end{center}
\end{figure}

Let us state a special case of another of our results, namely Le Jan's isomorphism (Theorem~\ref{thm:LeJanSznitman}). This result involves a $\sigma$-finite measure $\mu$ on the set of loops in the graph, a precise definition of which is given in Definition \ref{defmu}. From a real bundle $\F$, we form the bundle $\End(\F)$, whose fibre over a vertex $x$ is the space of endomorphisms of $\F_{x}$. The connection $h$ induces by conjugation a connection $h^{*}\otimes h$ on $\End(\F)\simeq \F^{*}\otimes \F$ and for all path $\gamma$ joining $x$ to $y$, the holonomy of $h^{*}\otimes h$ along $\gamma$ sends $u\in \End(\F_{x})$ to $\hol_{h}(\gamma)\circ u \circ \hol_{h}(\gamma)^{-1}\in \End(\F_{y})$. We denote the associated Gaussian free field of endomorphisms by $x\mapsto \Psi_{x}$.

\begin{theoremintro} \label{thmintro:LeJan} The field  $x\mapsto \frac{1}{2}\Tr(\Psi_{x}^{*} \Psi_{x})$
has the same distribution as the occupation field of a Poisson point process of loops with intensity $\frac{1}{2} \Tr(\hol_h(\cdot))^2 \mu$.
\end{theoremintro}

What makes this setting particularly pleasant is a positivity property, namely the positivity of the traces of the holonomies along loops in the bundle of endomorphism. A similar statement holds when $\F$ is a complex bundle, with the intensity of the Poisson point process multiplied by~$2$. In this complex case, the bundle of endomorphisms splits into the two real sub-bundles of Hermitian and skew-Hermitian operators. The connection preserves each of these bundles and retains on each of them the positivity property of the traces of holonomies (see Section \ref{sec:LJpositif}). The Gaussian field of Hermitian operators was considered and studied by Lupu in his work \cite{Lupu-Dynkin} following a first pre-published version of ours.

In our most general statement (see Theorem \ref{thm:LeJanSznitman}), this positivity property does not hold, and we must find a way of dealing with signed measures, ultimately introducing two Poisson point processes corresponding to the positive and negative parts of the measure.

Let us mention that after this work was first prepublished, \cite{Lawler-Panov} provided a combinatorial proof of the complex case of rank $1$ of Le Jan's theorem, in a formulation that does not involve Poisson point processes and thus bypasses the splitting mentioned above.

\subsection*{Background and related works} Let us now give some pointers to the literature and discuss some motivations of the present work.

\subsubsection*{Isomorphism theorems}

The classical subject of isomorphism theorems originated in an approach of Symanzik to constructive Euclidean quantum field theory~\cite{Symanzik}, also advocated by Nelson~\cite{Nelson}, which was initially implemented by Brydges, Fr\"{o}hlich, Spencer, and Sokal~\cite{BFS1,BFS2} in the context of classical spin systems, and later developed by Dynkin who expressed these physics ideas in the formal language of Markov processes~\cite{Dynkin-Markov, Dynkin,Dynkin-localtime}.

Further developments followed, with the isomorphisms of Eisenbaum~\cite{Eisenbaum}, Le Jan~\cite{LeJan-1}, and Sznitman~\cite{Sznitman}, among others. See~\cite{ Marcus-Rosen,LeJan-book,Sznitman-book} and the references therein for expository accounts of these results and their relation to the earlier Ray--Knight theorems. The introduction of the Brownian loop soup by Lawler and Werner~\cite{Lawler-Werner} was followed by~\cite{LeJan-2}, and an isomorphism for Brownian interlacement was given by Sznitman~\cite{Sznitman-continuum}. The whole subject has seen further recent developments, attested among many other papers by \cite{Lupu,Werner,Qian-Werner,LupuWerner,Zhai,Ding,SabotTarres,SznitmanRefinedIsomorphism,Abacherli-Sznitman}.

Very recently, in \cite{Helmuth}, the authors considered other geometries for the target space of the Gaussian field than the Euclidean one, namely spherical and hyperbolic geometries. They provide a unified exposition of three Dynkin isomorphism theorems relating fields to not necessarily Markovian random paths, namely the simple, vertex-reinforced, and vertex-diminished random walks. 

\subsubsection*{Lattice gauge theories}
The geometric framework of connections on graphs is the classical framework of lattice gauge theory, which gives a finite approximation scheme to continuous gauge theories, such as the Yang--Mills theory; see \cite{Levy-survey} for a recent survey of progress on Yang--Mills theory in dimension~$2$. One long-term general motivation for our study is the construction of a tractable probabilistic model for the interaction of a gauge field with a fermionic field, that is, for a random connection coupled to a model of statistical physics. The Gaussian free vector fields and the loop soups appearing in this paper would rather qualify as bosonic, but there are hints that they are related to the fermionic side. For instance, one way to describe Le Jan's loop measure is by means of Wilson's algorithm for sampling a uniform spanning tree by erasing loops from random walk trajectories (\cite{LeJan-book, Lawler-survey}). The distribution of the ensemble of erased loops is related to Le Jan's measure, and the uniform spanning tree, of determinantal nature, could qualify as fermionic matter.

In a work in preparation \cite{KL4} we introduce an analogue of uniform spanning tree measures in the context of vector bundles over graphs and discuss links to electrical networks and random walk holonomies. A preview of the definition of these quantum spanning forests is available in \cite[Section 1.5]{KL3}. Spanning forests on vector bundles of real, complex, or quaternionic rank~$1$ were investigated by Kenyon \cite{Kenyon} (see also \cite{Kassel-Kenyon}), and the slightly different but related case of loop soups on graph coverings was considered in a recent work by Le Jan~\cite{LeJan16}. 

In the work of physicists, it is customary that computations involving matrix models, connections, and gauge symmetry, can be expressed in terms of random surfaces or topological expansions. In the present paper, we are dealing with random paths rather than surfaces. Nevertheless, in the framework considered by Lupu in \cite{Lupu-Dynkin} and described before Theorem \ref{thmintro:LeJan}, topological expansions over maps do appear.  Another way in which such higher dimensional expansions arise is by considering Gaussian random forms on the cells of a simplicial complex and relating them to simplicial random walks \cite{KR}. A combinatorial relationship between vector bundles over graphs and simplicial complexes of higher dimension was also suggested in \cite[Section 7.1]{KL2}.

\subsection*{Organisation of the paper}\label{sec:intro-organisation}

Sections~\ref{sec:graphs} and \ref{sec:bundles} are devoted to a detailed description of the basic setup of vector bundles over graphs which we use throughout. The main novelty consists in the introduction, in Section~\ref{sec:bundles}, of a notion of twisted holonomy (Definition~\ref{def:twisted-holo}) which, we believe,
is a fruitful extension in our geometric framework of classical exponential functionals of local times. 

One of the most useful result for our study is a Feynman--Kac formula (Theorem~\ref{prop:feynman-kac}) which we prove in Section~\ref{sec:covFK}. This formula can be thought of as a discrete analogue of a continuum version which can be traced back at least to the work of Norris~\cite{Norris} (see also~\cite{Albeverio}) in stochastic differential geometry, and which can also be found, in a different guise for discrete time walks, in earlier works of Brydges, Fr\"{o}hlich, and Seiler~\cite{BFS} on lattice gauge theories as we discovered after completing our work. Batu G\"{u}neysu also kindly informed us of related work in \cite{Guneysu}.

The main results of the paper are presented in Sections \ref{sec:Dynkin}, \ref{sec:LJS}, and \ref{sec:Symanzik}. Section~\ref{sec:Dynkin} presents a very general formulation of both Dynkin's (Theorem~\ref{thm:Dynkin}) and Eisenbaum's (Theorem~\ref{thm:Eisenbaum}) isomorphism theorems. Our formulation is targeted at giving more insight into the relation between these isomorphism theorems and the probabilistic formulation of potential theory, see for instance \eqref{diffinvlap}. We believe that a part of the meaning of these isomorphisms appears more clearly in our general framework than in the classical case.
In any case, the relations that we prove between twisted holonomies and fields extend
the known relations between random paths and Gaussian fields, since the class of functionals obtained from twisted holonomies captures more faithfully, thanks to the non-commutativity of the gauge groups, the actual geometry of paths, whereas local times ignore much of the chronological unfolding of events in a trajectory.

Section~\ref{sec:LJS} presents a unified formulation (Theorem~\ref{thm:LeJanSznitman}) of the isomorphisms of Le Jan and Sznitman in the context of vector bundles. The most general statement is rather involved, with notions such as splittings, colours decompositions, signed measures, and amplitudes, all introduced gradually. We also present a special case of the statement (Theorem \ref{thm:LeJanSznitmanpositif}) which is much more concise and holds for certain type of connections which we call trace-positive.

Finally, Section~\ref{sec:Symanzik} builds on the previous sections to prove a statement (Theorem~\ref{thm:Symanzik}) relating moments of a large class of random sections to holonomies of random paths under an annealed gauge field. We call these formulas covariant Symanzik identities.

\subsection*{Acknowledgements} 
We thank Alain-Sol Sznitman for stimulating discussions on the subject.
We thank, for their hospitality and support, the Forschungsinstitut f\"{u}r Mathematik (FIM) in Zurich, the Centre Interfacultaire Bernoulli (CIB) in Lausanne, and the ETH Zurich where parts of this work were completed.

\section{Graphs, paths, and loops}\label{sec:graphs} \label{sec:paths}

\subsection{Graphs}  \label{sec:defgraph}
In this first section, we give a precise definition of the graphs that we will consider throughout this paper. One aspect in which this definition differs from that used for example in \cite{Sznitman-book} or \cite{LeJan-book} is that we allow more than one edge between two vertices of a graph. We also allow loop edges, that is, edges with identical end points.  
\begin{definition}\label{defgraph}
A \new{graph} is a quintuple $\G=(\V,\E,s,t,i)$ consisting of 
\begin{itemize}
\item two non-empty finite sets $\V$ and $\E$,
\item two maps $s,t:\E\to \V$, and
\item a fixed-point free involution $i$ of $\E$ such that $t\circ i=s$ holds.
\end{itemize}
\end{definition}

The elements of $\V$ are the \new{vertices} of the graph, the elements of $\E$ its \new{edges}. Each edge $e$ is oriented and joins its source $s(e)$ to its target $t(e)$. 

The involution $i$ induces an equivalence relation on $\E$ and the equivalence classes of this relations, that is, the pairs $\{e,i(e)\}$, are called the \new{geometric edges} of the graph. The set of geometric edges of the graph is denoted by $[\E]$.

We will often write $\ul{e}=s(e)$, $\ol{e}=t(e)$ and $e^{-1}=i(e)$ for the source, target, and inverse of an edge. We will also write a graph simply as a pair $(\V,\E)$ and not mention the maps $s,t,i$.

\begin{figure}[h!]
\begin{center}
\includegraphics[height=3.5cm]{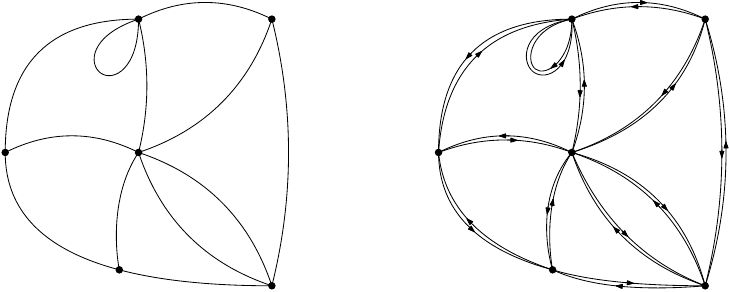}
\caption{\label{exgraph}\small The graph depicted on the left has 6 vertices and 12 geometric edges. The picture on the right shows the 24 edges of this graph.}
\end{center}
\end{figure}

We say that the graph $\G=(\V,\E)$ is \new{connected} if for all $x,y\in \V$, there exists a well-chained alternating sequence $(x_{0},e_{1},x_{1},\ldots,e_{n},x_{n})$ of vertices and edges such that $x_{0}=x$, $x_{n}=y$ and for all $k\in \{1,\ldots,n\}$, the edge $e_{k}$ starts at $x_{k-1}$ and finishes at $x_{k}$. In this paper, we will only work with connected graphs. 

\begin{definition} In a graph $\G=(\V,\E)$, we call \new{well} a proper subset $\W$ of the set $\V$ of vertices. The elements of $\V\setminus \W$ are called the \new{proper} vertices of the graph. 
The subset
\[\partial(\V \setminus \W)=\{x\in \V \setminus \W : \exists e\in \E, s(e)=x \mbox{ and } t(e)\in \W\}\]
of $\V$ is called the \new{rim} of the well. 
\end{definition}

A typical situation giving rise to graphs with wells is that where an infinite graph is exhausted by an increasing sequence $\V_{1}\subset \V_{2}\subset \ldots$ of finite subgraphs. Then, for each $n\geq 1$, one can consider the graph with vertex set $V_{n+1}$, with the edges inherited from the infinite ambient graph, and with well $V_{n+1}\setminus \V_{n}$.

\subsection{Weights and  measures} \label{wandm}
We will gradually introduce a certain amount of structure on the graphs that we consider. We start with a weighting of the edges.
\begin{definition}
Let $\G=(\V,\E)$ be a graph. A \new{conductance} on $\G$ is a positive real-valued function $\chi:\E\to (0,+\infty)$ such that $\chi\circ i=\chi$.
\end{definition} 

We think of $\chi$ as a measure on the set of edges of the graph. We shall usually write $\chi_{e}$ instead of $\chi(e)$ for the conductance of an edge $e$. The pair $(\G,\chi)$ is called a \new{weighted graph}.

Let $\G=(\V,\E)$ be a graph with a well endowed with a conductance $\chi$. We define the \new{reference measure} on $\G$ as the function $\lambda:\V\to (0,+\infty)$ defined by
\begin{equation}\label{refmeas}
\forall x\in \V, \ \lambda_{x}=\sum_{e\in \E : \ul{e}=x}\chi_{e}.
\end{equation}

Let us emphasise that $\lambda$ takes a positive value at every vertex, even on the vertices of the well.

The part of $\lambda$ that is due to the conductances of edges joining a proper vertex to a vertex of the well plays a special role. We define, for every proper vertex $x$,
\[\kappa_{x}=\sum_{e\in \E : \ul{e}=x, \ol{e}\in \W} \chi_{e}.\]
By definition, the support of the measure $\kappa$, that is, the set of vertices $x$ such that $\kappa_{x}>0$, is the rim of the well. 

In this paper and unless explicitly stated otherwise, by a \new{weighted graph with a well}, we will always mean the data of the collection of objects $(\W\subset\V,\E,s,t,i,\chi,\lambda,\kappa)$ as defined in this section, and with this notation. We will moreover always assume that the graph is connected, in the sense explained above.
\medskip

Let us conclude this section by a few terminological remarks. Firstly, the reader may find that there is an absence of distinction between measures and functions in our presentation, as we took advantage of the discrete nature of our framework to identify the measures $\chi,\lambda,\kappa$ with their densities with respect to the counting measures on $\E$ and $\V$. We adopted these identifications because they allow for a simpler and lighter notation, but we also believe that they hide to some extent the true nature of the objects which one manipulates (and which would appear more clearly in a continuous set-up). Therefore, we invite the reader to check periodically that our statements are consistent in this respect and that our formulas are, in the physical sense, homogeneous. 

Secondly, the objects that we are considering are very classical and receive several different names in the literature. The main reason for introducing them with care is that when dealing with twisted holonomies, one has to be careful with the precise definition of paths (when there are multiple edges, and when there are edges leading to the well, which we also consider backwards).
The vertex or set of vertices that we call the well is sometimes called the \new{sink}, 
and sometimes also the \new{cemetery}, and the authors who use the latter name call the quantity that we denote by $\kappa$ the \new{killing rate}. We prefer to avoid a morbid terminology, but respect the tradition and keep the notation $\kappa$.\footnote{We do find wells and their rims more inspiring than sinks and cemeteries, following Neruda \cite{Neruda}: \emph{Hay que sentarse a la orilla/del pozo de la sombra/y pescar luz ca\'ida/con paciencia}; {[\emph{We must sit on the rim/of the well of darkness/and fish for fallen light/with patience}]}.} An electric analogy is also often used in this context, accounting for the name of the conductance. In this terminology, the well should be though of as an electric \new{ground}. This may be the place to mention that, depending on the analogy which one choses, the conductance of an edge can be understood as the inverse of an electric resistance, the inverse of a length, or the section of a water pipe. 

\subsection{Discrete paths}\label{sec:discrete paths} 
Paths in graphs will play a crucial role in this study. The reader will soon notice that it is not only important for us to know which vertices a path visits, but also which edges it traverses, and because this is not the most widespread point of view, we give complete definitions of the objects that we consider. We will mainly study paths indexed by continuous time, but it is convenient to define paths indexed by discrete time first.

Let us fix a graph $\G=(\V,\E)$.

\begin{definition}
Let $n\geq 0$ be an integer. The set of \new{discrete paths} of length $n$ in $\G$ is the set
\begin{equation}\label{defDPk}
\DP_{n}(\G)=\big\{(x_{0},e_{1},x_{1},\ldots,e_{n},x_{n}) \in \V \times (\E\times \V)^{n} : \forall k\in \{1,\ldots,n\}, \ul{e_{k}}=x_{k-1} \mbox{ and } \ol{e_{k}}=x_{k}\big\}.
\end{equation}
The set of discrete paths on $\G$ is the disjoint union
\begin{equation}\label{defDP}
\DP(\G)=\bigcup_{n\geq 0} \DP_{n}(\G).
\end{equation}
\end{definition}

For example, a discrete path of length $0$ is simply a vertex, and a discrete path of length $2$ is a well-chained sequence $(x_{0},e_{1},x_{1},e_{2},x_{2})$ of vertices and edges. The \new{length} of a path, sometimes called its \new{combinatorial length}, is the number of edges traversed by this path. In this work, as the previous definition indicates, we will only consider \emph{finite} discrete paths.

A discrete path of positive length $k$ is of course completely characterised by the sequence of~$k$ edges which it traverses. Nevertheless, despite the fact that edges will play for us a more important role than is usual (because the parallel transport along parallel edges will not be assumed to be equal), we will still pay a lot of attention to the sequence of vertices visited by paths, and it is useful to keep explicit track of this information in their definition.

The initial and final vertex of a discrete path $p=(x_{0},e_{1},\ldots,x_{n})$ are respectively  denoted by $\ul{p}=x_{0}$ and $\ol{p}=x_{n}$. 

\begin{definition} A \new{discrete loop} is a discrete path with identical initial and final vertices. The set of discrete loops of length $n$ is denoted by $\DL_{n}(\G)$ and the set of all discrete loops by
\[\DL(\G)=\bigcup_{n\geq 0} \DL_{n}(\G).\]
\end{definition}

Note that $\DL_{0}(\G)=\DP_{0}(\G)= \V$.

\subsection{Continuous paths}\label{sec:CP}  A path indexed by continuous time is a discrete path which spends, at each vertex that it visits, a certain positive amount of time. Here is the formal definition.

\begin{definition}\label{def CP} Let  $\G=(\V,\E)$ be a graph. A {\em continuous path} in $\G$ is an element of the set
\[\CP(\G)=\bigcup_{n\geq 0} \left(\DP_{n}(\G)\times (0,+\infty)^{n}\times (0,+\infty]\right).\]
A {\em continuous loop} in $\G$ is an element of the set
\[\CL(\G)=\bigcup_{n\geq 0} \left(\DL_{n}(\G)\times (0,+\infty)^{n}\times (0,+\infty]\right).\]
\end{definition}

If $p=(x_{0},e_{1},\ldots,x_{n})$ is a discrete path, $\tau_{0},\ldots,\tau_{n-1}$ are positive reals and $\tau_{n}$ is an element of $(0,+\infty]$, then we will denote the continuous path $(p,\tau_{0},\ldots,\tau_{n})$ by
\[\gamma=((x_{0},\tau_{0}),e_{1},(x_{1},\tau_{1}),\ldots,e_{n},(x_{n},\tau_{n})).\]
This path is understood as the trajectory of a particle which starts from $x_{0}$, spends time $\tau_{0}$ at~$x_{0}$, jumps to $x_{1}$ through the edge $e_{1}$, spends time $\tau_{1}$ at $x_{1}$, and so on. The path $p$ is said to be the discrete path underlying the continuous path $\gamma$. Note that a continuous loop spends at least two intervals of time at its starting point, one at the beginning and the other at the end of its course.

The times $\tau_{0},\ldots,\tau_{n}$ are called the \new{holding times} of $\gamma$. The \new{total lifetime} of $\gamma$ is 
\[|\gamma|=\tau_{0}+\ldots+\tau_{n}.\]
The \new{proper lifetime} of $\gamma$ is the time at which $\gamma$ first hits the well:
\begin{equation}\label{eq:deftaugamma}
\tau(\gamma)=\tau_{0}+\ldots+\tau_{m-1}, \mbox{ where } m=\inf\{l\geq 0 : x_{l}\in \W\}.
\end{equation}
It is understood that if the starting point of $\gamma$ belongs to the well, then $\tau(\gamma)=0$, and that if $\gamma$ never hits the well, then $\tau(\gamma)=|\gamma|$.

To the path $\gamma$, we associate a right-continuous function from $[0,|\gamma|)$ to $\V$, which we also denote by $\gamma$. It is defined as follows: for every $t \in [0,|\gamma|)$, there is a unique $k\in \{0,\ldots,n\}$ such that $\tau_{0}+\ldots + \tau_{k-1} \leq t < \tau_{0}+\ldots +\tau_{k}$, and we set $\gamma_{t}=x_{k}$. If $t\in [0,\tau_{0})$, then $k=0$ and $\gamma_{t}=x_{0}$. Informally,
\[\gamma_{t}=\sum_{k=0}^{n} x_{k} \1_{[\tau_{0}+\ldots+\tau_{k-1},\tau_{0}+\ldots+\tau_{k})}(t).\]
For every $t \in (0,|\gamma|)$, we define the restriction of the continuous path $\gamma$ to $[0,t)$ by setting
\[\gamma_{|[0,t)}=((x_{0},\tau_{0}),e_{1},\ldots,e_{k-1},(x_{k-1},\tau_{k-1}),e_{k},(x_{k},t-(\tau_{0}+\ldots+\tau_{k-1}))),\]
where $k$ is the same integer as before. Note that the right-continuous function from $[0,t)$ to $\V$ associated to the continuous path $\gamma_{|[0,t)}$ is equal to the restriction to $[0,t)$ of the right-continuous function from $ [0,|\gamma|)$ to $\V$ associated to $\gamma$.

Finally, if the total lifetime of $\gamma$ is finite, we define the \new{time-reversal}, or \new{inverse} of $\gamma$ as the path
\[\gamma^{-1}=((x_{n},\tau_{n}),e_{n}^{-1},(x_{n-1},\tau_{n-1}),\ldots,e_{1}^{-1},(x_{0},\tau_{0})).\]
The total lifetime of $\gamma^{-1}$ is the same as that of $\gamma$, and the right-continuous function from $[0,|\gamma|)$ to $\V$ associated to $\gamma^{-1}$ is defined by
\[\gamma^{-1}_{t}=\lim_{s\uparrow |\gamma|-t}\gamma_{s}.\]

\subsection{Discrete time random walk}\label{sec:RW} We will now describe the natural random walk on a weighted graph with a well. It is a random continuous path on the graph, of which we will start by describing the underlying random discrete path.

There is nothing complicated about this random discrete path : from any vertex, it selects randomly one of the outgoing edges, using the conductances as probability weights, then goes to the other end of the selected edge, and from there repeats this procedure, until it hits the well. The reason why we give some detail about this random walk is that, because our graph can have several edges between two given vertices, and because we want to keep track of which edge the random walk chose (because the parallel transport of the connection along it will be specific to that edge), we are slightly outside the classical framework of Markov chains on the vertex set of a graph. The random walk that we want to consider can however easily be derived, as we will now explain, from a natural Markov chain on the set of \new{half-edges} of the graph.

Let us consider a graph $\G=(\V,\E)$. We introduce the sets
\begin{equation}\label{fibprods}
\Hs=\{(x,e) \in \V\times \E: x=s(e)\}  \mbox{ and } \Ht=\{(e,y)\in \E\times \V : t(e)=y\}
\end{equation}
of initial and final half-edges, and the set
\[\HE=\Hs\cup \Ht\]
of all half-edges. 

The conductances on the edges of our graph allow us to construct a Markovian transition matrix on the set $\HE$: it is the stochastic matrix $Q\in M_{\HE,\HE}(\R)$ such that 
\begin{itemize}
\item for all edge $e$, with $x=\ul{e}$ and $y=\ol{e}$, we have $Q_{(x,e),(e,y)}=1$,
\item for all vertex $x$ and all edges $e$ and $f$ such that $\ol{e}=x$ and $\ul{f}=x$, $Q_{(e,x),(x,f)}=\frac{\chi_{f}}{\lambda_{x}}.$
\end{itemize}
The only allowed transition from an initial half-edge to a final half-edge is to the final half of the same edge, and the only possible transitions from a final half-edge with target $x$ are to the initial half-edges starting from $x$. 

For every vertex $x$, let us denote by $\eta_{x}$ the following probability measure on the set of initial half-edges issued from $x$:
\[\eta_{x}=\sum_{e\in \E : \ul{e}=x} \frac{\chi_{e}}{\lambda_{x}} \delta_{(x,e)}.\]

For all vertex $x\in \V\setminus \W$, the hitting time of the set of final half-edges of the form $(e,w)$ with $w\in \W$ is almost surely finite for the Markov chain on $\HE$ with transition matrix $Q$ and initial distribution $\eta_{x}$. A path of this Markov chain stopped at this hitting time is thus of the form
\[((x_{0},e_{1}),(e_{1},x_{1}),\ldots,(x_{n-1},e_{n}),(e_{n},x_{n}))\]
with $x_{0}=x$, $x_{0},\ldots,x_{n-1}\in \V\setminus \W$, $x_{n}\in \W$.
The sequence $(x_{0},e_{1},x_{1},\ldots,e_{n},x_{n})$ is then a random element of $\DP(\G)$, of which we denote $\Q_{x}$ the distribution.

For $x\in \W$, we define $\Q_{x}$ as the Dirac mass at the constant path at $x$, of length $0$.

The measure $\Q_{x}$ can be described and characterised as follows.

\begin{proposition} Let $\G$ be a weighted graph with a well. For every vertex $x\in \V$, $\Q_{x}$  is the unique probability measure on the countable set $\DP(\G)$ such that for every discrete path $(x_{0},e_{1},\ldots,e_{n},x_{n})$, one has
\begin{equation}\label{defQx}
\Q_{x}(\{(x_{0},e_{1},\ldots,e_{n},x_{n})\})=\1_{\{x\}}(x_{0})\1_{\W}(x_{n}) \prod_{k=1}^{n} \left[\1_{\V\setminus\W}(x_{k-1})\frac{\chi_{e_{k}}}{\lambda_{x_{k-1}}}\right].
\end{equation}
\end{proposition}
Note that the probability $\Q_{x}$ is supported by the set of paths joining $x$ to the well. In particular, and although we do not make this explicit in the notation, the measure $\Q_{x}$ depends on the well $\W$.

\begin{proof} The expression \eqref{defQx} of $\Q_{x}$ follows immediately from its description given before the statement of the proposition. The uniqueness is obvious, since \eqref{defQx} describes the mass assigned by $\Q_{x}$ to every singleton of the countable set $\DP(\G)$.
\end{proof}

It will be convenient in the following to introduce the notation 
\begin{equation}\label{eq:defPxe}
P_{x,e}=\frac{\chi_{e}}{\lambda_{x}}
\end{equation}
defined for all $(x,e)\in \Hs$, that is, for all vertex $x$ and all edge $e$ starting from $x$. 

\subsection{Continuous time random walk}
The definition \ref{def CP} of the set $\CP(\G)$ of continuous paths on~$\G$ as
a countable union of Cartesian products of a finite set and finitely many intervals suggests the definition of a $\sigma$-field on $\CP(\G)$ which we adopt but do not deem necessary to write down explicitly.

\begin{definition}  Let $\G$ be a weighted graph with a well. For every $x\in \V$, we denote by $\P_{x}$ the unique probability measure on $\CP(\G)$ such that, for every bounded measurable function $F$ on~$\CP(\G)$,
\begin{align}
\nonumber
\int_{\CP(\G)} F(\gamma) \; \d\P_{x}(\gamma)&=\sum_{n\geq 0} \sum_{\substack{p\in \DP_{n}(\G)\\ p=(x_{0},e_{1},\ldots,x_{n})}} \Q_{x}(\{p\}) \\
&\hspace{-1.5cm} \int_{(0,+\infty)^{n}} F\big(((x_{0},\tau_{0}),e_{1},\ldots,(x_{n-1},\tau_{n-1}),e_{n},(x_{n},+\infty))\big) e^{-\tau_{0}-\ldots-\tau_{n-1}}\; \d\tau_{0}\ldots \d\tau_{n-1}. \label{defPx}
\end{align}
\end{definition}

In words, a sample of $\P_{x}$ can be obtained by first sampling $\Q_{x}$, in order to obtain a discrete path $(x_{0},e_{1},\ldots,e_{n},x_{n})$, and then sampling $n$ independent exponential random variables $\tau_{0},\ldots,\tau_{n-1}$ of parameter $1$. The continuous path 
\[((x_{0},\tau_{0}),e_{1},\ldots,e_{n-1},(x_{n-1},\tau_{n-1}),e_{n},(x_{n},+\infty))\]
has the distribution $\P_{x}$.

Note that $\P_{x}$ is supported by continuous paths with infinite lifetime. Moreover, for all $t>0$, the inverse of $\gamma_{|[0,t)}$ is well defined (see the end of Section \ref{sec:CP}).

The following lemma expresses the reversibility, or more precisely the $\lambda$-reversibility of the continuous time random walk on $\G$.

\begin{lemma}\label{reversibility} Let $\G$ be a weighted graph with a well. Let $x,y\in \V\setminus\W$ be proper vertices of $\G$. Let $t>0$ be a positive real. Let $F$ be a bounded measurable function on $\CP(\G)$. Then
\[\lambda_{x}\int_{\CP(\G)} F(\gamma_{[0,t)}^{-1}) \1_{\{\gamma_{t}=y\}} \; \d\P_{x}(\gamma)=\lambda_{y}\int_{\CP(\G)} F(\gamma_{[0,t)}) \1_{\{\gamma_{t}=x\}} \; \d\P_{y}(\gamma).\]
\end{lemma}

Note that on the $\P_{x}$-negligible event where the path $\gamma$ jumps at time $t$, the time-reversed path $\gamma_{|[0,t)}^{-1}$ starts from the vertex visited by $\gamma$ immediately before time $t$, which may not be $y$.

\begin{proof}
The left-hand side is equal to
\begin{align*}
&e^{-t}\sum_{n=0}^{\infty} \sum_{\substack{p\in \DP_{n}\\p=(x_{0},e_{1},\ldots,x_{n})\\x_{0}=x,x_{n}=y}} \lambda_{x}P_{x_{0},e_{1}}\ldots P_{x_{n-1},e_{n}} \\
&\hspace{3cm} \int_{0<t_{1}<\ldots<t_{n}<t}F(((x_{n},t-t_{n}),e_{n}^{-1},\ldots,(x_{1},t_{2}-t_{1}),e_{1}^{-1},(x_{0},t_{1})))\; \d t_{1}\ldots \d t_{n}.
\end{align*} 
Using the $\lambda$-reversibility of the discrete random walk, that is, the relation 
\[\lambda_{x_{0}} \frac{\chi_{e_{1}}}{\lambda_{x_{0}}}\ldots \frac{\chi_{e_{n}}}{\lambda_{x_{n-1}}}=\lambda_{x_{n}} \frac{\chi_{e_{n}^{-1}}}{\lambda_{x_{n}}}\ldots \frac{\chi_{e_{1}^{-1}}}{\lambda_{x_{1}}},\]
and performing the change of variables $s_{i}=t-t_{i}$ for $i\in \{1,\ldots,n\}$, we find an expression of the right-hand side.
\end{proof}

\subsection{Measures on paths and loops} In this section, we will use the probability measures $\P_{x}$ to construct two families of measures on $\CP(\G)$ which will play a central role in this work. We assume as always that $\G$ is a weighted graph with a well.

The first family of measures is defined as follows.

\begin{definition}\label{defnu} Let $x,y\in \V\setminus \W$ be proper vertices. The measure $\nu_{x,y}$ is the measure on $\CP(\G)$ such that for all non-negative measurable function $F$ on $\CP(\G)$, we have
\[\int_{\CP(\G)}F(\gamma) \d\nu_{x,y}(\gamma)=\int_{\CP(\G)\times (0,+\infty)}F(\gamma_{|[0,t)})\1_{\{\gamma_{t}=y\}} \frac{1}{\lambda_{y}}\; \d\P_{x}(\gamma)  \d t.\]
The measure $\nu_{x}$ is defined by
\[\nu_{x}=\sum_{y\in \V\setminus \W} \nu_{x,y}\lambda_{y},\]
so that for all bounded measurable function $F$ on $\CP(\G)$, we have
\begin{equation}\label{eq:defnux}
\int_{\CP(\G)}F(\gamma) \d\nu_{x}(\gamma)=\int_{\CP(\G)}\int_{0}^{\tau(\gamma)} F(\gamma_{|[0,t)}) \; \d t \d\P_{x}(\gamma).
\end{equation}
Finally, the measure $\nu$ is defined by
\[\nu=\sum_{x\in \V\setminus \W} \nu_{x}.\]
\end{definition}

In Section \ref{sec:covFK}, we shall compute various integrals with respect to these measures, and we will see in particular that they are finite. 

Let us comment and partially rephrase these definitions. For this, let us focus first on the description of the measure $\nu_{x}$ given by \eqref{eq:defnux}. It is an average under $\P_{x}$ of the following operation: 
\begin{itemize}
\item pick a path $\gamma$ starting from $x$ and let $\tau(\gamma)$ be the ($\P_{x}$-almost surely finite) time at which this path hits the well, then
\item consider the image on $\CP(\G)$ of the Lebesgue measure on the interval $[0,\tau(\gamma))$ by the map $t\mapsto \gamma_{|[0,t)}$. 
\end{itemize} 
In particular, the mass of $\nu_{x}$ is the expected hitting time of the well starting from $x$. 
Moreover, \eqref{eq:defnux} can be written at the level of measures as
\begin{equation}\label{eq:defnux2}
\nu_{x}=\Ex_{x} \int_{0}^{\tau(\gamma)} \delta_{\gamma_{|[0,t)}}\; \d t.
\end{equation}
Then, for all $y$, the measure $\nu_{x,y}$ is $\lambda_{y}^{-1}$ times the part of $\nu_{x}$ supported by paths ending at $y$.

The second family of measures is the following, and was introduced in \cite{LeJan-book}.

\begin{definition}\label{defmu} Let $x,y\in \V\setminus\W$ be proper vertices. The measure $\mu_{x,y}$ is the measure on $\CP(\G)$ such that for all non-negative measurable function $F$ on $\CP(\G)$, we have
\[\int_{\CP(\G)}F(\gamma) \d\mu_{x,y}(\gamma)=\int_{\CP(\G)\times (0,+\infty)}F(\gamma_{|[0,t)})\1_{\{\gamma_{t}=y\}} \frac{1}{\lambda_{y}}\; \d\P_{x}(\gamma)  \frac{\d t}{t}.\]
The measure $\mu_{x}$ is defined by
\[\mu_{x}=\sum_{y\in \V\setminus \W} \mu_{x,y}\lambda_{y},\]
so that for all bounded measurable function $F$ on $\CP(\G)$, we have
\[\int_{\CP(\G)}F(\gamma) \d\mu_{x}(\gamma)=\int_{\CP(\G)}\int_{0}^{\tau(\gamma)} F(\gamma_{|[0,t)}) \; \frac{\d t}{t} \d\P_{x}(\gamma).\]
Finally, the measure $\mu$ is defined by
\[\mu=\sum_{x\in \V\setminus \W} \mu_{x}.\]
\end{definition}

In analogy with \eqref{eq:defnux2}, we can write, for all proper vertex $x$,
\begin{equation}\label{eq:defmux2}
\mu_{x}=\Ex_{x}\int_{0}^{\tau(\gamma)} \delta_{\gamma_{|[0,t)}} \; \frac{\d t}{t}.
\end{equation}
The only difference between the definitions of the two families of measures that we consider is thus the presence of the factor $\frac{1}{t}$ in the second family.

The measures $\mu_{x,y}$ are finite when $x\neq y$, but the measures $\mu_{x,x}$ are infinite, because of the contribution of short loops, which in the present discrete setting are constant loops. Let us split the measure $\mu$ in a way that allows us to isolate this divergence.

\begin{definition} The measures $\mupl$ and $\tilde \mu$ are defined as follows:
\[\mupl=\sum_{x\in \V\setminus \W} \mu_{x,x}\lambda_{x} \ \mbox{ and } \ \tilde\mu=\sum_{x,y\in \V\setminus \W, x\neq y}  \mu_{x,y} \lambda_{y},\]
so that $\mupl$ is the restriction of $\mu$ to the set $\CL(\G)$ of continuous loops. We now split the measure~$\mupl$ as
\[\mupl=\mup+\mul,\]
where $\mup$ is the restriction of $\mupl$ (and hence of $\mu$) to the set of constant loops, and $\mul$ its restriction to the set of non-constant loops. 
\end{definition}

Note that the measure $\mup$ has a very simple expression: for every measurable function $F$ on~$\CP(\G)$, we have
\[\int_{\CP(\G)} F(\gamma)\d\mup(\gamma)=\sum_{x\in \V\setminus \W} \int_{0}^{+\infty} F((x,t)) e^{-t} \frac{\d t}{t}.\]
The measure $\mup$ is in particular infinite, and we shall see that $\mu-\mup=\mul+\tilde\mu$ is finite.

The following straightforward consequence of Lemma \ref{reversibility} will be useful.

\begin{lemma}\label{reversal} With the notation of Definitions \ref{defnu} and \ref{defmu}, the respective images by the map $\gamma\mapsto \gamma^{-1}$ of the measures $\nu_{x,y}$ and $\mu_{x,y}$ are the measures $\nu_{y,x}$ and $\mu_{y,x}$. The measure $\mupl$ is invariant under the map $\gamma\mapsto \gamma^{-1}$.
\end{lemma}

\section{Vector bundles}\label{sec:bundles}

In this section, we introduce the notion which plays the main role in our study, namely that of vector bundle over a graph.

\subsection{Bundles and bundle-valued forms} 

Let $\G=(\V,\E)$ be a graph. Let us choose a base field $\K$ to be either $\R$ or $\C$. Let $r\geq 1$ be an integer. 

\begin{definition}
A \new{$\K$-vector bundle of rank $r$ over $\G$} is a collection $\F=((\F_{x})_{x\in \V}, (\F_{e})_{e\in \E})$ of vector spaces over $\K$ which all have the same dimension $r$, and such that for all edge $e\in \E$, $\F_{e}=\F_{e^{-1}}$.
\end{definition}

We say that the bundle $\F$ is \new{Euclidean} (if $\K=\R$), or \new{Hermitian} (if $\K=\C$), if each of the vector spaces of which it consists are Euclidean,  or Hermitian. The bilinear, or sesquilinear form on the vector space $\F_{x}$ (resp. $\F_{e}$) is denoted by $\langle \cdot,\cdot \rangle_{x}$ (resp. $\langle \cdot,\cdot \rangle_{e}$). When $\K=\C$, we take this form to be antilinear in the first variable and linear in the second, according to the physicists' convention. We also denote by $\|\cdot\|_x$ and $\|\cdot\|_e$ the associated norms.

Most of our results will have the same form in the real and complex cases, up to the adjustment of a few constants. In order to uniformise the results, it turns out to be useful to introduce the parameter
\begin{equation}\label{beta0}
\beta=\dim_{\R}\K.
\end{equation}
In fact, this parameter will not be used until Section \ref{sec:prob measures sections}, and will be redefined there.\footnote{Although we will not treat this case in detail, our results hold for quaternionic bundles endowed with quaternionic unitary (symplectic) connections and quaternion Hermitian potentials. In other words, $\K=\bH$ is also allowed in our results, and in this case, $\beta=4$. We do not claim that the proofs of our results adapt directly to the quaternionic case, because they involve linear algebraic notions such as determinants which must be handled with care. However, a quaternionic bundle of rank $r$ endowed with a symplectic connection can be seen as a real bundle of rank $4r$ with an orthogonal connection, if a special one. This is similar to the fact that a complex bundle of rank $r$ with a unitary connection can be seen as a real bundle of rank $2r$ with an orthogonal connection which commutes to multiplication by $i$. Moreover, the real part trace of a (left) quaternionic linear transformation of a (right) quaternionic vector space is equal to 4 times the trace of the same transformation seen as real linear, of a real vector space. 

For the Eisenbaum, Le Jan and Sznitman isomorphisms, this remark allows for a straightforward application to the quaternionic case. For Dynkin's isomorphism, it is the form given by \eqref{eq:DynkinSandwich} which extends readily. 

The only part of the extension to the quaternionic case that we do not make explicit is the way in which Wick's formula extends. For more information on this, we refer the reader to \cite[Thm 3.1]{Wick-q}. This quaternionic formula was also used by Lupu in \cite{Lupu-Dynkin}.
}

\begin{definition}\label{def:Omegas} Let $\F$ be a vector bundle over a graph $\G=(\V,\E)$. We define the space of \new{sections} of $\F$, or \new{$0$-forms} with values in $\F$, as the vector space
\[\Omega^{0}(\F)=\bigoplus_{x\in \V}\F_{x}.\]
For all subset $\S$ of $\V$, we denote by $\Omega^{0}(\S,\F)$ the subspace of $\Omega^{0}(\F)$ consisting of all sections that vanish identically outside $\S$. Moreover, we denote by $\pi_{\S}:\Omega^{0}(\F)\to \Omega^{0}(\S,\F)$ the projection operator such that for every section $f$ of $\F$ over $\V$ and every $x\in \V$,
\begin{equation}\label{defpiS}
(\pi_{\S}f)(x)=\1_{\S}(x)f(x).
\end{equation}

We define the space of \new{$1$-forms} on $\G$ with values in $\F$ as the space
\[\Omega^{1}(\F)=\Big\{\omega \in \bigoplus_{e\in \E} \F_{e} : \forall e \in \E, \omega(e^{-1})=-\omega(e) \Big\}.\]
\end{definition}

It would of course be possible to define $1$-forms on subsets of $\E$, but we will not use this construction in this work.

A fundamental example of a vector bundle over a graph $\G$ is the trivial vector bundle of rank~$r$, denoted by $\K^{r}_{\G}$, which is such that for all $a\in \V\cup \E$, the fibre $(\K^{r}_{\G})_{a}$ over $a$ is $\K^{r}$. When $r=1$, the sections of this bundle are exactly the functions over $\V$, and $1$-forms with values in this bundle are the usual scalar-valued $1$-forms on the graph $\G$, for example in the sense of \cite{LeJan-book}.

\subsection{Hermitian structures}\label{sec:Hermitian} Let us now assume that $\G$ is a weighted graph with a well. What matters here is not so much the well as the availability of the measures $\chi$ on $\E$ and $\lambda$ on $\V$. In what follows, we use the word {\em Hermitian} in place of {\em Euclidean or Hermitian}.

For all $f_{1},f_{2}\in \Omega^{0}(\F)$, we set
\[(f_{1},f_{2})_{\Omega^{0}}=\sum_{x\in \V} \lambda_{x} \langle f_{1}(x),f_{2}(x)\rangle_{x}.\]
This endows $\Omega^{0}(\F)$ with the structure of a Hermitian space.

In the case where $\F$ is the rank $1$ trivial bundle $\K_{\G}$, of which sections are simply functions on the graph, we use the notation $\left(\cdot,\cdot\right)$ without any subscript for the Hermitian structure on~$\Omega^{0}(\K_{\G})$.

There is also a Hermitian scalar product on the space of $1$-forms with values in $\F$. 
For all $\omega_{1},\omega_{2}\in \Omega^{1}(\F)$, we define
\begin{equation}\label{defscal1}
(\omega_{1},\omega_{2})_{\Omega^{1}}=\frac{1}{2}\sum_{e\in \E} \chi_{e} \langle \omega_{1}(e),\omega_{2}(e)\rangle_{e}.
\end{equation}

\subsection{Connections}\label{sec:connections} Recall from \eqref{fibprods} the definition of the sets $\Hs$ and $\Ht$ of initial and final half-edges of the graph.

\begin{definition}
A \new{connection} on the Hermitian (resp. Euclidean) vector bundle $\F$ is a collection~$h$ of unitary (resp. orthogonal) isomorphisms between some of the vector spaces which constitute~$\F$, namely the data, 
\begin{itemize}
\item for each $(x,e)\in \Hs$, of a unitary isomorphism $h_{e,x}:\F_{x}\to \F_{e}$, and
\item for each $(e,y)\in \Ht$, of a unitary isomorphism $h_{y,e}:\F_{e}\to \F_{y}$,
\end{itemize}
such that for all $(x,e)\in \Hs$, the relation $h_{x,e^{-1}}=h_{e,x}^{-1}$ holds.
\end{definition}

This definition is illustrated in Figure \ref{conn} below. We shall denote the set of connections on $\F$ by $\A(\F)$.

\begin{figure}[h!]
\begin{center}
\includegraphics{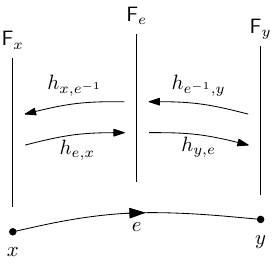}
\caption{\label{conn} \small The isomorphisms $h_{e,x}$ and $h_{y,e}$ correspond to parallel transport along the oriented edge $e$.}
\end{center}
\end{figure}

Our notation for the isomorphisms which constitute a connection may seem complicated. To remember it, one should notice that, whether $a$ and $b$ be vertices or edges, $h_{a,b}$ sends the fibre over $b$ to the fibre over $a$. Moreover, although this is redundant, the edge is written with the orientation which is consistent with the parallel transport expressed by $h_{a,b}$.\footnote{Although the word \emph{connection} is common in this context, and was used for instance by Kenyon in \cite{Kenyon},~$h$ is really, in the language of differential geometry,  a discrete analogue of the parallel transport induced by a connection.}

As a useful example, let us mention that $\K^{r}_{\G}$, the trivial bundle of rank $r$ carries a  connection that we call the canonical connection, for which all the isomorphisms $h_{a,b}$ are the identity of $\K^{r}$.

\subsection{Gauge transformations} Let us define the group of automorphisms of our vector bundle, called the \new{gauge group}. For every Hermitian (resp. Euclidean) vector space $F$, let us denote by $\Un(F)$ the group of unitary (resp. orthogonal) linear transformations of $F$. 

\begin{definition}
The \new{gauge group} of the vector bundle $\F$ over $\G=(\V,\E)$ is the group 
\[\Aut(\F)=\bigg\{j \in \prod_{x\in \V} \Un(\F_{x}) \times \prod_{e\in \E} \Un(\F_{e}) :  j_{e}=j_{e^{-1}}  \mbox{ \rm for all } e\in \E\bigg\}.\]
The elements of the gauge group are called \new{gauge transformations}.
\end{definition}

The gauge group will act on all the objects that we define and we will systematically say how. For instance, $\Aut(\F)$ acts on the set $\A(\F)$ of connections on $\F$, as follows. Let $h$ be a connection on $\F$. Let $j$ be a gauge transformation. The connection $j\cdot h$ is defined by
\begin{align*}
&(j\cdot h)_{x,e}=j_{x}\circ h_{x,e}\circ j_{e}^{-1}\ \, \mbox{ for } (x,e) \in \Hs \mbox{ and } \\
&(j\cdot h)_{e,y}=j_{e}\circ h_{e,y}\circ j_{y}^{-1}\ \  \mbox{ for } (e,y)\in \Ht.
\end{align*}

\subsection{Holonomy}

Let $h$ be a connection on the vector bundle $\F$ over the graph $\G$. For each edge $e\in \E$, we define the holonomy along $e$ as the isomorphism
\[h_{e}=h_{\ol{e},e} \circ h_{e,\ul{e}}:\F_{\ul{e}}\to \F_{\ol{e}}.\]
The holonomy of the connection $h$ along a discrete path $p=(x_{0},e_{1},\ldots,e_{n},x_{n})$ is defined as
\[\hol_{h}(p)=h_{e_{n}}\circ \cdots \circ h_{e_{1}}:\F_{x_{0}} \to \F_{x_{n}}.\]
This definition extends to the case of a continuous path $\gamma$ with underlying discrete path $p$ (see Section \ref{sec:CP}) by setting
\[\hol_{h}(\gamma)=\hol_{h}(p).\]

Let us write the modification of the holonomy along a continuous path induced by the action of a gauge transformation on the connection. Suppose $h$ is a connection, $j$ a gauge transformation, and $\gamma$ a path, discrete or continuous, in our graph. Then
\[\hol_{j\cdot h}(\gamma)=j_{\ol{\gamma}} \circ \hol_{h}(\gamma)\circ  j_{\ul{\gamma}}^{-1}.\]

\subsection{Twisted holonomy}\label{sec:twistedhol} The definition of the holonomy along a continuous path which we just gave is natural, but it depends only on the discrete path which underlies it.
We propose a definition of the holonomy along a continuous-time trajectory in a graph which is genuinely dependent on its continuous time structure.  

In order to give this definition, we need to introduce a new ingredient. From the bundle $\F$, we can form a new bundle $\End(\F)$ such that for all $a\in \V\cup \E$, the fibre of $\End(\F)$ over $a$ is
\[\End(\F)_{a}=\End(\F_{a}),\]
the vector space of $\K$-linear endomorphisms of $\F_{a}$. We give a name to some sections of this new bundle.

\begin{definition}\label{def:potentiel} Let $\G$ be a weighted graph with a well. Let $\F$ be a Euclidean (resp. Hermitian) vector bundle over $\G$. A \new{potential} on $\F$ is an element $H\in \Omega^{0}(\V\setminus\W,\End(\F))$, that is, a section of the vector bundle $\End(\F)$ vanishing over the well and such that for every vertex $x\in \V$, the operator $H_{x}\in \End(\F_{x})$ is symmetric (resp. Hermitian). 
\end{definition}

We can now give an enhanced definition of the holonomy along a continuous path.

\begin{definition}\label{def:twisted-holo} Let $h$ be a connection and $H$ a potential on a Hermitian vector bundle $\F$ over a graph $\G$. Let $\gamma=((x_{0},\tau_{0}),e_{1},\ldots,e_{n},(x_{n},\tau_{n}))$ be a continuous path in $\G$. Assume that $\tau_{n}<\infty$ or $H_{x_{n}}=0$. The {\em holonomy of $h$ twisted by $H$ along $\gamma$} is the linear map
\[\hol_{h,H}(\gamma)=e^{-\tau_{n}H_{x_{n}}} \circ h_{e_{n}}\circ \cdots \circ e^{-\tau_{1}H_{x_{1}}} \circ h_{e_{1}} \circ e^{-\tau_{0}H_{x_{0}}}:\F_{x_{0}} \to \F_{x_{n}}.\]
\end{definition}

Let us emphasize that the twisted holonomy does not behave well with respect to the time-reversal of paths: in general, even if the path $\gamma^{-1}$ is defined (see the end of Section \ref{sec:CP}), it may be that
\[\hol_{h,H}(\gamma^{-1})\neq \hol_{h,H}(\gamma)^{-1}.\]
For example, if $\gamma=((x,\tau))$ is a constant path at a vertex $x$ such that $H_{x}\neq 0$, then $\gamma^{-1}=\gamma$ but $\hol_{h,H}(\gamma^{-1})=e^{-\tau H_{x}}\neq e^{\tau H_{x}}=\hol_{h,H}(\gamma)^{-1}$. Note however that the equality
\begin{equation}\label{inverse-adjoint}
\hol_{h,H}(\gamma^{-1})= \hol_{h,H}(\gamma)^{*}
\end{equation}
holds, and we will indeed use this equality a lot.

Let us explain how twisted holonomies are affected by gauge transformations. For this, we need to explain how the gauge group acts on the space of potentials. Let $H$ be a potential and~$j$ a gauge transformation. We define $j\cdot H$ as the potential such that for all vertex $x$,
\[(j\cdot H)_{x}=j_{x}\circ H_{x} \circ j_{x}^{-1}.\]
Then, for all continuous path $\gamma$, we have the relation
\[\hol_{j\cdot h,j\cdot H}(\gamma)=j_{\ol{\gamma}} \circ \hol_{h,H}(\gamma)\circ  j_{\ul{\gamma}}^{-1}.\]

We would expect the definition of the twisted holonomy to seem rather strange to some readers, especially after the comment that we just made, and we hope that  the results that will be proved in this paper will convince them of its interest. In the meantime, let us devote the next paragraph to a discussion of two points of view from which, we hope, this definition can be found natural.

\subsection{Twisted holonomy and hidden loops}

The first point of view is quite informal and inspired by quantum mechanics. We already alluded to the fact that a path in a graph can be seen as a discrete model for the time evolution of a particle in space. In this picture, the connection represents the way in which an ambient gauge field acts on the state of the particle as it moves around. If we now understand the operator $H_{x}$, up to a factor $i$ or $i/\hbar$, as the Hamiltonian of the particle at the point $x$, then the term $e^{-\tau_{x}H_{x}}$ represents the evolution of the state of the particle as it spends a stretch of time $\tau_{x}$ at the vertex $x$, and the twisted holonomy along the path followed by the particle simply represents the subsequent modification of its state. 

From a second and more mathematical point of view, the twisted holonomy as we defined it can be seen, at least in the case where the operators $H_{x}$ are non-negative, as an averaged version of the ordinary holonomy in a larger graph, which has a small loop attached to each vertex. 

To explain this, let us consider a weighted graph with a well $\G$. Let us associate to $\G$ a new graph $\Gc=(\V,\Ec)$ which has the same vertices as $\G$, and an extended set of edges
\[\Ec=\E\cup\bigcup_{x\in \V} \{{ l}_{x},{ l}_{x}^{-1}\},\]
where the edge ${ l}_{x}$ and its inverse have source and target $x$. We call these new edges \new{looping edges}.

\begin{figure}[h!]
\begin{center}
\includegraphics[height=3.7cm]{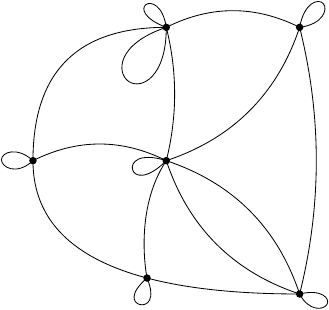}
\caption{\small The graph $\Gc$ obtained from the graph depicted in Figure \ref{exgraph}.}
\end{center}
\end{figure}

Let us assume that for every $x\in \V$, the operator $H_{x}$ is Hermitian non-negative on $\F_{x}$. For each $x\in \V$, let us associate to the edge ${ l}_{x}$ a positive jump rate $r_{x}$ and a holonomy $h_{{ l}_{x}}$, in such a way that
\[H_{x}=r_{x} \left(2\Id_{\F_{x}}-(h_{{ l}_{x}}+h_{{ l}_{x}}^{-1})\right).\]
It is indeed the case that every Hermitian non-negative operator can be written, although not uniquely, as $r(2\Id-(U+U^{-1}))$ with $r>0$ and $U$ unitary. One can for example take any $r$ larger than one quarter of the largest eigenvalue of $H$, and $U=\exp i \arccos (1-\frac{H}{2r})$.

Let us now consider the continuous time random walk in the graph $\Gc$ which, when it stands at a vertex $x$, jumps from $x$ through the edge $e\in \E$ at rate ${\chi_{e}}/{\lambda_{x}}$, and through each of the looping edges $l_{x}$ and $l_{x}^{-1}$ at rate $r_{x}$. It will be convenient to use probabilistic language, and for this, we will denote this random continuous path on $\Gc$ by $\Gamma^{\circ}$. Let us emphasise that $\Gamma^{\circ}$ is not exactly a random walk on the graph $\Gc$ in the sense of Section \ref{sec:RW}, because it jumps at rate $1+2r_{x}$ from the vertex $x$, whereas the random walk on a weighted graph normally jumps at rate~$1$. Incidentally, it is possible, if one so wishes, to take all $r_{x}$ equal to some large enough real number, with the effect that the jump rate of $\Gamma^{\circ}$ becomes constant over the set of vertices.

The reason for our change of jumping rate in the definition of $\Gamma^{\circ}$ is the following: from $\Gamma^{\circ}$, we can recover the usual random walk, which we will denote by $\Gamma$, in the original graph $\G$, by simply ignoring the jumps across the loop-edges of $\Gc$. More formally, there is a shearing map $S:\CP(\Gc)\to \CP(\G)$ which forgets the jumps across the looping edges of $\Ec\setminus \E$, and we are claiming that $S(\Gamma^{\circ})$ and $\Gamma$ have the same distribution. We reckon that a formal definition of $\Gamma^{\circ}$ and the map $S$, followed by a formal proof of this claim, would take a lot of space and time, and bring little additional light on the present discussion. 

Let us now compare the ordinary holonomy $\hol^{\circ}_{h}$ of the connection $h$ in the graph $\Gc$ with the holonomy $\hol_{h,H}$ of $h$  twisted by $H$ in $\G$. 

\begin{proposition}\label{prop:shearing} 
For all $t\geq 0$ and all $x,y\in \V$, the following equality almost surely holds in $\mathrm{Hom}(\F_x,\F_{y})$:
\[\Ex_{x}[\hol^{\circ}_{h}(\Gamma^{\circ}_{|[0,t)})|S(\Gamma^{\circ}_{|[0,t)})]\1_{\{\Gamma^{\circ}_{t}=y\}}=\hol_{h,H}(S(\Gamma^{\circ}_{|[0,t)}))\1_{\{\Gamma^{\circ}_{t}=y\}}.\]
\end{proposition}

The way in which we want to read this equality is the following: for every path $\gamma$ in $\G$ with lifetime $t$, the holonomy of $h$ twisted by $H$ along $\gamma$ is given by 
\[\hol_{h,H}(\gamma)=\Ex[\hol^{\circ}_{h}(\Gamma^{\circ}_{t})|S(\Gamma^{\circ}_{|[0,t)})=\gamma],\]
the average over a set of non-observed paths in the larger graph $\Gc$ of an ordinary unitary holonomy, and of which the observed path $\gamma$ is an approximation. 

\begin{proof} Conditional on $S(\Gamma^{\circ}_{|[0,t)})=((x_{0},\tau_{0}),e_{1},\ldots,(x_{n},\tau_{n}))$, the path $\Gamma^{\circ}_{|[0,t)}$ is equal to $S(\Gamma^{\circ}_{|[0,t)})$ to which have been grafted, at each of its stays at one of the vertices $x_{1},\ldots,x_{n}$, an independent Poissonian set of excursions through the loop-edges at this vertex. 

For each $k\in \{0,\ldots,n\}$, let $N_{k}$ be an independent Poisson variable with parameter $2r_{x_{k}}\tau_{k}$ and let $\beta_{x_{k},\tau_{k}}$ be a random discrete loop of length $N_{k}$ based at $x_{k}$, which traverses $N_{k}$ times one of the edges $l_{x}$ and $l_{x}^{-1}$ (and no other), independently with equal probability at each jump. Then
\begin{align*}
\Ex_{x}[\hol^{\circ}_{h}(X^{\circ}_{|[0,t)})|S(X^{\circ}_{|[0,t)})=((x_{0},\tau_{0}),e_{1},\ldots,(x_{n},\tau_{n}))]&=\\
&\hspace{-4.2cm} \Ex[\hol^{\circ}_{h}(\beta_{x_{n},\tau_{n}})]\circ h_{e_{n}}\circ \Ex[\hol^{\circ}_{h}(\beta_{x_{n-1},\tau_{n-1}})] \circ \ldots \circ h_{e_{1}} \circ \Ex[\hol^{\circ}_{h}(\beta_{x_{0},\tau_{0}})].
\end{align*}
On the other hand, for each $k\in \{0,\ldots,n\}$,
\begin{align*}
\Ex[\hol^{\circ}_{h}(\beta_{x_{k},\tau_{k}})]&=e^{-2r_{x_{k}}\tau_{k}}\sum_{n=0}^{\infty} \frac{(2r_{x_{k}}\tau_{k})^{n}}{n!} \frac{1}{2^{n}}\sum_{m=0}^{n} \binom{n}{m} h_{l_{x_{k}}}^{m} h_{l_{x_{k}}^{-1}}^{n-m}\\
&=e^{-\tau_{k}r_{x_{k}}(2\Id_{\F_{x_{k}}}-(h_{l_{x_{k}}}+h_{l_{x_{k}}}^{-1}))}\\
&=e^{-\tau_{k}H_{x_{k}}},
\end{align*}
and the expected result follows.
\end{proof}

\subsection{Occupation measures and local times}

A particular case of interest of the twisted holonomy is that where for every vertex $x$, the operator $H_{x}$ is scalar. In this case, the twisted holonomy of a path can be related to its occupation measure, or to its local time, which we now define. 

\begin{definition}\label{local time} Let $\gamma=((x_{0},\tau_{0}),e_{1},\ldots,e_{n},(x_{n},\tau_{n}))$ be a continuous path in the graph $\G$. The \new{occupation measure} of $\gamma$ is the measure $\vartheta(\gamma)$ on $\V$ such that for all vertex $x$,
\[\vartheta_{x}(\gamma)=\sum_{k=0}^{n}\delta_{x,x_{k}}\tau_{k}.\]
The \new{local time} of $\gamma$ is the density of its occupation measure with respect to the reference measure~$\lambda$: for every vertex $x$,
\[\ell_{x}(\gamma)=\frac{\vartheta_{x}(\gamma)}{\lambda_{x}}=\frac{1}{\lambda_{x}}\sum_{k=0}^{n}\delta_{x,x_{k}}\tau_{k}.\]
\end{definition}

From Section \ref{sec:LJS} on, we will consider Poissonian ensembles of paths. We will then use the following extended definition of the occupation measures and local times: if $\mathcal P$ is a set of paths, we define, for all vertex $x$,
\[\ell_{x}(\mathcal P)=\sum_{\gamma\in \mathcal P}\ell_{x}(\gamma) \mbox{ and } \vartheta_{x}(\mathcal P)=\sum_{\gamma\in \mathcal P}\vartheta_{x}(\gamma).\]

The following lemma is a straightforward consequence of the definition of the twisted holonomy. In its statement, we use the same notation for a scalar operator and its unique eigenvalue.

\begin{lemma}\label{hol-tloc} Assume that for all vertex $x$, the operator $H_{x}$ is scalar on $\F_{x}$, that is, equal to $H_{x}\Id_{\F_{x}}$ for some $H_{x}\in \K$. Let $\gamma$ be a continuous path. Assume that $\tau(\gamma)<+\infty$  or $H_{\ol{\gamma}}=0$. Then one has the equality
\begin{equation}\label{sortH}
\hol_{h,H}(\gamma)=e^{-\sum_{x\in \V} H_{x}\vartheta_{x}(\gamma)}\hol_{h}(\gamma).
\end{equation}
\end{lemma}

This lemma shows how the twisted holonomy, which, in the classical language of the theory of Markov processes, could be called a multiplicative functional, extends the classical definition of exponential functionals of local times. 

\subsection{Differentials} Let us consider a graph $\G=(\V,\E)$, a vector bundle $\F$ over $\G$ and a connection $h$ on $\F$. 

\begin{definition}
The \new{differential} is the linear operator 
\[d:\Omega^{0}(\F) \to \Omega^{1}(\F)\subset \bigoplus_{e\in \E}\F_{e}\]
defined by setting, for every $f\in \Omega^{0}(\F)$ and every $e\in \E$,
\[(df)(e)=h_{e^{-1},\ol{e}} f(\ol{e})-h_{e,\ul{e}} f(\ul{e}).\]
\end{definition}
The range of $d$ is contained in $\Omega^{1}(\F)$: indeed, if the inverse of the edge $e$ is defined, then replacing $e$ by $e^{-1}$ in this definition exchanges $\ul{e}$ and $\ol{e}$, so that $(df)(e^{-1})=-(df)(e)$ as expected.

The definition of $d$ depends on the connection $h$, but we prefer not to use the notation $d^{h}$, which we find too heavy.

\begin{figure}[h!]
\begin{center}
\includegraphics[width=10cm]{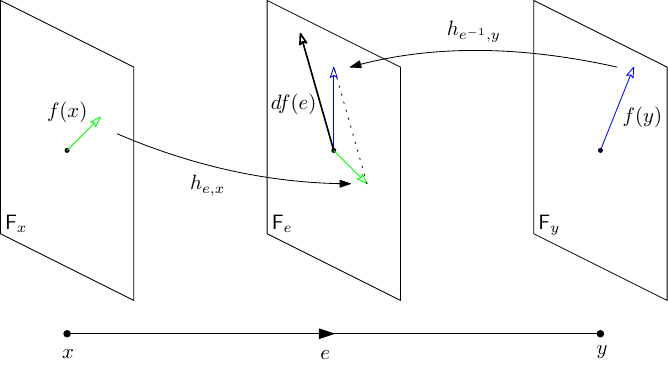}
\caption{\small The computation of $df(e)$.}
\end{center}
\end{figure}

Let us now assume that the graph $\G$ is a weighted graph with a well. Recall from \eqref{eq:defPxe} the notation $P_{x,e}=\frac{\chi_{e}}{\lambda_{x}}$.

\begin{definition}
The \new{codifferential} is the operator
\[d^{*}:\Omega^{1}(\F) \to \Omega^{0}(\F)\]
defined by setting, for all $\omega \in \Omega^{1}(\F)$ and all $x\in \U$,
\begin{equation}\label{d*propre}
d^{*} \omega(x)=\sum_{e\in \E : \ol{e}=x} P_{x,e} h_{x,e}\omega(e)=-\sum_{e\in \E : \ul{e}=x} P_{x,e} h_{x,e^{-1}}\omega(e).
\end{equation}
\end{definition}
The following lemma justifies the notation that we used for the codifferential.

\begin{lemma}\label{adjoint} The operators $d$ and $d^{*}$ are adjoint of each other with respect to the Hermitian forms $(\cdot,\cdot)_{\Omega^{0}}$ and $(\cdot,\cdot)_{\Omega^{1}}$.
\end{lemma}

Let us emphasise that the proof of this lemma depends on the fact that the connection $h$ consists in unitary operators.

\begin{proof}
Let $f\in\Omega^{0}(\U,\F)$ and $\omega\in\Omega^{1}(\E,\F)$. We have
\begin{align*}
(f,d^{*}\omega)_{\Omega^{0}}& = \sum_{x\in \V} \lambda_{x} \langle f(x),d^{*}\omega(x)\rangle_{x} \\
& = \frac{1}{2}\sum_{x\in \V} \sum_{e\in \E: \ol{e}=x}  \chi_{e} \langle f(x) , h_{x,e} \omega(e)\rangle_{x} -\frac{1}{2}\sum_{x\in \V} \sum_{e\in \E: \ul{e}=x} \chi_{e} \langle f(x), h_{x,e^{-1}} \omega(e)\rangle_{x}\\
& = \frac{1}{2}\sum_{e\in \E} \chi_{e} \langle f(\ol{e}), h_{\ol{e},e} \omega(e)\rangle_{\ol{e}}-\frac{1}{2}\sum_{e\in \E} \chi_{e} \langle  f(\ul{e}),h_{\ul{e},e^{-1}}\omega(e)\rangle_{\ul{e}}\\
& = \frac{1}{2}\sum_{e\in \E}  \chi_{e} \langle h_{e^{-1},\ol{e}} f(\ol{e}),\omega(e)\rangle_{e}-\frac{1}{2}\sum_{e\in \E} \chi_{e} \langle h_{e,\ul{e}} f(\ul{e}),\omega(e)\rangle_{e}\\
& =\frac{1}{2} \sum_{e\in \E} \chi_{e} \langle df (e),\omega(e)\rangle_{e}= (df,\omega)_{\Omega^{1}}\,.
\end{align*}
We used the unitarity of $h$ between the third and the fourth line.
\end{proof}

\subsection{Laplacians}\label{sec:Laplacians} To the situation that we are describing since the beginning of this paper, is associated a Laplace operator, which is our main object of interest. A simple definition of this operator, acting on $\Omega^{0}(\F)$, would be the following:
\[\ol{\Delta}_{h}=d^{*}\circ d \in \End(\Omega^{0}(\F)).\] 
However, the Laplacian that we will use most of the time is the compression of this operator to the subspace $\Omega^{0}(\V\setminus\W,\F)$ of sections that vanish on the well (see Definition \ref{def:Omegas}). In the following definition, we will use the notation $\pi_{\V\setminus \W}$, defined by \eqref{defpiS}, for the orthogonal projection of $\Omega^{0}(\F)$ onto $\Omega^{0}(\V\setminus \W,\F)$. The adjoint $\pi_{\V\setminus \W}^{*}$ of this projection is the inclusion of $\Omega^{0}(\V\setminus \W,\F)$ into $\Omega^{0}(\F)$.

\begin{definition} Let $\G$ be a weighted graph with a well. Let $\F$ be a vector bundle over $\G$. Let~$h$ be a connection on $\F$. The $h$-Laplacian, or simply Laplacian, is defined on $\Omega^{0}(\V\setminus \W,\F)$ as the linear mapping
\[\Delta_{h}=\pi_{\V\setminus \W} \circ d^{*} \circ d \circ \pi^{*}_{\V\setminus \W} \in \End(\Omega^{0}(\V\setminus \W,\F)).\]
\end{definition}

Consider $f\in \Omega^{0}(\V\setminus \W,\F)$ and choose $x\in \V\setminus \W$. Thanks to \eqref{d*propre}, we have
\begin{align}
(\Delta_{h} f)(x)&=-\sum_{e\in \E: \ul{e}=x}P_{x,e} h_{x,e^{-1}} (df)(e) \nonumber \\
&=\sum_{e\in \E: \ul{e}=x}P_{x,e}(f(x)-h_{e^{-1}} f(\ol{e})) \nonumber\\
&=f(x)-\sum_{e\in \E : \ul{e}=x} P_{x,e}\1_{\V\setminus \W}(\ol{e}) h_{e^{-1}} f(\ol{e}). \label{I-Ph}
\end{align}
This last expression will be useful in the sequel.

Using the Laplacian, we can define the Dirichlet energy functional on $\Omega^{0}(\V\setminus\W,\F)$.

\begin{definition}\label{Dirichlet energy} Let $f$ be an element of $\Omega^{0}(\V\setminus\W,\F)$. The  \new{Dirichlet energy} of $f$ is the real number
\[\mathcal{E}_{h}(f)=\left(f,\Delta_{h}f\right)_{\Omega^{0}}.\]
\end{definition}
Remembering that an element of $\Omega^{0}(\V\setminus \W,\F)$ is seen as a section of $\F$ over $\V$ vanishing over~$\W$, a short computation yields
\begin{equation}\label{eq:dirichlet-energy}
\mathcal{E}_{h}(f)=\frac{1}{2} \sum_{e\in \E}\chi_{e} \| f(\ul{e})-h_{e^{-1}}f(\ol{e})\|^{2}_{e},
\end{equation}
which makes it obvious that $\mathcal{E}_{h}(f)$ is non-negative.

\begin{proposition}\label{deltainv} The operator $\Delta_{h}$ is Hermitian, non-negative, and invertible on $\Omega^{0}(\V\setminus\W,\F)$.
\end{proposition}

\begin{proof} It follows from Lemma \ref{adjoint} that $\Delta_{h}=\pi_{\V\setminus \W}\circ d^{*}\circ d \circ \pi_{\V\setminus \W}$ is Hermitian and non-negative on $\Omega^{0}(\V\setminus \W,\F)$. There remains to prove that it is invertible.

Consider a section $f$ such that $\Delta_{h}f=0$. Since $f$ vanishes identically on the well, \eqref{eq:dirichlet-energy} and the fact that $\mathcal{E}_{h}(f)=0$ imply that $f$ vanishes everywhere on $\V$.
\end{proof}

In the special case of the trivial bundle $\K_{\G}$ over $\G$ endowed with the canonical connection (see the end of Section \ref{sec:connections}), we will denote the Laplacian simply by $\Delta$. 

\subsection{The smallest eigenvalue of the Laplacian}

It follows from Lemma \ref{deltainv} that the lowest eigenvalue of $\Delta_{h}$ is positive. In this section, we give a lower bound on this first eigenvalue.
Let us define
\begin{equation}\label{eq:rayleigh-quotient}
\sigma_h=\min \mathrm{Spec} \,\Delta_h = \min_{f\in\Omega^0(\V\setminus \W,\F)\setminus \{0\}}\frac{(f,\Delta_h f)_{\Omega^0}}{(f,f)_{\Omega^0}}\,.
\end{equation}
On the trivial bundle $\K$ with the usual Laplacian $\Delta$, we denote this quantity by $\sigma$.

\begin{proposition}[Discrete Kato's lemma]\label{kato}
The following inequality holds: 
\begin{equation*}
\sigma_h \ge \sigma >0\,.
\end{equation*}
\end{proposition}

\begin{proof} The inequality $\sigma>0$ is a special case of Lemma \ref{deltainv}. Let us prove that $\sigma_{h}\geq \sigma$. For this, consider a section $f\in\Omega^0(\V\setminus \W,\F)$. For all $e\in \E$, we have
\begin{align*}
\| f(\ul{e})-h_{e^{-1}}f(\ol{e})\|_{\ul{e}} \geq \left| \| f(\ul{e})\|_{\ul{e}}-\| h_{e^{-1}}f(\ol{e})\|_{\ul{e}} \right|=\left|\| f(\ul{e})\|_{\ul{e}}-\| f(\ol{e})\|_{\ol{e}}\right|\, ,
\end{align*}
because $h_{e}$ is unitary. Using~\eqref{eq:dirichlet-energy}, this implies the inequality
\begin{equation*}
(f,\Delta_{h}f)_{\Omega^0}\ge (\|f\|, \Delta \|f\|)\, ,
\end{equation*}
where $\|f\|$ is the function on $\V$ defined by $\|f\|(x)=\|f(x)\|_{x}$. Since $(f,f)_{\Omega^{0}}=(\|f\|,\|f\|)$, the inequality $\sigma_{h}\geq \sigma$ follows. 
\end{proof}

\subsection{Generalised Laplacians}\label{sec:GLHS} We will consider linear operators on the space of sections of~$\F$ that are similar to, but more general than the Laplacian $\Delta_{h}$. In analogy with the case of a Hermitian vector bundle over a Riemannian manifold, for which there exists many second order differential operators with the same symbol, which all deserve to be called Laplacians, and which differ by a term of order $0$, we set the following definition.

Recall from Definition \ref{def:potentiel} that we call potential a section of the vector bundle $\End(\F)$ over~$\V$ such that $H_{x}$ is Hermitian (or, in the real case, symmetric) for each $x\in \V$. There is a natural inclusion $\Omega^{0}(\End(\F))\subset \End(\Omega^{0}(\F))$ thanks to which a potential can act on a section of $\F$. Concretely, if $H$ is a potential and $f$ a section of $\F$, then the section $Hf$ is simply defined by the fact that for all $x\in \V$,
\[(Hf)(x)=H_{x}(f(x)).\]

\begin{definition} Let $\F$ be a Hermitian vector bundle over a graph with a well $\G$, endowed with a connection $h$. Let $H$ be a potential on $\F$. The {\em generalised Laplacian} on $\F$ associated to $h$ and~$H$, or $(h,H)$-Laplacian, is the following linear endomorphism of $\Omega^{0}(\V\setminus\W,\F)$:
\[\Delta_{h,H}=\Delta_{h}+H.\]
\end{definition}

Imitating Definition \ref{Dirichlet energy}, we define the generalised Dirichlet energy of a section $f$ as the number
\begin{align}\label{generalised Dirichlet energy}
\mathcal{E}_{h,H}(f)&=\left(f,(\Delta_{h}+H)f\right)_{\Omega^{0}}\\
&=\frac{1}{2} \sum_{e\in \E}\chi_{e} \| f(\ul{e})-h_{e^{-1}}f(\ol{e})\|^{2}_{e}+\sum_{x\in \V}\lambda_{x} \langle f(x),H_{x}f(x)\rangle_{x}.\nonumber
\end{align}
Recall from the previous section that we denote by $\sigma_{h}$ the smallest eigenvalue of $\Delta_{h}$, and that $\sigma_{h}>0$. If $H_{x}>-\sigma_{h}$ for every vertex $x$, then the operator $\Delta_{h,H}$ is invertible. This is in particular the case as soon as $H_{x}\geq 0$ for every vertex $x$.

Let us conclude by describing the action of the gauge group on the operators that we just defined. The gauge group acts naturally on $\Omega^{0}(\F)$ : for all gauge transformation $j$ and all section~$f$, the section $j\cdot f$ is defined by
\[(j\cdot f)(x)=j_{x}(f(x)).\]
This action extends to an action by conjugation on $\End(\Omega^{0}(\F))$: if $B$ is an endomorphism of~$\Omega^{0}(\F)$ and $f$ is a section, then
\[(j\cdot B)(f)=j\cdot (B(j^{-1}\cdot f)).\]
These actions allow us to express the way in which the Laplacians are acted on by the gauge group: with our current notation,
\begin{equation}\label{eq:jDelta}
j\cdot \Delta_{h,H}=\Delta_{j\cdot h,j\cdot H}.
\end{equation}
Since the gauge group acts on each fibre by unitary transformations, it follows that the Dirichlet energy satisfies
\[\mathcal{E}_{j\cdot h, j\cdot H}(j\cdot f)=\mathcal{E}_{h,H}(f).\]

\section{A covariant Feynman--Kac formula}
\label{sec:covFK}

\subsection{A Feynman--Kac formula} In this section, we will prove our first theorem, which is a version of the Feynman--Kac formula featuring holonomies, and which will be one of the most useful results for our study. 

\begin{theorem}[Feynman--Kac formula]\label{prop:feynman-kac}
Let $\G$ be a weighted graph with a well. Let $\F$ be a vector bundle over $\G$, endowed with a connection $h$ and a potential $H$. Let $X$ be the random walk on $\G$ defined in Section \ref{sec:RW}.

Let $f$ be an element of $\Omega^{0}(\V\setminus \W,\F)$. For all $x\in \V\setminus \W$ and all $t\geq 0$, the following equality holds:
\[\left(e^{-t \Delta_{h,H}} f\right)(x)=\int_{\CP(\G)} \hol_{h,H}(\gamma_{|[0,t)}^{-1}) f(\gamma_{t})\; \d\P_{x}(\gamma).
\]
\end{theorem}

Note that since $f$ vanishes identically on the well, the paths that reach the well before time $t$ do not contribute to the right-hand side.

\begin{proof} We will use the fact that for any two endomorphisms $A$ and $B$ of a finite dimensional vector space and all real $t>0$,
\[e^{t(A+B)}=\sum_{k=0}^{\infty}\int_{0<t_{0}<\ldots<t_{k-1}<t} e^{t_{0}A}Be^{(t_{1}-t_{0})A}\ldots B e^{(t-t_{k-1})A}\; \d t_{0}\ldots \d t_{k-1},\]
a formula that can be checked by expanding the exponentials on the right-hand side in power series, computing the Eulerian integrals that appear, and observing that one recovers the left-hand side. In our context, we use this formula in the following way: we first write $e^{-t\Delta_{h,H}}=e^{-t}e^{-tH+t(\Id-\Delta_{h})}$ and deduce that
\[e^{-t\Delta_{h,H}}=e^{-t}\sum_{k=0}^{\infty} \int_{0<t_{0}<\ldots<t_{k-1}<t} e^{-t_{0}H} (\Id-\Delta_{h})e^{-(t_{1}-t_{0})H}\ldots (\Id-\Delta_{h}) e^{-(t-t_{k-1})H} \; \d t_{0}\ldots \d t_{k-1}.\]
Let $f$ be a section of $\F$ and $x$ a vertex of $\G$. Using the expression \eqref{I-Ph} of $\Delta_{h}$, we find
\begin{align}\nonumber
(e^{-t\Delta_{h,H}}f)(x)&=e^{-t}\sum_{k=0}^{\infty} \sum_{e_{1},\ldots,e_{k}\in \E}
P_{x,e_{1}}P_{\ol{e_{1}},e_{2}}\ldots P_{\ol{e_{k-1}},e_{k}} \prod_{q=1}^{k} \1_{\V\setminus\W}(\ol{e_{q}})\\ 
& \hspace{-.5cm} \int_{0<t_{0}<\ldots<t_{k-1}<t}  e^{-t_{0}H_{x}}h_{e_{1}^{-1}} e^{-(t_{1}-t_{0})H_{\ol{e_{1}}}}\ldots h_{e_{k}^{-1}} e^{-(t-t_{k-1})H_{\ol{e_{k}}}}f(\ol{e_{k}})\; \d t_{0}\ldots \d t_{k-1}.\label{exptdelta1}
\end{align}
On the other hand, according to \eqref{defQx} and \eqref{defPx},
\begin{align*}
\int_{\CP(\G)} \hol_{h,H}(\gamma_{|[0,t)}^{-1}) f(\gamma_{t})\; \d\P_{x}(\gamma)
&= 
  \sum_{\substack{0\leq k < n \\ p\in \DP_{n}(\G)\\ p=(x_{0},e_{1},\ldots,e_{n},x_{n})}}  \Q_{x}(\{p\})  \\
& \hspace{-4.5cm}  \int_{0<t_{0}<\ldots<t_{k-1}<t<t_{k}<\ldots < t_{n-1}} \hspace{-1cm}  e^{-t_{0}H_{x}}h_{e_{1}^{-1}} e^{-(t_{1}-t_{0})H_{x_{1}}}\ldots h_{e_{k}^{-1}} e^{-(t-t_{k-1})H_{x_{k}}}f(x_{k})e^{-t_{n-1}} \; \d t_{0}\ldots \d t_{n-1}\\
&\hspace{-4.5cm}= e^{-t}  \sum_{k=0}^{\infty}\sum_{e_{1},\ldots,e_{k}\in \E }  P_{x,e_{1}}P_{\ol{e_{1}},e_{2}}\ldots P_{\ol{e_{k-1}},e_{k}}\prod_{q=1}^{k} \1_{\V\setminus\W}(\ol{e_{q}}) \\
&\hspace{-3.5cm} \Bigg( \sum_{l=0}^{\infty} \sum_{e'_{1},\ldots,e'_{l}\in \E} P_{\ol{e_{k}},e'_{1}}P_{\ol{e'_{1}},e'_{2}}\ldots P_{\ol{e'_{l-1}},e'_{l}}  \1_{\W}(\ol{e'_{l}})\prod_{r=1}^{l} \1_{\V\setminus\W}(\ol{e'_{r}}) \Bigg)\\
& \hspace{-2.5cm}  \int_{0<t_{0}<\ldots<t_{k-1}<t} \hspace{0cm}  e^{-t_{0}H_{x}}h_{e_{1}^{-1}} e^{-(t_{1}-t_{0})H_{\ol{e_{1}}}}\ldots h_{e_{k}^{-1}} e^{-(t-t_{k-1})H_{\ol{e_{k}}}}f(\ol{e_{k}}) \; \d t_{0}\ldots \d t_{k-1}.
\end{align*}
The sum over $l$ between the brackets is the total mass of the probability measure $\Q_{\ol{e_{k}}}$, that is,~$1$, and we recover the right-hand side of \eqref{exptdelta1}.
\end{proof}

We will often use a slightly different, but equivalent, version of the Feynman--Kac formula. In order to state it, let us introduce the following notation. The vector space of linear endomorphisms of $\Omega^{0}(\V\setminus \W,\F)$ is isomorphic to
\[\End(\Omega^{0}(\V\setminus \W,\F))=\End\Big(\bigoplus_{x\in \V\setminus \W}\F_{x}\Big)=\bigoplus_{x,y\in \V\setminus \W}\Hom(\F_{y},\F_{x}).\]
This decomposition is nothing more than the block decomposition of the matrix representing an endomorphism of $\Omega^{0}(\V\setminus\W,\F)$, each block corresponding to one of the fibres of the bundle. 
For all linear operator $B\in \End(\Omega^{0}(\V\setminus \W,\F))$ and all $x,y\in \V$, we will denote by $B_{x,y}$ the linear map from $\F_{y}$ to $\F_{x}$ which appears in the right-hand side of the decomposition above.

Proposition \ref{prop:feynman-kac} immediately implies the following.

\begin{corollary}\label{FKxy} Under the assumptions of Theorem \ref{prop:feynman-kac}, and for all $x,y\in \V\setminus\W$, the following equality holds in $\Hom(\F_{y},\F_{x})$:
\[\left(e^{-t \Delta_{h,H}}\right)_{x,y}=\int_{\CP(\G)} \hol_{h,H}(\gamma_{|[0,t)}^{-1}) \1_{\{\gamma_{t}=y\}}\; \d\P_{x}(\gamma).\]
\end{corollary}

Provided we are careful enough about the source and target spaces of the linear operators that we are writing, we can express the corollary in the following more compact form:
\[e^{-t \Delta_{h,H}}=\bigoplus_{x\in \V} \left(e^{-t \Delta_{h,H}}\right)_{x,y}=\int_{\CP(\G)} \hol_{h,H}(\gamma_{|[0,t)}^{-1}) 
\; \d\P_{x}(\gamma).\]

\subsection{The Green section}\label{sec:Green} We devote this very short section to the definition of the analogue, in our situation, of the Green function. 
For this, let us define the operator $\Lambda$ on $\Omega^{0}(\F)$ such that for all section $f$ and all vertex $x\in \V$, 
\[(\Lambda f)(x)=\lambda_{x}f(x).\]

\begin{definition} Let $\G$ be a weighted graph with a well. Let $\F$ be a fibre bundle over $\G$. Let $h$ be a connection on $\F$ and $H$ be a potential on $\F$. The \new{Green section} associated with $h$ and $H$ is the operator
\[G_{h,H}= \left(\Lambda\circ \Delta_{h,H}\right)^{-1} \in \End(\Omega^{0}(\V\setminus \W,\F)).\]
\end{definition}

According to \eqref{eq:jDelta}, given a gauge transformation $j$ of $\F$, we have
\begin{equation}\label{eq:Green-equivariance}
j\cdot G_{h,H}=j\circ G_{h,H}\circ j^{-1}=G_{j\cdot h,j\cdot H}.
\end{equation}

Let us make a brief comment about the distinction between sections and measures. The operator $\Lambda$ takes a section and multiplies it by a measure. The result of this operation is an object that can be paired pointwise with a section, using the Hermitian scalar product of each fibre, thus producing a scalar function that can be integrated over the graph. It would thus be fair to say that $\Lambda$ sends $\Omega^{0}(\F)$ into its dual space, and, since $\Delta_{h,H}$ is a genuine operator on $\Omega^{0}(\V\setminus\W,\F)$, that $G_{h,H}$ sends the dual of $\Omega^{0}(\V\setminus \W,\F)$ into $\Omega^{0}(\V\setminus \W,\F)$ itself. In a sense that we will not make very precise, the kernel of $G_{h,H}$ is thus analogous to a function of two variables.

\subsection{Two elementary algebraic identities} The Feynman-Kac formula relates the semigroup of the generalised Laplacian to the average of the twisted holonomy along the paths of the random walk in the graph. In this paragraph, we will derive two consequences of this formula which will be the bases of our proofs of the isomorphism theorems. The two formulas which we will now prove ultimately rely on the following elementary lemma. 

\begin{lemma}
Let $A$ and $B$ be two positive Hermitian matrices of the same size. Then 
\begin{equation}\label{alg1}
\int_{0}^{\infty} e^{-tA}\; \d t=A^{-1}
\end{equation}
and
\begin{equation}\label{alg2}
\int_0^\infty \frac{e^{-t A}-e^{-t B}}{t} \;\d t = \log B- \log A\,.
\end{equation}
\end{lemma}

\begin{proof}
It suffices to prove that for all positive real $a$,
\[\int_0^\infty \frac{e^{-t a}-e^{-t}}{t} \;\d t = - \log a\,.\]
Now, for all $\epsilon>0$, $\int_{\epsilon}^{\infty} \frac{e^{-ta}}{t}\; \d t=\int_{a\epsilon}^{\infty} \frac{e^{-t}}{t}\; \d t$ and the integral that we want to compute is equal to $\int_{a\epsilon}^{\epsilon} \frac{e^{-t}}{t}\; \d t=\int_{a\epsilon}^{\epsilon} \frac{1}{t}\; \d t+O(\epsilon)$, from which the result follows immediately.
\end{proof}

From the Feynman-Kac formula (Theorem \ref{prop:feynman-kac}), the definition of the measure $\nu$ (Definition \ref{defnu}), and \eqref{alg1}, we deduce the following proposition.

\begin{proposition}\label{propholdeltanu} For all $x,y\in \V\setminus\W$ the equality
\begin{equation}\label{masterDE1}
\int_{\CP(\G)} \hol_{h,H}(\gamma^{-1}) \; \d\nu_{x,y}(\gamma) = \big(G_{h,H}\big)_{x,y}
\end{equation}
holds in $\Hom(\F_{y},\F_{x})$. Moreover, we have the following equality in $\End(\Omega^{0}(\V\setminus\W,\F))$:
\begin{equation}\label{masterDE}
\int_{\CP(\G)} \hol_{h,H}(\gamma^{-1}) \; \d\nu(\gamma) = \Delta_{h,H}^{-1}.
\end{equation}
\end{proposition}

Similarly, from Theorem \ref{prop:feynman-kac}, Definition \ref{defmu} and \eqref{alg2}, we deduce the following. 

\begin{proposition}\label{prop:hol-delta}
Let $h,h'$ be two connections on $\F$ and $H,H'$ be two non-negative potentials. We have the following identity in $\End(\Omega^0(\V\setminus\W,\F))$:
\begin{equation}\label{masterLJS}
\int_{\CP(\G)}\left(\hol_{h,H}(\gamma^{-1})-\hol_{h',H'}(\gamma^{-1})\right)\; \d\mu(\gamma)=\log \Delta_{h',H'}-\log\Delta_{h,H}\,. 
\end{equation}
\end{proposition}

It turns out that \eqref{masterLJS} is difficult to use in its present form, and that it is often more convenient to use the less detailed but still informative equation obtained from it by taking the trace of both sides. Here, by the trace, we mean the usual trace in the space of endomorphisms of $\Omega^{0}(\V\setminus\W,\F)$. Thus, if $B$ is such an endomorphism, we call trace of $B$ the number
\[\Tr(B)=\sum_{x\in \V\setminus\W} \Tr_{\F_{x}}(B_{x,x}).\]

\begin{corollary}\label{coro:log-det-2}
Let $h$ be a connection and $H$ a non-negative potential. We have the identity
\begin{equation}\label{masterLJStr}
\int_{\CL(\G)}\left(\Tr\,\hol_{h,H}(\gamma^{-1})-\Tr\,\hol_{h}(\gamma^{-1})\right)\; \d\mupl(\gamma)=\log\frac{\det\Delta_{h}}{\det\Delta_{h,H}}\,. 
\end{equation}
\end{corollary}

The reader may wonder why we wrote Proposition~\ref{prop:hol-delta} as we did with a difference of holonomies, and what the integral of the twisted holonomy along a random path with respect to the measure~$\mu$ is. The answer is that this integral is ill defined, because of the presence of the infinite measure~$\mup$ within $\mu$. However, it makes sense to compute this integral against $\mu-\mup$. The result is given by the next proposition.

\begin{proposition}
Let $h$ be a connection and $H$ a positive potential. In $\End(\Omega^0(\V\setminus\W,\F))$, we have the identity
\begin{equation}\label{iplush}
\int_{\CP(\G)}\hol_{h,H}(\gamma^{-1})\; \d(\mu-\mup)(\gamma)=-\log \Delta_{h,H} + \log(\Id+H)\,. 
\end{equation}
\end{proposition}

\begin{proof} Consider $x$ and $y$ in $\V$. We have
\begin{align*}
\mbox{l.h.s. of } \eqref{iplush}&=\frac{1}{\lambda_{y}}  \int_{\CP(\G)\times (0,+\infty)}
\hol(\gamma_{|[0,t)}^{-1}) 
(1-\1_{\{\gamma_{|[0,t)}=(x,t)\}}) \;\d\P_{x}(\gamma)\frac{\d t}{t}\\
&=\frac{1}{\lambda_{y}}\int_{0}^{\infty} \left((e^{-t\Delta_{h,H}})_{x,y}-\delta_{x,y} e^{-t}e^{-tH_{x}} \right) \frac{\d t}{t}\\
&=\frac{1}{\lambda_{y}}\left(\int_{0}^{\infty} \frac{e^{-t\Delta_{h,H}}- e^{-t(\Id +H)}}{t}\; \d t\right)_{x,y}.
\end{align*}
Summing over $y$ with respect to the reference measure $\lambda$ and over $x$ with respect to the counting measure, and using \eqref{alg2}, we find the expected result. \end{proof}

\section{The Gaussian free vector field}\label{sec:GFF}

In this section, we consider a weighted graph with a well $\G$, over which we are given a vector bundle $\F$ with a connection $h$ and a potential $H$. In this setting, we will construct a Gaussian probability measure on $\Omega^{0}(\V\setminus \W,\F)$.

\subsection{Probability measures on $\Omega^{0}(\V\setminus \W,\F)$} \label{sec:prob measures sections}

Recall from Section \ref{sec:GLHS} that if $H_{x}$ is a non-negative Hermitian operator on $\F_{x}$ for every $x\in \V\setminus\W$, then the operator $\Delta_{h,H}$ is positive Hermitian on $\Omega^{0}(\V\setminus \W,\F)$. We are going to use it to define a probability measure on $\Omega^{0}(\V\setminus\W,\F)$.

Let us first discuss Lebesgue measures on $\Omega^{0}(\V\setminus\W,\F)$. Let us agree that the natural Lebesgue measure on a Euclidean or Hermitian space is that which gives measure $1$ to any real cube the edges of which have length $1$. With this convention, there is a natural Lebesgue measure $\Leb_{\Omega^{0}}$ on $\Omega^{0}(\V\setminus\W,\F)$ and, for each $x\in \V\setminus\W$, a natural Lebesgue measure $\Leb_{x}$ on $\F_{x}$. Considering the way in which the Euclidean or Hermitian structures on $\Omega^{0}(\V,\F)$ on one hand and on the fibres of $\F$ on the other hand are related, through the measure $\lambda$, we find the equality
\[\Leb_{\Omega^{0}}=\bigotimes_{x\in \V\setminus\W} \left(\lambda_{x}^{\frac{\beta}{2}r}\;  \Leb_{x}\right),\]
where, in order to treat the Euclidean and Hermitian cases simultaneously, we used the constant
\begin{equation}\label{beta}
\beta=\dim_{\R}\K,
\end{equation}
equal to $1$ in the Euclidean case and to $2$ in the Hermitian case.

Let us denote by $|\V|$ the cardinality of $\V$. We define the probability measure 
$\P^{h,H}$ on $\Omega^{0}(\V\setminus\W,\F)$ by setting
\begin{equation}\label{defPhH}
\d\P^{h,H}(f)=\left(\frac{\det \Delta_{h,H}}{\pi^{r|\V|}}\right)^{\frac{\beta}{2}}e^{-\frac{1}{2}(f,\Delta_{h,H} f)_{\Omega^{0}}} \;\d\Leb_{\Omega^{0}}(f).
\end{equation}
We shall denote by $\Phi$ the canonical process (that is, the identity map) on $\Omega^{0}(\V\setminus\W,\F)$, so that for all bounded measurable function $F:\Omega^{0}(\V\setminus\W,\F)\to \R$, we have
\[\Ex^{h,H}[F(\Phi)]=\int_{\Omega^{0}(\V,\F)} F(f)\; \d\P^{h,H}(f).\]
The random section $\Phi$ is called the \new{Gaussian free vector field}, or simply Gaussian free field, on~$\G$ associated with $h$ and $H$.

We will often consider the Gaussian free vector field associated to a connection $h$ and the zero potential. In that case, we will use the notation $\Ex^{h}$ instead of $\Ex^{h,0}$.

\subsection{Covariance} The measure $\P^{h,H}$ is Gaussian on $\Omega^{0}(\V\setminus\W,\F)$ and we will now compute its covariance function. Let us introduce some notation which will be useful to express this covariance, and several of the results that we shall prove. 

Consider a Euclidean or Hermitian vector space $(V,(\cdot,\cdot))$. Recall that if $V$ is Hermitian, we take the Hermitian scalar product to be linear in the second variable. We define the \new{conjugation map} to be the following:
\begin{align*}
V & \longrightarrow V^{*}\\
v & \longmapsto \ol{v}=(v, \cdot).
\end{align*}
This map is an antilinear isomorphism between $V$ and its dual, and it satisfies, for every scalar~$z$ and every vector $v$, the relation $\ol{zv}=\ol{z}\; \ol{v}$.

If $V$ and $W$ are Euclidean or Hermitian spaces, and if $v$ and $w$ are elements of $V$ and $W$ respectively, then $w\otimes \ol{v}$ is an element of $W\otimes V^{*}\simeq\Hom(V,W)$. More specifically, $w\otimes \ol{v}$ is the linear map of rank 1 from $V$ to $W$ such that for all $u\in V$,
\[(w\otimes \ol{v})(u)=(v,u) w.\]
With this notation, a standard Gaussian computation yields the following identity.

\begin{proposition}\label{covariancebrute} Let $\F$ be a vector bundle over $\G$, endowed with a connection $h$ and a potential~$H$. Let $\Phi$ be the associated Gaussian free vector field on $\G$. Then
\[\Ex^{h,H}[\Phi\otimes \ol{\Phi}]=\Delta_{h,H}^{-1}.\]
In other words, for all $f,g\in \Omega^{0}(\V\setminus\W,\F)$, 
\begin{equation}\label{covpastropbrute}
\Ex^{h,H}\left[(f,\Phi)_{\Omega^{0}}(\Phi,g)_{\Omega^{0}}\right]=(f,\Delta_{h,H}^{-1}g)_{\Omega^{0}}.
\end{equation}
\end{proposition}
Note that this formulation is true in the Euclidean case as well as in the Hermitian case.

It is often useful to have an expression of the covariance of the values of the Gaussian free vector field in two specific fibres of $\F$. Passing from the space of sections to the individual fibres involves the reference measure $\lambda$, so that the inverse of the Laplacian will be replaced by the Green section. 

\begin{proposition}\label{prop:covariance} Let $\F$ be a vector bundle over $\G$, endowed with a connection $h$ and a potential~$H$. Let $\Phi$ be the Gaussian free vector field on $\G$. Let $G_{h,H}$ be the Green section of $\F$. For all $x,y\in \V\setminus\W$, we have the following identity in $ \Hom(\F_{y},\F_{x})$:
\begin{equation}\label{EphiG}
\Ex^{h,H}\Big[\Phi_{x}\otimes \ol{\Phi_{y}}\Big]=(G_{h,H})_{x,y}.
\end{equation}
\end{proposition}

\begin{proof} Let us choose two vertices $x$ and $y$ in $\V$ and two vectors $\xi\in \F_{x}$ and $\eta\in \F_{y}$. Applying Proposition \ref{covariancebrute} to the sections $f=\xi \1_{\{x\}}$ and $g=\eta \1_{\{y\}}$, which vanish respectively everywhere except at $x$ and $y$, and satisfy $f(x)=\xi$ and $g(y)=\eta$, we find
\[\lambda_{x} \lambda_{y} \Ex^{h,H}\left[\langle \xi,\Phi_{x}\rangle_{x}\langle \Phi_{y},\eta\rangle_{y}\right]=\lambda_{x} \langle \xi,\big(\Delta_{h,H}^{-1}\big)_{x,y} \eta\rangle_{x}.\]
This equality can be written
\[\Ex^{h,H}\left[\langle \xi,\Phi_{x}\rangle_{x}\langle \Phi_{y},\eta\rangle_{y}\right]= \langle \xi,\big(\Delta_{h,H}^{-1}\Lambda^{-1}\big)_{x,y} \eta\rangle_{x},\]
or even
\[\left\langle \xi,\Ex^{h,H}\Big[\Phi_{x}\otimes \ol{\Phi_{y}}\Big]\eta\right\rangle_{x}= \langle \xi,\left(G_{h,H}\right)_{x,y} \eta\rangle_{x},\]
which, because it is true for all $\xi$ and $\eta$, implies the result.
\end{proof}

This proposition and \eqref{eq:Green-equivariance} allow us to understand how the Gaussian free vector field transforms under the action of a gauge transformation. Indeed, let us choose $j\in \Aut(\F)$. For all vertices $x,y$ and all $\xi\in \F_{x}$ and $\eta\in \F_{y}$, we have
\begin{align*}
\Ex^{h,H}\big[\langle \xi,j_{x}(\Phi_{x})\rangle_{x} \langle j_{y}(\Phi_{y}),\eta\rangle_{y}\big]&=\Ex^{h,H}\big[\langle j_{x}^{-1}(\xi),\Phi_{x}\rangle_{x} \langle \Phi_{y},j_{y}^{-1}(\eta)\rangle_{y}\big]\\
&=\langle\xi,(j_{x}\circ (G_{h,H})_{x,y}\circ j_{y}^{-1})(\eta)\rangle_{x}\\
&=\langle\xi,(G_{j\cdot h,j\cdot H})_{x,y}(\eta)\rangle_{x}\\
\end{align*}
Thus, if $\Phi$ is the Gaussian free vector field associated to the pair $(h,H)$ and if $j$ is a gauge transformation, then the Gaussian free vector field associated to the pair $(j\cdot h, j\cdot H)$ has the same distribution as $j(\Phi)$.

\subsection{Laplace transform}
For future reference, let us record the Laplace transform and a formula for the higher moments of the Gaussian free vector field. We keep the notation of the previous section.

The next proposition follows from \eqref{covpastropbrute} and the standard computation of a Gaussian integral. Note that it is written in such a way as to be true in the real case as well as in the complex case.

\begin{proposition}\label{prop:Laplace} For all $f\in\Omega^0(\V\setminus\W,\F)$, we have
\[\Ex^{h,H}\left[\left|e^{\frac{1}{2}(f,\Phi)_{\Omega^0}}\right|^{2}\right]=e^{\frac{1}{2}(f,\Delta_{h,H}^{-1}f)_{\Omega^0}}\,.\]
\end{proposition}

The following formula, which is a reformulation of Wick's theorem (see \cite{Janson}) in our context, combined with \eqref{masterDE}, will be useful in the proof of the covariant Symanzik identity (Theorem \ref{thm:Symanzik}). It is one of the results where the real and complex case are quite different.

\begin{proposition}\label{wick} Assume that $\K=\R$. Then, for all integer $k\geq 0$ and all sections $f_1,\ldots,f_{k}\in\Omega^0(\V\setminus\W,\F)$, we have
\begin{equation}\label{Wickreel}
\Ex^{h,H}\left[\prod_{i=1}^{k}(f_i,\Phi)_{\Omega^0}\right]=\sum_{\pi}\prod_{\{i,j\}\in \pi}\int_{\CP(\G)}\lambda_{\ul{\gamma}} \langle f_{i}(\ul{\gamma}),\hol_{h,H}(\gamma^{-1})f_{j}(\ol{\gamma})\rangle_{\ul{\gamma}}\; \d\nu(\gamma),
\end{equation}
where the sum is over all partitions of the set $\{1,\ldots,k\}$ by pairs. In particular, this expectation is zero if $k$ is odd.

Assume that $\K=\C$. Then, for all integers $k,l\geq 0$, all $f_1,\ldots,f_{k}$ and $f'_{1},\ldots,f'_{l}$ in $\Omega^0(\V\setminus\W,\F)$, we have
\begin{equation}\label{Wickcomplexe}
\Ex^{h,H}\left[\prod_{i=1}^{k}(f_i,\Phi)_{\Omega^0}\prod_{j=1}^{l}(\Phi,f'_{j})_{\Omega^{0}}\right]=\sum_{\sigma}\prod_{i=1}^{k}\int_{\CP(\G)}\lambda_{\ul{\gamma}} \langle f_{i}(\ul{\gamma}),\hol_{h,H}(\gamma^{-1})f'_{\sigma(i)}(\ol{\gamma})\rangle_{\ul{\gamma}}\; \d\nu(\gamma),
\end{equation}
where the sum is over all bijections from $\{1,\ldots,k\}$ to $\{1,\ldots,l\}$. In particular, this expectation is zero if $k\neq l$.
\end{proposition}

\subsection{Square of the shifted Gaussian free vector field}

In this section, and in preparation for the proof of Theorem \ref{thm:LeJanSznitman}, we record a useful lemma which expresses the Laplace transform of the square of the shifted Gaussian free vector field. This is a generalisation of~\cite[Proposition~2.14]{Sznitman-book} to the present covariant setting. Recall the definition of $\sigma_{h}$ from Proposition \ref{kato}.

\begin{proposition}\label{useful-gaussian-proposition}
Let $h$ be a unitary connection and $H$ a potential on $\F$. Assume that $H> -\sigma_{h}$. For all $f\in\Omega^0(\V\setminus\W,\F)$, we have
\begin{equation}\label{ugp}
\Ex^h\left[e^{-\frac{1}{2}\left(\Phi+f,H (\Phi+f)\right)_{\Omega^0}}\right]=\Ex^h\left[e^{-\frac{1}{2}\left(\Phi,H \Phi\right)_{\Omega^0}}\right] e^{\frac{1}{2}\left(\Delta_h f, (\Delta_{h,H}^{-1}-\Delta_{h}^{-1})\Delta_h f\right)_{\Omega^0}}.
\end{equation}
\end{proposition}

\begin{proof} Let us compute the left-hand side of the equality to prove. We will successively apply the definition \eqref{defPhH} of the measure $\P^{h}$, writing $\d\phi$ instead of $\d\Leb_{\Omega^{0}}(\phi)$, then perform the usual completion of the square, then use the invariance by translation of the Lebesgue measure. We find
\begin{align*}
\left(\frac{\pi^{r|\V|}}{\det \Delta_{h}}\right)^{\frac{\beta}{2}} \Ex^{h}\left[e^{-\frac{1}{2}\left(\Phi+f,H (\Phi+f)\right)_{\Omega^0}}\right]&=\int_{\Omega^{0}(\V\setminus\W,\F)} e^{-\frac{1}{2}\left(\phi+f,H (\phi+f)\right)_{\Omega^0}} e^{-\frac{1}{2}(\phi,\Delta_{h}\phi)_{\Omega^{0}}} \;\d\phi\\
&\hspace{-2.5cm}=\int_{\Omega^{0}(\V\setminus\W,\F)} e^{-\frac{1}{2}\left(\phi+\Delta_{h,H}^{-1}Hf, \Delta_{h,H} (\phi +\Delta_{h,H}^{-1}Hf)\right)_{\Omega^{0}}} \;\d\phi \ e^{\frac{1}{2}\left(f,(H\Delta_{h,H}^{-1}H-H) f\right)_{\Omega^{0}}}\\
&\hspace{-2.5cm}=\int_{\Omega^{0}(\V\setminus\W,\F)} e^{-\frac{1}{2}\left(\phi, \Delta_{h,H} \phi\right)_{\Omega^{0}}}\;\d\phi \ e^{\frac{1}{2}\left(\Delta_{h} f,(\Delta_{h,H}^{-1}-\Delta_{h}^{-1}) \Delta_{h}f\right)_{\Omega^{0}}}.
\end{align*}
In the second exponential of the last line, we used the equality 
\[a(a+b)^{-1}a-a=(a+b-b)(a+b)^{-1}(a+b-b)-a=b(a+b)^{-1}b-b\]
which is true in any associative algebra and that we applied in this instance to $a=H$ and $b=\Delta_{h}$ in the algebra $\End(\Omega^{0}(\V\setminus\W,\F))$.

Now, in the integral on the right-hand side we recognise, up to the normalisation factor by which we multiplied both sides, the expectation on the right-hand side of \eqref{ugp}. 
\end{proof}

\section{Isomorphism theorems of Dynkin and Eisenbaum}\label{sec:Dynkin}

In this section, we explain how the two classical isomorphism theorems of Dynkin \cite{Dynkin} and Eisenbaum \cite{Eisenbaum} extend to the present setting, in a way which incorporates the geometry of the vector bundle over the graph. 
Our main tool will be the following combination of \eqref{masterDE1} and \eqref{EphiG}, which reads, for any two vertices $x$ and $y$ of a weighted graph with a well,
\begin{equation}\label{covariance-holonomy}
\Ex^{h,H}\Big[\Phi_x\otimes\ol{\Phi_y}\Big]=\left(G_{h,H}\right)_{x,y}=\int_{\CP(\G)} \hol_{h,H}(\gamma^{-1})\; \d\nu_{x,y}(\gamma).
\end{equation}

\subsection{Dynkin's isomorphism} Our generalisation of Dynkin's isomorphism is the following. We use the notation $\int_{\nu_{x,y}}$ to denote the integral with respect to the finite measure $\nu_{x,y}$. Recall from \eqref{beta} the definition of the constant $\beta$.

\begin{theorem}\label{thm:Dynkin} Let $\G$ be a weighted graph with a well. Let $\F$ be a Hermitian or Euclidean vector bundle over $\G$. Let $h$ be a connection on $\F$ and $H$ a potential on $\F$. For all vertices $x,y$ of~$\V\setminus\W$, the following identity holds in $\Hom(\F_y,\F_x)$:
\begin{equation}\label{Dynkin}
\Ex^h \otimes \mbox{$\int_{\nu_{x,y}}$}  \left[e^{-\frac{1}{2}(\Phi,H\Phi)_{\Omega^0}}\,  \hol_{h,H}(\gamma^{-1})\right]=\Ex^h\left[e^{-\frac{1}{2}(\Phi,H\Phi)_{\Omega^0}}\, \Phi_x\otimes\ol{\Phi_y} \right]\,.
\end{equation}
\end{theorem}

In particular, if $\xi\in \F_{x}$ and $\eta\in \F_{y}$,
\begin{equation}\label{eq:DynkinSandwich}
\Ex^h \otimes \mbox{$\int_{\nu_{x,y}}$}  \left[e^{-\frac{1}{2}(\Phi,H\Phi)_{\Omega^0}}\, \langle  \xi,\hol_{h,H}(\gamma^{-1})(\eta)\rangle_{x} \right]=\Ex^{h,H}\left[\langle \xi,\Phi_{x}\rangle_{x}\langle \Phi_{y},\eta\rangle_{y}\right]
\end{equation}
\begin{proof}
By definition of the measure $\P^{h,H}$, we have
\[\frac{\Ex^h\left[e^{-\frac{1}{2}(\Phi,H\Phi)_{\Omega^0}}\Phi_x\otimes\ol{\Phi_y} \right]}{\Ex^h\left[e^{-\frac{1}{2}(\Phi,H\Phi)_{\Omega^0}}\right]}=\Ex^{h,H}\Big[\Phi_x\otimes\ol{\Phi_y}\Big]
\,,\]
which equals $\int_{\nu_{x,y}} \left[\hol_{h,H}(\gamma^{-1})\right]$ by \eqref{covariance-holonomy}. The result follows immediately. 
\end{proof}

Let us explain how successive specialisations of this theorem will allow us to recover Dynkin's isomorphism in its classical version. Let us consider first the case where $H$ is a scalar potential, that is, the case where for each vertex $x\in \V$, the operator $H_{x}$ has a unique eigenvalue, which we also denote by $H_{x}$. In that case, according to Lemma \ref{hol-tloc}, the twisted holonomy along a continuous path can be expressed in terms of its classical holonomy and an exponential functional of its local time, and \eqref{Dynkin} becomes
\[\Ex^h \otimes \mbox{$\int_{\nu_{x,y}}$} \left[e^{-\frac{1}{2}\sum_{x\in \V} \lambda_{x} H_{x}( \|\Phi_{x}\|^{2}+\ell_{x}(\gamma))}\hol_{h}(\gamma^{-1})\right]=\Ex^{h}\left[e^{-\frac{1}{2}\sum_{x\in \V} \lambda_{x} H_{x} \|\Phi_{x}\|^{2}} \Phi_{x}\otimes \ol{\Phi_{y}}\right]. \]
Going one step further in the specialisation, let us apply this formula to the case of the trivial real bundle of rank $1$ endowed with the trivial connection. In that case, $\Phi$ is a random real valued function on $\V\setminus\W$ and we find one of the classical formulations of Dynkin's isomorphism, namely
\[\Ex \otimes \mbox{$\int_{\nu_{x,y}}$}\left[e^{-\frac{1}{2}\sum_{x\in \V} \lambda_{x} H_{x}(|\Phi_{x}|^{2}+\ell_{x}(\gamma))}\right]=\Ex\left[e^{-\frac{1}{2}\sum_{x\in \V} \lambda_{x} H_{x} |\Phi_{x}|^{2}} \Phi_{x} \Phi_{y}\right].\]

It appears that Dynkin's theorem is one of these deep results that, with the benefit of a few decades of hindsight and the appropriate set of preliminary definitions, become almost tautological. Nevertheless, from our point of view, it remains definitely not trivial in that it expresses the covariance of the Gaussian free vector field as a weighted sum of holonomies along paths in our graph. For example, if we pick a vertex $x$, then the covariance matrix of the random vector~$\Phi_{x}$ in the vector space $\F_{x}$ is given by
\[\Ex^{h,H}\Big[\Phi_{x}\otimes \ol{\Phi_{x}}\Big]=\int_{\CL(\G)} \hol_{h,H}(\gamma^{-1}) \; \d\nu_{x,x}(\gamma) \in \End(\F_{x}).\]
Let us make a few comments on this expression. Firstly, the left-hand side makes it clear that the endomorphism of $\F_{x}$ which both sides of this equality express is Hermitian and non-negative. If $H=0$, the Hermitian character of the right-hand side can be easily recognised from the invariance of the measure $\nu_{x,x}$ by inversion, by writing
\[\int_{\CL(\G)} \hol_{h}(\gamma^{-1}) \; \d\nu_{x,x}(\gamma)=\int_{\CL(\G)} \frac{1}{2}\left(\hol_{h}(\gamma)+ \hol_{h}(\gamma)^{-1}\right) \; \d\nu_{x,x}(\gamma).\]

Secondly, let us compute an example of this formula. Let $\G$ be the graph with two vertices: $x$ and $w$, and two (geometric) edges: one looping edge $e$ based at $x$, and one edge joining $x$ to $w$. Of course, the well of this graph is the singleton $\W=\{w\}$. 
\begin{figure}[h!]
\begin{center}
\includegraphics{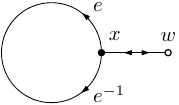}
\caption{\small A very simple graph.}
\end{center}
\end{figure}

Let $\chi_{e}$ denote the conductance of $e$, and $\kappa_{x}$ the conductance of the edge that joins $x$ to $w$. A computation similar to the one that we did in the proof of Proposition \ref{prop:shearing} yields
\[\int_{\CL(\G)} \hol_{h}(\gamma^{-1}) \; \d\nu_{x,x}(\gamma)=(\chi_{e}(2\Id_{\F_{x}}-(h_{e}+h_{e}^{-1}))+\kappa_{x} \Id_{\F_{x}})^{-1}.\]
As we explained there, this operator can be made, by appropriate choices of $\chi_{e}$, $\kappa_{x}$ and $h_{e}$, equal to any positive Hermitian operator on $\F_{x}$. This shows that the presence of the connection can make the Gaussian free vector field as anisotropic as one wishes at a given vertex, and that this anisotropy is an effect of the non-triviality of the holonomy of the connection. 

\subsection{Eisenbaum's isomorphism} We will now state a generalisation of Eisenbaum's isomorphism. Our understanding is that this theorem gives a probabilistic expression in terms of the Gaussian free field of the difference between the inverses of the Laplacians associated to a same connection but to different potentials (see \eqref{diffinvlap} below).

\begin{theorem}\label{thm:Eisenbaum} Let $\G$ be a weighted graph with a well. Let $\F$ be a Hermitian or Euclidean vector bundle over $\G$. Let $h$ be a connection on $\F$ and $H$ a potential on $\F$. Let $f$ be a section of $\F$. Then the following identity holds in $\Omega^{0}(\V\setminus\W,\F)$:
\begin{equation}\label{Eisenbaum}
\Ex^{h} \otimes \mbox{$\int_{\nu}$}\left[e^{-\frac{1}{2}\left(\Phi+f, H (\Phi+f)\right)_{\Omega^{0}}}\hol_{h,H}(\gamma^{-1})\right]\Delta_{h}f=\Ex^{h}\left[e^{-\frac{1}{2}\left(\Phi+f, H (\Phi+f)\right)_{\Omega^{0}}}(\Phi+f)\right].
\end{equation}
More concretely, for all vertex $x$, we have
\begin{equation}\label{Eisenbaum2}
\sum_{y\in \V\setminus \W}\lambda_{y} \int_{\nu_{x,y}} \hol_{h,H}(\gamma^{-1})\big((\Delta_{h}f)(y)\big) \; \d\nu_{x,y}(\gamma)=\frac{\Ex^{h}\left[e^{-\frac{1}{2}\left(\Phi+f, H (\Phi+f)\right)_{\Omega^{0}}}(\Phi_{x}+f_{x})\right]}{\Ex^{h}\left[e^{-\frac{1}{2}\left(\Phi+f, H (\Phi+f)\right)_{\Omega^{0}}}\right]}.
\end{equation}
\end{theorem}

Before we prove the theorem, let us make a few comments. The first remark is that, although the left-hand side of \eqref{Eisenbaum2} depends in a manifestly linear way on $f$, it is not clear that the right-hand side does. This is however the case, as will be clear from the proof. This feature is shared by Eisenbaum's original statement, which we shall deduce from our theorem later.

A second remark is that one of the prominent features of Eisenbaum's theorem, namely the use of unconditioned measures on the trajectories of the Markov process, seems to have disappeared in our statement. We shall explain below how to fill this apparent gap between our formulation and Eisenbaum's original statement. 

A third remark is that, in the right-hand side of \eqref{Eisenbaum2}, $f$ can be taken out of the expectation, thus producing, thanks to \eqref{masterDE}, a statement of the form 
\[\Delta_{h,H}^{-1}\Delta_{h}f=f+\ldots\]
where the dots hide a non-obviously linear functional of $f$. Applying this formula not to $f$ but to $\Delta_{h}^{-1}f$, we find
\begin{equation}\label{diffinvlap}
\Delta^{-1}_{h,H}f=\Delta^{-1}_{h}f+\frac{\Ex^{h}\big[e^{-\frac{1}{2}(\Phi+\Delta^{-1}_{h}f,H(\Phi+\Delta^{-1}_{h}f))_{\Omega^{0}}}\Phi\big]}{\Ex^{h}\big[e^{-\frac{1}{2}(\Phi+\Delta^{-1}_{h}f,H(\Phi+\Delta_{h}^{-1}f))_{\Omega^{0}}}\big]},
\end{equation}
the announced probabilistic expression of the difference between the inverses of two Laplacians.

Let us now turn to the proof of the theorem. 

\begin{proof} Using successively \eqref{masterDE} and Proposition \ref{covariancebrute}, we find
\[\bigg(\int_{\CP(\G)}\hol_{h,H}(\gamma^{-1}) \; d\nu(\gamma)\Delta_{h}f\bigg)(x)=\Delta_{h,H}^{-1}(\Delta_{h}f)(x)=\Ex^{h,H}\left[(\Phi, \Delta_{h}f)_{\Omega^{0}}\Phi_{x}\right].\]
We now use the following Gaussian lemma. Let $(X,Y)$ be an $(r+1)$-dimensional Gaussian vector such that $X$ is $r$-dimensional and $Y$ is $1$-dimensional. Assume that the vector $(X,Y)$ is symmetric, that is, centred in the real case, and, in the complex case, such that $(e^{i\theta}X,e^{i\theta}Y)$ has the same distribution as $(X,Y)$ for every real $\theta$. Then 
\[\Ex[\ol{Y}X]=\frac{\Ex[e^{\Re Y}X]}{\Ex[e^{\Re Y}]}.\]
We apply this lemma $X=\Phi_{x}$ and $Y=(\Delta_{h}f,\Phi)_{\Omega^{0}}$. The Gaussian vector $(X,Y)$ is symmetric because $\Phi$ and $-\Phi$ have the same distribution and, in the complex case, the same distribution as $e^{i\theta}\Phi$ for all real $\theta$. We find
\[\bigg(\int_{\CP(\G)}\hol_{h,H}(\gamma^{-1}) \; \d\nu(\gamma)\Delta_{h}f\bigg)(x)=\frac{\Ex^{h,H}\left[e^{ \Re ( \Phi, \Delta_{h}f)_{\Omega^{0}}}\Phi_{x} \right]}{\Ex^{h,H}\left[e^{ \Re( \Phi, \Delta_{h}f)_{\Omega^{0}}}\right]}.\]
Replacing, by an affine change of variables, $\Phi$ by $\Phi+f$ in the numerator and the denominator, we find that this quotient is equal to  
\[\frac{\Ex^{h,H}\left[e^{- \Re(\Phi, Hf)_{\Omega^{0}}}(\Phi_{x}+f_{x}) \right]}{\Ex^{h,H}\left[e^{- \Re( \Phi, Hf)_{\Omega^{0}}}\right]}=\frac{\Ex^{h}[e^{-\frac{1}{2}(\Phi+f,H(\Phi+f))_{\Omega^{0}}}(\Phi_{x}+f_{x}) ]}{\Ex^{h}[e^{-\frac{1}{2}(\Phi+f,H(\Phi+f))_{\Omega^{0}}}]},\]
as expected.
\end{proof}

The next lemma will help us to bridge the gap between Theorem \ref{thm:Eisenbaum} and Eisenbaum's original isomorphism. In order to state it, recall from \eqref{eq:deftaugamma} the following notation: for every continuous path $\gamma$, 
\[\tau(\gamma)=\inf\{t\geq 0 : \gamma_{t}\in \W\}\]
is the hitting time of the well by $\gamma$. 

\begin{lemma}\label{T pas T} Let $x$ be a proper vertex of $\G$. Let $F$ be a non-negative measurable function on~$\CP(\G)$. Then
\[\sum_{y\in \V\setminus\W}\kappa_{y} \int_{\CP(\G)} F(\gamma) \; \d\nu_{x,y}(\gamma)=\int_{\CP(\G)}F(\gamma_{|[0,\tau(\gamma))}) \; \d\P_{x}(\gamma).\]
\end{lemma}

\begin{proof} Let us compute the left-hand side. It is equal to
\begin{align*}
&\sum_{y\in \V\setminus\W}\frac{\kappa_{y}}{\lambda_{y}} \int_{0}^{\infty}e^{-t} \sum_{n=0}^{\infty} \sum_{(x,e_{1},x_{1},\ldots,x_{n-1},e_{n},y)\in \DP_{n}(\G)} P_{x,e_{1}}P_{x_{1},e_{2}}\ldots P_{x_{n-1},e_{n}}\\
&\hspace{4cm} \int_{0<t_{1}<\ldots<t_{n}<t} F((x,t_{1}),e_{1},(x_{1},t_{2}-t_{1}),\ldots,(y,t-t_{n}))\; \d t_{1}\ldots \d t_{n},
\end{align*}
that is, to
\[\sum_{\substack{(x,e_{1},\ldots,y)\in \DP_{n}(\G)\\e\in \E : \ul{e}=y,\ol{e}\in \W}} \Q_{x}((x,e_{1},\ldots,y,e,\ol{e}))\int_{(0,+\infty)^{n+1}}F((x,\tau_{0}),e_{1},\ldots,(y,\tau_{n}))e^{-\tau_{0}-\ldots-\tau_{n}}\; \d\tau_{0}\ldots \d\tau_{n},
\]
in which we recognise the right-hand side.
\end{proof}

Using this lemma, we will now state a corollary of Theorem \ref{thm:Eisenbaum} from which it will be easy to deduce the classical statement.

Let us introduce the operator $K$ on $\Omega^{0}(\V\setminus \W,\F)$ defined by
\[(K f)(x)=\kappa_{x}f(x).\]

\begin{corollary}\label{cor:Eisenbaum} Let $\F$ be a Hermitian vector bundle over a weighted graph with a well $\G$. Let $h$ be a connection on $\F$ and $H$ a potential on $\F$. Let $x$ be a vertex of $\G$. Let $b$ be a section of $\F$ over $\partial(\V\setminus \W)$. Define a global section $f$ of $\F$ over $\V$ by setting
\[f=G_{h}Kb.\]
Then the following identity holds in $\F_x$:
\begin{equation}\label{Eisenbaum3}
\Ex^{h}\otimes \Ex_{x}\left[e^{-\frac{1}{2}\left(\Phi+f,H(\Phi+f)\right)_{\Omega^{0}}} \hol_{h,H}(\gamma_{|[0,\tau(\gamma))}^{-1})(b) \right]=\Ex^{h}\left[e^{-\frac{1}{2}\left(\Phi+f,H(\Phi+f)\right)_{\Omega^{0}}}(\Phi_{x}+f_{x})\right].
\end{equation}
\end{corollary}

Note that the path $\gamma_{[0,\tau(\gamma))}^{-1}$ does not start from a vertex of the well, but instead from the last proper vertex visited by $\gamma$ before hitting the well. This last visited proper vertex belongs to the set $\partial(\V\setminus \W)$, which is precisely the set where $b$ is defined.

\begin{proof} According to Lemma \ref{T pas T} and \eqref{covariance-holonomy},
\[\Ex_{x}[ \hol_{h,H}(\gamma_{|[0,\tau(\gamma))})^{-1})(b)]=(G_{h,H} K b)(x)=\Delta_{h,H}^{-1}\Delta_{h}f,\]
which, given \eqref{masterDE}, is computed by Theorem \ref{thm:Eisenbaum}.
\end{proof}

Just as in the case of Dynkin's isomorphism, this result can be specialised to the original version of Eisenbaum's isomorphism. Assume that $H$ is scalar in each fibre $\F_{x}$, equal to $H_{x}\Id_{\F_{x}}$. In that case, observing that the full trajectory of the random walk and the trajectory stopped at the hitting time of the well have the same local time at every vertex of $\V$, \eqref{Eisenbaum3} becomes
\begin{align*}
&\Ex^{h}\otimes\Ex_{x}\left[e^{-\frac{1}{2}\sum_{x\in \V\setminus\W} \lambda_{x} H_{x}( \|\Phi_{x}+f_{x}\|^{2}+\ell_{x}(\gamma))}\hol_{h}(\gamma_{|[0,\tau(\gamma))}^{-1})(b)\right]\\
&\hspace{7cm}=\Ex^{h}\left[e^{-\frac{1}{2}\sum_{x\in \V\setminus\W} \lambda_{x} H_{x} \|\Phi_{x}+f_{x}\|^{2}} (\Phi_{x}+f_{x})\right].
\end{align*}

Assume now that the fibre bundle is the trivial real bundle of rank $1$ and the connection is the trivial connection. Let $b$ be the section, that is, the function, which is identically equal to a real number $s$ on $\partial (\V\setminus \W)$. In that case, $f=GKb=sG\kappa=s\1_{\V}$, the function identically equal to $s$ on $\V\setminus \W$. Then the formula above becomes
\[\Ex\otimes\Ex_{x}\left[e^{-\frac{1}{2}\sum_{x\in \V\setminus\W} \lambda_{x}H_{x}((\Phi_{x}+s)^{2}+\ell_{x}(X))}\right]s=\Ex\left[e^{-\frac{1}{2}\sum_{x\in \V\setminus\W} \lambda_{x}H_{x}  (\Phi_{x}+s)^{2}} (\Phi_{x}+s)\right].\]
Provided $s$ is not zero, dividing by $s$ yields one of the classical formulations of Eisenbaum's isomorphism.
\medskip

Let us finally explain how the results of this section can be understood in terms of the resolution of a Dirichlet problem in $\G$. 

Recall from the beginning of Section \ref{sec:Laplacians} that there exists an uncompressed Laplacian $\ol{\Delta}_{h}=d^{*}\circ d$ acting on the space $\Omega^{0}(\V,\F)$ of all sections of $\F$, even those that do not vanish on the well. Adding a potential $H$ to this Laplacian, we obtain the operator $\ol{\Delta}_{h,H}=\ol{\Delta}_{h}+H$.

Suppose that we are given a section $w\in \Omega^{0}(\W,\F)$ of $\F$ over the well and we want to solve the Dirichlet problem with this boundary condition, that is, to find a section $f\in \Omega^{0}(\V,\F)$ of $\F$ such that
\begin{equation}\label{probDiri}
\left\{\begin{array}{rl} \ol{\Delta}_{h,H} f=0 & \mbox{on } \V\setminus \W,\\
f=w & \mbox{on } \W.
\end{array}\right.
\end{equation}
Assuming that $H$ vanishes on the well as usual, the Laplacians $\Delta_{h,H}$ and $\ol{\Delta}_{h,H}$ agree everywhere but on the rim and for every $x\in \partial (\V\setminus \W)$, we have 
\[\ol{\Delta}_{h,H} f(x)=\Delta_{h,H}f(x)-\sum_{\ul{e}=x, \ol{e}\in \W} h_{e}^{-1} f(\ol{e}).\]
Thus, $f$ solves the Dirichlet problem \eqref{probDiri} if and only if it is satisfies
\begin{equation}\label{probDiri2}
\Delta_{h,H}f(x)=\left\{\begin{array}{ll} 0 & \mbox{if } x\in (\V\setminus \W)\setminus \partial(\V\setminus \W),\\
b & \mbox{if } x\in \partial(\V\setminus \W),
\end{array}\right.
\end{equation}
where $b$ is the section of $\F$ over $\partial \V$ defined by 
\[b(x)=\sum_{\ul{e}=x, \ol{e}\in \W} h_{e}^{-1} w(\ol{e}).\]
It appears that solving the Dirichlet problem for the uncompressed Laplacian in the graph with a boundary condition on the well amounts to applying the inverse of the compressed Laplacian to a certain section of $\F$ over the rim. 
 
It is a classical fact that the solution of the Dirichlet problem in a domain with a certain boundary condition is the average of the Gaussian free field on this domain conditioned to satisfy this boundary condition. The relation \eqref{diffinvlap} extends this result in the following sense: provided one knows how to solve the Dirichlet problem associated with the Laplacian $\Delta_{h}$ for a certain boundary condition, this corollary gives an expression in terms of the Gaussian free field of the solution of the Dirichlet problem with the same boundary condition but for any of the generalised Laplacians $\Delta_{h,H}$.

\section{Isomorphism theorems of Le Jan and Sznitman}\label{sec:LJS}

In this section, we continue our investigation of the way in which classical isomorphism theorems can be extended to the covariant setting and we turn to Le Jan's and Sznitman's isomorphism theorems, which relate the magnitude of the Gaussian free vector field to local times of Poissonian ensembles of loops and paths.

Just as our approach to the generalisation of the isomorphism theorems of Dynkin and Eisenbaum was based on one main equality, namely \eqref{masterDE}, the results of this section will ultimately be based on \eqref{masterLJS}. In particular, the measure $\mu$ in its various forms (see Definition \ref{defmu}) will play a prominent role in this study.

However, \eqref{masterLJS} turns out to be more difficult to use than \eqref{masterDE} and we use instead the traced version of \eqref{masterLJS}, namely \eqref{masterLJStr}, which we recall here for the convenience of the reader:
\begin{equation}\label{masterLJStr2}
\int_{\CL(\G)}\left(\Tr\,\hol_{h,H}(\gamma^{-1})-\Tr\,\hol_{h}(\gamma^{-1})\right)\d\mupl(\gamma)=\log\frac{\det\Delta_{h}}{\det\Delta_{h,H}}\,. 
\end{equation}

Using the definition of the measure $\mupl$, this equality can be written

\begin{equation}\label{masterLJStr3}
\sum_{x\in \V}\int_{0}^{\infty}\! \int_{\CL(\G)}\Big(\Tr_{\F_{x}}\big(\hol_{h,H}(\gamma_{|[0,t)}^{-1})\big)-\Tr_{\F_{x}}\big(\hol_{h}(\gamma_{|[0,t)}^{-1})\big)\Big)\1_{\{\gamma_{t}=x\}}\; \d\P_{x}(\gamma) \frac{\d t}{t}=\log \frac{\det \Delta_{h}}{\det \Delta_{h,H}}.
\end{equation}

\subsection{Overview of the approach}

Let us explain, in the Euclidean case, and with Le Jan's theorem in mind, how we are going to use \eqref{masterLJStr2}. Firstly, we are going to exponentiate it and recognise, in the right-hand side, the expectation $\Ex^{h}\big[e^{-\frac{1}{2}(\Phi, H\Phi)_{\Omega^{0}}}\big]$. The problem is to interpret the left-hand side
\begin{equation}\label{f-g}
\exp \int_{\CL(\G)}\left(\Tr\,\hol_{h,H}(\gamma^{-1})-\Tr\,\hol_{h}(\gamma^{-1})\right)\d\mupl(\gamma).
\end{equation}
A probabilistic interpretation of this quantity relies on Campbell's formula for Poisson point processes, which goes as follows. Given a diffuse $\sigma$-finite Borel measure $m$ on a Polish space $\mathcal X$, let us denote by 
$X$ the Poisson point process on $\mathcal X$ with intensity $m$. Then the Campbell formula asserts that for every Borel function $f$ on $\mathcal X$ such that $f\geq 1$, 
\[\Ex\Big[\prod_{x\in X} f(x)\Big]=\exp{\int_{\mathcal X}(f-1)\; \d m}.\]

In order to put \eqref{masterLJStr2} in the form of Cambpell's formula, we will proceed as follows. Firstly, we will use the fact that the measure $\mupl$ is invariant under the path reversal map $\gamma\mapsto \gamma^{-1}$ (see Lemma \ref{reversal}) and the fact that the twisted holonomy is turned into its adjoint by composition by the same map (see \eqref{inverse-adjoint}), to say that the imaginary part of the integral in \eqref{masterLJStr2} vanishes and that we can replace traces by their real parts.

The second step one would like to take is a simple algebraic manipulation leading  to the form
\begin{equation}\label{badequation}
\exp \int_{\CL(\G)}\left(\frac{\Re \Tr\,\hol_{h,H}(\gamma^{-1})}{\Re \Tr\,\hol_{h}(\gamma^{-1})}-1\right) \; \d({\Re \Tr\,\hol_{h}(\gamma^{-1})}\mupl).
\end{equation}
This form has two drawbacks. The first is that the quotient of real part of traces is difficult to interpret. With Le Jan's work \cite{LeJan-book} in mind, we would like to understand it as the exponential of some linear functional of $H$ and the local time of $\gamma$. However, the non-commutativity of the present setting makes it difficult to extract from this quotient the contribution of $H$. The second drawback is 
the fact that we are now integrating with respect to a signed measure. In certain specific situations, it so happens that the trace of the holonomy along any loop is non-negative; we will give an example of such a situation in Section \ref{sec:LJpositif}. But in general, we will 
simply write the signed measure as the difference of two positive measures, and use Campbell's formula for each of them. We will thus obtain a quotient of two instances of Campbell's formula. 

On the other hand, we offer a solution to the problem of the interpretation of the quotient of real parts of traces of holonomies. It consists in lifting the integral from the space of loops to a larger space, on which holonomies, or rather quotients of holonomies, become scalar quantities. This requires a detailed explanation, which will occupy us for the next few sections.

\subsection{Splitting of vector bundles} The main new piece of structure that we need in our study of Le Jan's and Sznitman's isomorphisms is a decomposition of each fibre of the vector bundle $\F$ as an orthogonal direct sum of linear subspaces.

\begin{definition} Let $\G$ be a graph. Let $\F$ be a vector bundle over $\G$. A \new{splitting} of $\F$, or \new{colouring} of $\F$, is a collection $\I=\{(I_{x},(\F^{i}_{x})_{i\in I_{x}}) : x\in \V\}$ in which, for each $x\in \V$, $I_{x}$ is a set and $\{\F^{i}_{x}:i\in I_{x}\}$ is a family of pairwise orthogonal linear subspaces of $\F_{x}$ such that
\begin{equation}
\F_x=\bigoplus_{i\in I_x}^{\perp} \F_x^i\,.
\end{equation}
For each $x\in \V$ and $i\in I_{x}$, we denote by $\pi_x^i:\F_{x}\to \F^{i}_{x}$ the orthogonal projection. 
\end{definition}

A special case of splitting is the trivial splitting, in which each fibre of $\F$ is simply written as being equal to itself. As trivial as it is, this splitting will be useful for us, and we will denote it by $\T=\{(\{x\},\F_{x}):x\in \V\}$.

Another special case is the case of complete splittings, in which each fibre is written as an orthogonal direct sum of lines. 
 
There is a natural partial ordering of the set of all splittings of a fibre bundle : we say that the splitting $\I$ is finer than the splitting $\I'$ if for each vertex $x$ and each $i\in I_{x}$, there exists $i'\in I'_{x}$ such that $\F^{i}_{x}\subset \F^{i'}_{x}$. The trivial splitting is the maximum of this order, and the complete splittings are its minimal elements.
 
We say that a splitting $\I$ and a potential $H$ on $\F$ are \new{adapted} to each other if for all $x\in \V$ and all $i\in I_{x}$, the space $\F^{i}_{x}$ is an eigenspace of $H_{x}$. In this case, we will denote by $H^{i}_{x}$ the unique eigenvalue of the restriction of $H_{x}$ to $\F^{i}_{x}$, so that for all $x\in \V$, we have
\[H_{x}=\sum_{i\in I_{x}}H^{i}_{x}\pi^{i}_{x}.\]

For example, to say that a potential is adapted to the trivial splitting means that it is scalar on each fibre. This is a kind of potential that we already considered twice, in order to specialise Theorems \ref{thm:Dynkin} and \ref{thm:Eisenbaum} respectively to the classical Dynkin and Eisenbaum isomorphisms.

To every potential $H$ we can associate its \new{eigensplitting}, the splitting of $\F$ obtained by writing each fibre $\F_{x}$ as the direct sum of the eigenspaces of $H_{x}$. This splitting is of course adapted to $H$ and it is, among all splittings adapted to $H$, the one that is maximal for the partial order that we just described. 

\subsection{Splitting of the Gaussian free vector field}\label{sec:splitGFF} A splitting of $\F$ allows us, among other things, to split the Gaussian free vector field on $\G$.

Let $h$ be a connection and $H$ be a potential on our vector bundle $\F$ over $\G$. Let $\Phi$ be the Gaussian free vector field on $\F$ associated to $h$ and $H$. Let $\I$ be a splitting of $\F$. For each vertex~$x$ and each $i\in I_{x}$, we define
\begin{equation}\label{defPhiix}
\Phi^{i}_{x}=\pi^{i}_{x}(\Phi_{x}).
\end{equation}
We meet here for the first time a new kind of field, namely $(\Phi^{i}_{x})_{x\in \V\setminus\W, i\in I_{x}}$, indexed not only by the vertices of $\G$, but also, at each vertex $x$, by elements of $I_{x}$, which we will call \new{colours}. Without giving too precise a meaning to this term, we will speak of a coloured field.

This coloured field $(\Phi^{i}_{x})_{x\in \V\setminus\W, i\in I_{x}}$ is a Gaussian random element of $\bigoplus_{x\in \V\setminus\W, i\in I_{x}}\F^{i}_{x}$. It is centred, and an application of \eqref{EphiG} allows us to compute its covariance: for all $x,y\in \V\setminus\W$ and all $i\in I_{x}$, $j\in I_{y}$,
\begin{equation}\label{covphicol}
\Ex^{h,H}\big[\Phi^{i}_{x} \otimes \ol{\Phi^{j}_{y}}\big]=\pi^{i}_{x}\circ (G_{h,H})_{x,y}\circ \pi^{j}_{y},
\end{equation}
an equality to be read in $\Hom(\F^{j}_{y},\F^{i}_{x})$.

The case where $\I$ is a complete splitting is of particular interest. Let us consider this case, and let us also assume that the set of colours $I_{x}$ is the same for each vertex $x\in \V$, namely $I_{x}=\{1,\ldots,r\}$. Let us finally choose a unit vector $u^{i}_{x}$ in each line $\F^{i}_{x}$. Then, for each $i\in \{1,\ldots,r\}$, we can define a scalar random field $\Phi^{i}$ on $\V$ by setting, for all $x\in \V$, and with a conflict of notation that we do not deem too serious,
\[\Phi^{i}_{x}=\langle u^{i}_{x},\Phi_{x}\rangle_{x}.\]
Each of the $r$ scalar fields $\Phi^{i},i\in\{1,\ldots,r\}$ is a centred Gaussian field and for all $i\in\{1,\ldots,r\}$, the covariance of the scalar field $\Phi^i$ is given, for all $x,y\in \V$, by
\[\Ex^{h,H}\left[\Phi^i_x\ol{\Phi^i_y}\right]=\langle u^i_x,\left(G_{h,H}\right)_{x,y} u^i_y\rangle_x\,.\]
Although this is clear from \eqref{covphicol}, we would like to stress that 
the fields $\Phi^{1},\ldots,\Phi^{r}$ are correlated, due to the presence of the connection $h$ and the potential $H$ (see Figure \ref{GFF-color}).

\begin{figure}
\centering
\begin{tabular}{c}
\includegraphics[width=8cm]{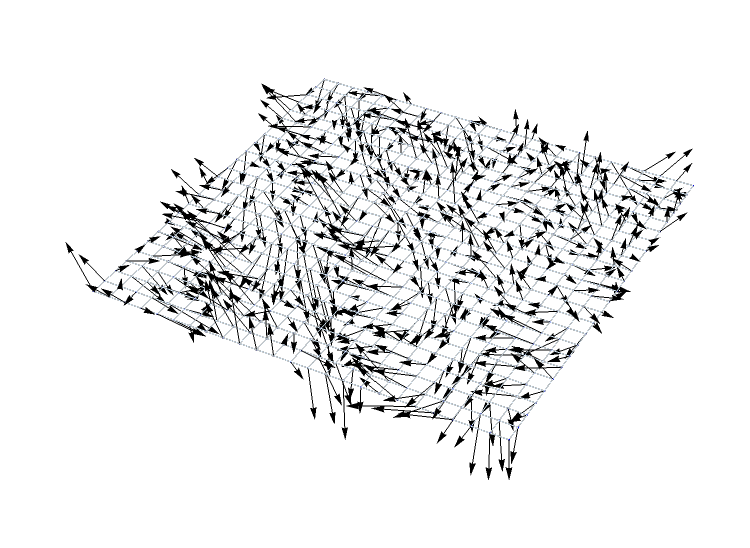}\\[-1cm]
\includegraphics[width=12cm]{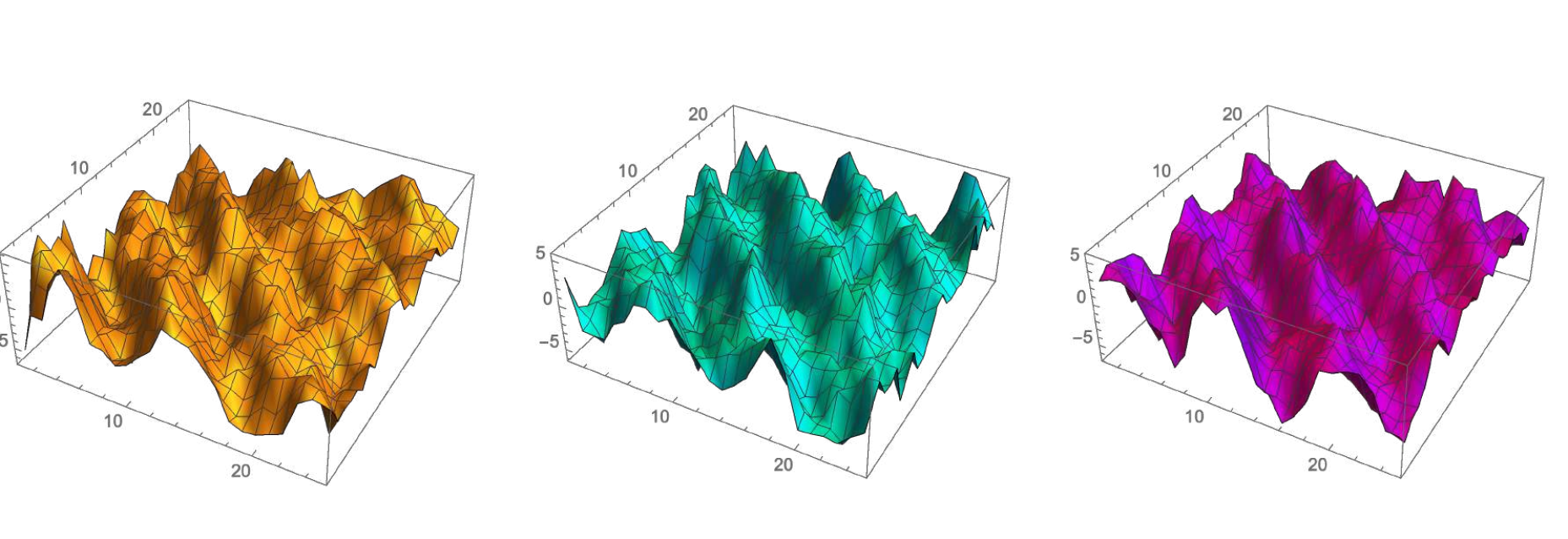}
\end{tabular}
\caption{\small A sample of a Gaussian free vector field $\Phi$ on a square grid of size $25$ endowed with a trivial real bundle of rank~$3$, along with its components $\Phi^{1},\Phi^{2},\Phi^{3}$. The connection is a fixed rotation in each coordinate axis.}\label{GFF-color}
\end{figure}

\subsection{Coloured paths, coloured loops and coloured local time} We are now going to explain how the choice of a splitting of the vector bundle allows us to define an enriched version of the space of loops on the graph. From the point of view of paths that we will adopt now, we prefer to think of each subspace of a fibre of $\F$ as a colour, and of the splitting itself as a colouring.

Let $\I$ be a colouring of $\F$, that is, a splitting of $\F$. We define the space $\CPc(\G)$ of \new{$\I$-coloured paths} on $\G$ as the set of all sequences
\[\eta=((x_{0},\tau_{0},i_{0}),e_{1},(x_{1},\tau_{1},i_{1}),\ldots,e_{n},(x_{n},\tau_{n},i_{n}))\]
such that
\[\gamma=((x_{0},\tau_{0}),e_{1},(x_{1},\tau_{1}),\ldots,e_{n},(x_{n},\tau_{n}))\]
is an element of $\CP(\G)$ and, for each $k\in \{0,\ldots,n\}$, we have $i_{k}\in I_{x_{k}}$.

The coloured path $\eta$ will be said to have the colour $i_{k}$ when it visits the vertex $x_{k}$. Note that a coloured path can have different colours at successive visits of the same vertex. 

A $\T$-coloured path, where $\T$ is the trivial colouring, is nothing but a path which, at each vertex that it visits, has the only existing colour at that vertex. Accordingly, we will identify freely the spaces $\CPv(\G)$ and $\CP(\G)$.

We say that the coloured path $\eta$ written above is a coloured loop if $x_{n}=x_{0}$ and $i_{n}=i_{0}$. We denote by $\CLc(\G)$ the set of coloured loops on $\G$.

By analogy with Definition \ref{local time}, we define, for all $x\in \V$ and all $i\in I_{x}$, the occupation measure and the local time of $\eta$ with colour $i$ at the vertex $x\in \V$ by
\[\vartheta^{i}_{x}(\eta)=\sum_{k=0}^{n} \delta_{x,x_{k}}\delta_{i,i_{k}} \tau_{k} \, \mbox{ and }\,  \ell^{i}_{x}(\eta)=\frac{1}{\lambda_{x}} \vartheta^{i}_{x}(\gamma).\]

In our naive quantum mechanical picture of paths on the graph, a coloured path describes not only the motion of a particle, but the successive states in which this particle can be found at the successive vertices that it visits. If the splitting that we are considering is the eigensplitting of a potential $H$, then these states can be measured at each vertex $x$ by the observable $H_{x}$.

Let us conclude this section by saying something about the way in which the set of $\I$-coloured paths depends on the colouring $\I$. Let us consider two colourings $\I$ and $\I'$ of $\F$ such that $\I$ is finer than $\I'$. Extending the metaphor of colours, there is a bleaching map $\b_{\I',\I}:\CPc(\G)\to \CPcp(\G)$ which, at each vertex $x$ visited by a path, changes the colour $i\in I_{x}$ into the unique colour $i'\in I'_{x}$ such that $\F^{i}_{x}\subset \F^{i'}_{x}$. If $\I''$ is a third colouring such that $\I'$ is finer than $\I''$, then the relation $\b_{\I'',\I'}\circ \b_{\I',\I}=\b_{\I'',\I}$ holds.

The special case where $\I'$ is the trivial colouring will be of particular interest, and we will use the simpler notation $[\eta]=\b_{\T,\I}(\eta)$ for the element of $\CP(\G)$ obtained by forgetting altogether the colours of an $\I$-coloured path $\eta$.

\subsection{Amplitudes} We will now use the colourings of a path to decompose the trace of the holonomy along that path, and ultimately to simplify the expression \eqref{badequation}.

Let $\eta=((x_{0},\tau_{0},i_{0}),e_{1},(x_{1},\tau_{1},i_{1}),\ldots,e_{n},(x_{n},\tau_{n},i_{n}))$ be an $\I$-coloured path. Recall that we chose a connection $h$ and a potential $H$ on the vector bundle $\F$. We define the \new{amplitude} of $h$ along the $\I$-coloured path $\eta$ twisted by $H$ as the operator
\[\amp^{\I}_{h,H}(\eta)=\pi^{i_{n}}_{x_{n}} \circ e^{-\tau_{n}H_{x_{n}}} \circ h_{e_{n}} \circ \ldots \circ \pi^{i_{1}}_{x_{1}}\circ e^{-\tau_{1}H_{x_{1}}}\circ h_{e_{1}}\circ \pi^{i_{0}}_{x_{0}}\circ e^{-\tau_{0}H_{x_{0}}} : \F^{i_{0}}_{x_{0}}\to \F^{i_{n}}_{x_{n}}.\]
This definition is designed for the following equality to hold: for every (non-coloured) path $\gamma$,  
\begin{equation}\label{amp-trace}
\sum_{\eta\in\CPc(\G): [\eta]=\gamma}\amp^{\I}_{h,H}(\eta)=\hol_{h,H}(\gamma)\,.
\end{equation}
More generally, if $\I$ is finer than $\I'$, then for all $\eta'\in \CPcp(\G)$
\begin{equation}\label{amp-trace2}
\sum_{\eta\in\CPc(\G): \b_{\I',\I}(\eta)=\eta'}\amp^{\I}_{h,H}(\eta)=\amp^{\I'}_{h,H}(\eta')\,.
\end{equation}

Note that for every $\I$-coloured path $\eta$, we have
\begin{equation}\label{etaeta}
\amp^{\I}_{h,H}(\eta^{-1})=\amp^{\I}_{h,H}(\eta)^{*},
\end{equation}
the adjoint of $\amp^{\I}_{h,H}(\eta)$. In particular, if $\eta$ is a coloured loop, then the traces  $\Tr\, \amp^{\I}_{h,H}(\eta)$ and $\Tr\,\amp^{\I}_{h,H}(\eta^{-1})$ are conjugate complex numbers. 

One of the reasons to introduce amplitudes is the following: if $H$ and $\I$ are adapted, then
\begin{equation}\label{loc-amp}
\amp^{\I}_{h,H}(\eta)=e^{-\sum_{x\in \V\setminus \W,i\in I_{x}} H^{i}_{x} \vartheta^{i}_{x}(\eta) }\amp^{\I}_{h}(\eta),
\end{equation}
a formula analogous to \eqref{sortH}, but which holds without the assumption that $H$ be scalar in each fibre. This indicates that, after lifting \eqref{badequation} from $\CL(\G)$ to $\CLc(\G)$, 
it will be much easier to deal with the quotient of holonomies, which will become a quotient of amplitudes. Before we implement this lifting procedure, we must say a word about the measures, indeed the signed measures that we will use on $\CLc(\G)$ and $\CPc(\G)$. 

\subsection{Signed measures on the sets of coloured loops and paths} In this section, we will define two families of measures on sets of coloured paths. Firstly, we will define a signed measure $\muplc_{h}$ on the set of $\I$-coloured loops, which in a sense lifts a measure on the set of loops absolutely continuous with respect $\mupl$ (see Definition \ref{defmu}), with a density depending on the connection $h$. Then, for every section $f$ of $\F$, we will define a signed measure $\nu^{\I}_{h,f}$ on the set of $\I$-coloured paths, which in a similar sense lifts and generalises the measures $\nu_{x,y}$ (see Definition \ref{defnu}).

\begin{definition} Let $h$ be a connection on $\F$. Let $\I$ be a colouring of $\F$. The measure $\muplc_{h}$ is defined on $\CLc(\G)$ by the fact that for all bounded non-negative measurable function $F$ on $\CLc(\G)$,
\[\int_{\CLc(\G)} F(\eta) \; \d\muplc_{h}(\eta)=\int_{\CL(\G)} \sum_{\eta\in\CLc(\G): [\eta]=\gamma }F(\eta)\Re\Tr(\amp^{\I}_{h}(\eta^{-1})) \; \d\mupl(\gamma).\]
\end{definition}

To make it clear that the integral above exists, we can rewrite it as
\begin{equation}\label{defmuplc}
\int_{\CL(\G)} \sum_{\eta\in\CLc(\G): [\eta]=\gamma }F(\eta)\Re\Tr(\amp^{\I}_{h}(\eta^{-1})) \; \d\mul(\gamma) +\sum_{x\in \V} \int_{0}^{+\infty} \sum_{i\in I_{x}} F((x,t,i)) e^{-t}\frac{\d t}{t},
\end{equation}
so that the measure $\muplc_{h}$ appears as the sum of an honest signed measure and a $\sigma$-finite positive measure. 

In the particular case of the trivial splitting, the formula above defines a measure $\muplv_{h}$ on $\CL(\G)$, which we will denote simply by $\mupl_{h}$ and which can be written as
\begin{equation}\label{eq:defmuplh}
\d\mupl_{h}(\gamma)=\Re\Tr(\hol_{h}(\gamma^{-1}))\; \d\mupl(\gamma).
\end{equation}

The next proposition tells us how to lift to the space of coloured loops the integrals we are interested in.
\begin{proposition}\label{liftmu} Let $F$ be a bounded non-negative function on $\CL(\G)$. Then
\begin{equation}\label{LIL}
\int_{\CL(\G)} F(\gamma) \; \d\mupl_{h}(\gamma)=\int_{\CLc(\G)}F([\eta]) \; \d\muplc_{h}(\eta).
\end{equation}
Assume that $F(\gamma)=F(\gamma^{-1})$ for every loop $\gamma$. Then
\begin{equation}\label{LILhH}
\int_{\CL(\G)}F(\gamma)\Tr(\hol_{h,H}(\gamma^{-1}))\; \d\mupl(\gamma)=\int_{\CLc(\G)}F([\eta])e^{-\sum_{x\in \V\setminus \W,i\in I_{x}} H^{i}_{x} \vartheta^{i}_{x}(\eta) }\;  \d\muplc_{h}(\eta).
\end{equation}
\end{proposition}

\begin{proof} The first assertion is a direct consequence of the definition of $\muplc_{h}$ and \eqref{amp-trace}.

For the second assertion, we start from the right-hand side, apply the definition of $\muplc_{h}$, use \eqref{loc-amp} and then \eqref{amp-trace}. This yields the real part of the left-hand side of the equality that we wish to prove. To see that this left-hand side is indeed real, we use Lemma \ref{reversal}, which asserts that the measure $\mupl$ is invariant under the involution $\gamma\mapsto \gamma^{-1}$, and \eqref{etaeta}.
\end{proof}

Note that we can use \eqref{amp-trace2} to generalise \eqref{LIL}: if $\I$ is finer than $\I'$, then for every bounded non-negative measurable function on $\CLcp(\G)$, we have
\begin{equation}\label{bii}
\int_{\CLc(\G)} F \circ \b_{\I',\I} \; \d\muplc=\int_{\CLcp(\G)} F \; \d\muplp.
\end{equation}

In the statement of our main result, we will make use of the positive and negative parts of the measure $\muplc_{h}$, which we will denote respectively by $\muplcp_{h}$ and $\muplcm_{h}$. Note that $\muplcp_{h}$ contains the second term of the right-hand side of \eqref{defmuplc}, corresponding to constant coloured loops, and is thus an infinite positive measure, whereas $\muplcm_{h}$ is a finite positive measure supported by the set of non-constant coloured loops.

Let us mention a delicate point which arises from the fact that we are using signed measures. Although \eqref{bii} expresses the compatibility of the family of measures $\muplc_{h}$ with the bleaching maps $\b_{\I',\I}$, it is not the case that the measures $\muplcp_{h}$, nor the measures $\muplcm_{h}$ satisfy the same compatibility relations: in symbols, we have in general
\[\muplcp_{h}\circ \b_{\I',\I}^{-1}\neq \muplpp \, \mbox{ and } \, \muplcm_{h}\circ \b_{\I',\I}^{-1}\neq \muplpm,\]
but
\[(\muplcp_{h}-\muplcm_{h})\circ \b_{\I',\I}^{-1}= (\muplpp - \muplpm).\]
This is because the amplitudes of all the $\I$-colourings of a path compatible with a given $\I'$-colouring of the same path need not have traces of the same sign. See Figure \ref{pm} below for a graphical explanation of this phenomenon in the case where $\I'$ is the trivial colouring. We shall come back to this point after the proof of Theorem \ref{thm:LeJanSznitman} below.

\begin{figure}[h!]
\begin{center}
\includegraphics{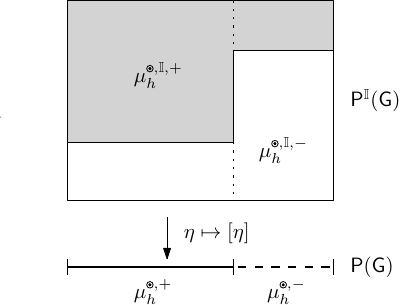}
\caption{\small \label{pm} This schematic picture shows the supports of the positive and negative parts of the measures that we introduced on $\CPc(\G)$ and $\CP(\G)$.}
\end{center}
\end{figure}

Let us turn to the definition of the second family of measures announced at the beginning of this section. It is a family of measures on the set of coloured paths that are not necessarily loops. 

\begin{definition} Let $h$ be a connection on $\F$. Let $\I$ be a colouring of $\F$.
Let $f\in \Omega^{0}(\V\setminus \W,\F)$ be a section of $\F$. The measure $\nu^{\I}_{h,f}$ is the measure on $\CPc(\G)$ such that for all non-negative measurable function $F$ on $\CPc(\G)$,
\begin{equation}\label{defnuhf}
\int_{\CPc(\G)} F(\eta)\; \d\nu^{\I}_{h,f}(\eta)=\int_{\CP(\G)}\sum_{\eta\in \CP^\I(\G):[\eta]=\gamma}\lambda_{\ul{\gamma}} F(\eta) \Re\big\langle (\Delta_h f)_{\ul{\gamma}} , \amp^{\I}_h (\eta^{-1}) (\Delta_h f)_{\ol{\gamma}} \big\rangle_{\ul{\gamma}} \; \d\nu(\gamma).
\end{equation}
\end{definition}

Since $\nu$, in contrast to $\mupl$, is a finite measure, $\nu^{\I}_{h,f}$ is a genuine signed measure, with finite positive and negative parts.

Applied to the trivial splitting, this definition yields a measure $\nu^{\T}_{h,f}$ on $\CP(\G)$ which we denote simply by
\[\d\nu_{h,f}(\gamma)= \lambda_{\ul{\gamma}} \Re\big\langle (\Delta_{h}f)(\ul{\gamma}),\hol_{h}(\gamma^{-1}) (\Delta_{h}f)(\ol{\gamma})\big\rangle_{\ul{\gamma}} \; \d\nu(\gamma).\]
Recall from Definition \ref{defnu} that the definition of $\nu$ involves the value of $\lambda$ at the final point of the path, so that the equation above is more symmetric than it looks, and could be written as
\[\d\nu_{h,f}(\gamma)= \sum_{x,y\in \V\setminus \W}\lambda_{x}\lambda_{y} \Re\big\langle (\Delta_{h}f)(x),\hol_{h}(\gamma^{-1}) (\Delta_{h}f)(y)\big\rangle_{x} \; \d\nu_{x,y}(\gamma).\]

We have for the measures $\nu^{\I}_{h,f}$ the following formulas, analogous to those of Proposition \ref{liftmu}.

\begin{proposition}\label{liftnu} Let $F$ be a non-negative measurable function on $\CPc(\G)$. Then
\begin{equation}\label{PIP2}
\int_{\CP(\G)} F(\gamma)\; \d\nu_{h,f}(\gamma)=\int_{\CPc(\G)}F([\eta])\;  \d\nu^{\I}_{h,f}(\eta)
\end{equation}
and
\begin{align}\label{PIPhH}
&\int_{\CP(\G)}\!\!\lambda_{\ul{\gamma}}F(\gamma)\big\langle (\Delta_{h}f)(\ul{\gamma}),\hol_{h,H}(\gamma^{-1}) (\Delta_{h}f)(\ol{\gamma})\big\rangle_{\ul{\gamma}}\; \d\nu(\gamma)\\
\nonumber
&\hspace{6cm}=\int_{\CPc(\G)}\!\! F([\eta])e^{-\sum_{x\in \V\setminus\W,i\in I_{x}} H^{i}_{x} \vartheta^{i}_{x}(\eta) }\;  \d\nu^{\I}_{h,f}(\eta).
\end{align}
\end{proposition}

\begin{proof} The proof uses almost exactly the same arguments as the proof of Proposition \ref{liftmu}. The only difference is the fact that the measure $\nu$ is not invariant under the map $\gamma\mapsto \gamma^{-1}$, but the measure $\lambda_{\ul{\gamma}}d\nu(\gamma)$, that is, the measure $\sum_{x,y\in \V\setminus\W}\lambda_{x}\lambda_{y}\nu_{x,y}$, is invariant.
\end{proof}

In the next section, we will make use of the positive and negative parts of this measure, which we will denote respectively by $\nu^{\I,+}_{h,f}$ and $\nu^{\I,-}_{h,f}$. A simple yet useful observation is that if $f$ is the zero section, then these measures are the null measures.

The measures $\nu^{\I}_{h,f}$, where $h$ and $f$ are fixed and $\I$ varies in the set of splittings of the vector bundle $\F$, satisfy the exact same compatibility relations with respect to the bleaching maps as the measures $\muplc_{h}$, and the same precautions are in order when one considers the behaviour of the positive and negative parts of these measures with respect to the bleaching maps. 

\subsection{Le Jan and Sznitman's isomorphisms}

Let us introduce a last piece of notation. Let $\I$ be a colouring of the fibre bundle $\F$. Let $\mathcal P$ be a set of $\I$-coloured paths. We denote by $\ell^{\I}(\mathcal P)$ the collection 
\[\ell^{\I}(\mathcal P)=(\ell^{i}_{x}(\mathcal P))_{x\in \V\setminus \W,i\in I_{x}},\] where for all $x\in \V\setminus \W$ and all $i\in I_{x}$,
\[\ell^{i}_{x}(\mathcal P)=\sum_{\eta\in \mathcal P}\ell^{i}_{x}(\eta).\]
Let now $f$ be a section of $\F$. We denote by $\|\!\cdot\!\|^{2,\I}(f)$ the collection 
\[\|\!\cdot\!\|^{2,\I}(f)=(\|\pi^{i}_{x}(f)\|^{2}_{x})_{x\in \V\setminus \W,i\in I_{x}}.\]
In the informal terminology that we introduced in Section \ref{sec:splitGFF}, both $\ell^{\I}(\mathcal P)$ and $\|\!\cdot\!\|^{2,\I}(f)$ are scalar coloured fields. The following theorem, which generalises the isomorphism theorems of Le Jan and Sznitman, states the equality in distribution of random scalar coloured fields of this nature. 

\begin{theorem}\label{thm:LeJanSznitman} Let $\G$ be a weighted graph with a well. Let $\F$ be a vector bundle over $\G$. Let~$h$ be a connection on $\F$. Let $\Phi$ be the Gaussian free vector field on $\G$ associated with~$h$. Let~$\I$ be a colouring of $\F$. Let $f\in\Omega^0(\V\setminus \W,\F)$ be a section of $\F$. 
Let $\L_+$ and $\L_-$ be Poissonian ensembles in $\CL^\I(\G)$ with respective intensities $\frac{\beta}{2}\muplcp_{h}$ and $\frac{\beta}{2}\muplcm_{h}$. 
Let $\cE_+$ and $\cE_-$ be  Poissonian ensembles in $\CPc(\G)$ with respective intensities $\frac{\beta}{2}\nu^{\I,+}_{h,f}$ and $\frac{\beta}{2}\nu^{\I,-}_{h,f}$. 
Assume that $\Phi, \L_{+},\L_{-},\cE_{+},\cE_{-}$ are independent.
Then the following equality holds in distribution:
\begin{equation}\label{LeJanSznitman}
\ell^{\I}(\L_+ \cup \cE_+)\build{=}_{}^{(d)}\frac{1}{2}\|\!\cdot\!\|^{2,\I}(\Phi+f)+\ell^{\I}(\L_- \cup \cE_-).
\end{equation}

Moreover, 
\begin{equation}\label{LeJanSznitman-2}
\frac{1}{2}\|\!\cdot\!\|^{2,\I}(\Phi)+\ell^{\I}(\cE_+)\build{=}_{}^{(d)}\frac{1}{2}\|\!\cdot\!\|^{2,\I}(\Phi+f)+\ell^{\I}(\cE_-).
\end{equation}

\end{theorem}

To be clear, the conclusion of~\eqref{LeJanSznitman} in the theorem is that the two random vectors
\[\{\ell^{i}_{x}(\L_+ \cup \cE_+) : x\in \V\setminus \W, i \in I_{x}\} \mbox{ and } \left\{\frac{1}{2}\| \pi^{i}_{x}(\Phi_{x}+f_{x})\|^{2}_{x}+\ell^{i}_{x}(\L_- \cup \cE_-) : x\in \V\setminus \W, i \in I_{x}\right\}\]
have the same distribution. Likewise, the meaning of~\eqref{LeJanSznitman-2} is that the two random vectors
\[\left\{\frac{1}{2}\| \pi^{i}_{x}(\Phi_{x})\|^{2}_{x}+\ell^{i}_{x}(\cE_+) : x\in \V\setminus \W, i \in I_{x}\right\} \mbox{and} \left\{\frac{1}{2}\| \pi^{i}_{x}(\Phi_{x}+f_{x})\|^{2}_{x}+\ell^{i}_{x}(\cE_-) : x\in \V\setminus \W, i \in I_{x}\right\}\]
have the same distribution. 

\begin{proof} We start by proving the theorem in the case where $f=0$. In this case, the ensembles $\cE_{+}$ and $\cE_{-}$ are almost surely empty. 

Let us choose a potential $H$ adapted to $\I$ and such that $H_{x}$ is non-negative for all $x\in \V$. We start from \eqref{masterLJStr}, which we multiply by $\frac{\beta}{2}$ and of which we exponentiate both sides. On the right-hand side we find
\[\left(\frac{\det \Delta_{h}}{\det \Delta_{h,H}}\right)^{\frac{\beta}{2}}=\Ex^{h}\left[e^{-\frac{1}{2} (\Phi,H\Phi)_{\Omega^{0}}}\right].\]
On the left-hand side, we find, thanks to \eqref{LILhH} applied once to the zero potential and once to the potential $H$,
\[\exp \frac{\beta}{2}\int_{\CLc(\G)} \left(e^{-\sum_{x\in \V\setminus\W,i\in I_{x}} H^{i}_{x} \vartheta^{i}_{x}(\eta) }-1\right) \; \d\muplc_{h}(\eta).\]
Splitting the measure $\muplc_{h}$ into its positive and negative parts and writing  Campbell's formula for each part, we find that
\begin{equation}\label{eqn:lejan-colore}
\Ex\left[e^{-\sum_{x\in \V\setminus\W, i\in I_{x}} \lambda_{x} H^{i}_{x} \ell^{i}_{x}(\L_{+})}\right] = \Ex^{h}\left[e^{-\frac{1}{2} \sum_{x\in \V\setminus\W,i \in I_{x}} \lambda_{x} H^{i}_{x}\|\pi^{i}_{x}(\Phi_{x})\|^{2}_{x}}\right]\Ex\left[e^{-\sum_{x\in \V\setminus\W, i\in I_{x}} \lambda_{x} H^{i}_{x} \ell^{i}_{x}(\L_{-})}\right],
\end{equation}
where the first expectation (resp. the last) is taken with respect to the distribution of the Poisson point process with intensity $\frac{\beta}{2}\muplcp_{h}$ (resp. $\frac{\beta}{2}\muplcm_{h}$).
Since this equality holds for arbitrary non-negative values of the numbers $H^{i}_{x}$, it says exactly that the Laplace transforms of both sides of \eqref{LeJanSznitman} are equal.

We now turn to the proof of the general case. We proceed again to show the equality of the Laplace transforms of both sides of \eqref{LeJanSznitman}. Combining the  special case of the theorem that we just proved, where $f=0$,  and Proposition \ref{useful-gaussian-proposition}, which was precisely designed to be used here, we conclude that it suffices to prove the equality
\[\Ex\left[e^{-\sum_{x\in\V\setminus \W,i\in I_x}\lambda_{x} H^{i}_{x} \ell^{i}_{x}(\cE_+)}\right]=e^{ \frac{1}{2}\left(\Delta_h f, (\Delta_{h,H}^{-1}-\Delta_h^{-1})\Delta_h f\right)_{\Omega^0}} \Ex\left[e^{-\sum_{x\in\V\setminus \W,i\in I_x}\lambda_{x}H^{i}_{x} \ell^{i}_{x}(\cE_-)}\right]\,.\]
Following backwards the same arguments that we used in the first part of this proof, we see that we must prove that
\begin{equation}\label{toprove}
\int_{\CPc(\G)} \left(e^{-\sum_{x\in \V\setminus \W,i\in I_{x}} H^{i}_{x} \vartheta^{i}_{x}(\eta) }-1\right) \; \d\nu^{\I}_{h,f}(\eta)=\big(\Delta_{h} f, (\Delta_{h,H}^{-1}-\Delta_{h}^{-1})\Delta_{h} f\big)_{\Omega^0}.
\end{equation}
By \eqref{masterDE} and Definition \ref{defnu}, the right-hand side is equal to
\begin{equation}\label{intermed}
\sum_{x,y\in \V\setminus \W} \lambda_{x}\lambda_{y} \int_{\CP(\G)} \big\langle (\Delta_{h}f)(x),(\hol_{h,H}(\gamma^{-1})-\hol_{h}(\gamma^{-1}))(\Delta_{h}f)(y)\big\rangle_{x} \; \d\nu_{x,y}(\gamma).
\end{equation}
By \eqref{PIPhH}, this is equal to the left-hand side of \eqref{toprove}, as expected.

Equation~\eqref{LeJanSznitman-2} follows from the same proof by using the result for $f=0$ to replace local times of $\L_+$ and $\L_-$ appropriately by the squared norms of Gaussian free vector field parts.
\end{proof}

Just as for the theorems of Dynkin and Eisenbaum, the classical theorems of Le Jan and Sznitman correspond to the special case of our theorem where $\F=\R_{\G}$ is the trivial real vector bundle of rank $1$ endowed with the trivial connection. In that case, all measures are positive, and the ensembles $\L_{-}$ and $\cE_{-}$ are almost surely empty. Let us write $\L=\L_{+}$ and $\cE=\cE_{+}$.

When $f=0$, we recover Le Jan's theorem, which asserts, in our notation, that $\ell(\L)$ has the same distribution as $\frac{1}{2}\Phi^{2}$. When $f$ is constant equal to a real $\sqrt{2s}$, we have $\Delta f=\sqrt{2s}\frac{\kappa}{\lambda}$, and the measure $\frac{\beta}{2}\nu^{\I}_{h,f}$ becomes simply
\[s\sum_{x,y\in \V\setminus \W} \kappa_{x}\kappa_{y}\nu_{x,y}.\]
Then we recover Sznitman's theorem which asserts that $\ell(\L\cup \cE)$ has the same distribution as $\frac{1}{2}(\Phi+\sqrt{2s})^{2}$.

It is tempting to compare the conclusions of Theorem \ref{thm:LeJanSznitman} for two different colourings $\I$ and~$\I'$ such that $\I$ is finer than $\I'$. It seems that of the two statements associated with two such colourings, no one is logically stronger than the other.

Let us discuss without proof the case where $\I$ is a complete splitting such that $I_{x}=\{1,\ldots,r\}$ for every vertex $x$, and $\I'$ the trivial splitting. Let us also assume, for the sake of simplicity, that $f=0$. Then, using the equality $\|\Phi^{1}_{x}\|^{2}_{x}+\ldots+\|\Phi^{r}_{x}\|^{2}_{x}=\|\Phi_{x}\|^{2}_{x}$ for every $x$, one can prove the equality in distribution
\[\left\{ \sum_{i=1}^{r}\ell^{i}_{x}(\L^{\I}_{+}) + \ell_{x}(\L_{-}):x\in \V\setminus \W\right\}\build{=}_{}^{(d)}\left\{ \sum_{i=1}^{r}\ell^{i}_{x}(\L^{\I}_{-}) + \ell_{x}(\L_{+}):x\in \V\setminus \W\right\},\]
where the ensembles corresponding to $\I$ carry a superscript $\I$, and the ensembles corresponding to $\I'$ carry no superscript. The common distribution of these random vectors can moreover be described as the distribution of the local time of a Poissonian ensemble of non-coloured paths with intensity equal to the image by the bleaching map $\eta\mapsto [\eta]$ of the measure $\1_{C^{\I}}\left|\muplc_{h}\right|$, where
\[C^{\I}=\left\{\eta \in \CLc(\G) : {\rm sgn}(\Re \Tr (\amp_{h}(\eta^{-1})))={\rm sgn}(\Re\Tr (\hol_{h}([\eta]^{-1})))\right\}.\]
In Figure \ref{pm}, the set $C^{\I}$ is the union of the top left and bottom right rectangles. 

\subsection{Trace-positive connections}\label{sec:LJpositif}

Our extension of Le Jan's isomorphism theorem, that is, the specialisation of Theorem \ref{thm:LeJanSznitman} to the case where $f=0$, takes a particularly nice form for certain connections that we call \new{trace-positive}, and which are characterised by the fact that
\begin{equation} \label{eq:trpositive}
\forall \gamma \in \CL(\G), \ \Re\Tr(\hol_{h}(\gamma))\geq 0.
\end{equation}
For a trace-positive connection, \eqref{eq:defmuplh} defines $\mupl_{h}$ as a positive $\sigma$-finite measure, and the theorem, stated for the trivial splitting, takes the following form. 

\begin{theorem}\label{thm:LeJanSznitmanpositif} Let $\G$ be a weighted graph with a well, and $\F$ a vector bundle over $\G$. Let $h$ be a trace-positive connection on $\F$ and $\Phi$ be the Gaussian free vector field on $\G$ associated with $h$. Let $\L$ be a Poissonian ensemble in $\CL(\G)$ with intensity 
\[\frac{\beta}{2}\Re\Tr(\hol_{h}(\cdot))\; \d\mupl(\cdot),\]
independent of $\Phi$. Then 
\begin{equation}\label{LeJanSznitmanpositif}
\ell(\L)\build{=}_{}^{(d)}\frac{1}{2}(\|\Phi_{x}\|^{2})_{x\in \V\setminus\W}.
\end{equation}
\end{theorem}

Let us describe several ways of constructing trace-positive connections from an arbitrary real or complex bundle $\F$ of rank $r$ endowed with a connection $h$.  

The first way consists in taking the direct sum of $\F$ with the trivial bundle $\K^{r}$ of the same rank and endowed with the trivial connection. Then, denoting by $h\oplus \Id$ the connection on the bundle $\F\oplus \K^{r}$, we have, for any loop $\gamma$,
\[\Re\Tr(\hol_{h\oplus \Id}(\gamma))=\Re\Tr(\hol_{h}(\gamma)\oplus \Id_{\K^{r}})=\Re\Tr(\hol_{h}(\gamma))+r\geq 0.\]

A second way consists in taking the tensor product of $\F$ with its dual bundle, that is, considering the bundle $\End(\F)\simeq \F^{*}\otimes \F$. On this bundle, the connection $h$ induces by conjugation a connection $h^{*}\otimes h$, such that for all path $\gamma$ and for all $f\in \End(\F)_{\ul{\gamma}}=\End(\F_{\ul{\gamma}})$, 
\[\hol_{h^{*}\otimes h}(\gamma)(f)=\hol_{h}(\gamma)\circ f \circ \hol_{h}(\gamma)^{-1}.\]
The trace-positivity of the connection $h^{*}\otimes h$ follows from the fact that if $u$ is a unitary transformation of a Euclidean or Hermitian space $E$, then the trace of the linear transformation $f\mapsto ufu^{-1}$ of $\End(E)$ is $|\Tr(u)|^{2}$. Thus,
\[\Re\Tr(\hol_{h^{*}\otimes h}(\gamma))=\Tr(\hol_{h^{*}\otimes h}(\gamma))=|\Tr(\hol_{h}(\gamma))|^{2}\geq 0.\]
The bundle $\End(\F)$ splits as the direct sum $\End_{+}(\F)\oplus \End_{-}(\F)$ of the bundles of self-adjoint and skew-self-adjoint endomorphisms, each of which is stabilised by the connection $h^{*}\otimes h$. In the real case, the restrictions of $h^{*}\otimes h$ to these sub-bundles are not trace-positive, except in trivial cases where the sub-bundles have rank $1$. However, in the complex case, the restrictions to Hermitian and skew-Hermitians operators are trace-positive. In that case, the Gaussian free fields $\Psi$ on these sub-bundles are independent and the induced squared-norm fields $x\mapsto \frac{1}{2}\|\Psi_{x}\|^{2}$ have the same distribution which is equal to the occupation time of a Poissonian ensemble in~$\CL(\G)$ with intensity $\frac{1}{2}|\Tr(\hol_{h}(\cdot))|^2\; \d\mupl(\cdot)$. 

The bundle of Hermitians operators of a complex bundle was studied by Lupu \cite{Lupu-Dynkin}, who gave topological expansions, in terms of maps on surfaces and traces of holonomies of random paths, for the expression of the mixed moments of traces of squares of the random Hermitian endomorphisms at different vertices. His formulas generalise both the classical Br\'ezin--Itzykson--Parisi--Zuber formula (in the case of moments of the Gaussian unitary (random matrix) ensemble, corresponding to a graph with only one proper vertex) and the Dynkin isomorphism, thus unifying, in probabilistic terms, the approaches of 't Hooft and Symanzik.

\section{Symanzik identity}\label{sec:Symanzik}

In this last section, we investigate Symanzik's identity, which gives an expression of the moments of some non-Gaussian random sections of the vector bundle over our graph.

\subsection{Non-Gaussian random sections}

The random fields that one considers in Symanzik's identity are annealed versions of the Gaussian free vector field: they are Gaussian with respect to a random potential. In our covariant situation, we let not only the potential, but also the connection be random.
In the classical case, a further family of fields was considered in~\cite{BFS1} using Fourier transforms instead of Laplace transform, which make the probabilistic formulation less easy to state.

Let as usual $\G$ be a weighted graph with a well and $\F$ a vector bundle over $\G$. Recall that the set of connections on $\F$ is denoted by $\A(\F)$. Let $\H(\F)$ denote the set of potentials on $\F$, in the sense of Definition \ref{def:potentiel}. The space $\A(\F)$ is a compact topological space and $\H(\F)$ is a finite-dimensional vector space. We endow both of them with their Borel $\sigma$-field.

\begin{definition} Let $\PP$ be a Borel probability measure on $\A(\F)\times \H(\F)$. The measure $\P^{\PP}$ is the probability measure on $\Omega^{0}(\V\setminus \W,\F)$ such that for all bounded measurable function $F$ on $\Omega^{0}(\V\setminus \W,\F)$, 
\[\int_{\Omega^{0}(\V\setminus \W,\F)} F(f)\; \d \P^{\PP}(f)=\frac{1}{Z^{\PP}}\int_{\Omega^{0}(\V\setminus \W,\F)\times \A(\F)\times \H(\F)} F(f)e^{-\frac{1}{2}(f,\Delta_{g,K}f)_{\Omega^{0}}} \; \d\Leb_{\Omega^{0}}(f) \d\PP(g,K),\]
where
\begin{equation}\label{defZ}
Z^{\PP}=\pi^{\frac{\beta r |\V|}{2}} \int_{\A(\F)\times \H(\F)} \det \Delta_{g,K}^{-\frac{\beta}{2}}\; \d\PP(g,K).
\end{equation}
\end{definition}
If $\PP$ is the Dirac mass at $(h,H)$, then $\P^{\PP}=\P^{h,H}$. In that case, we denote $Z^{\PP}$ simply by $Z^{h,H}$, instead of $Z^{\delta_{(h,H)}}$. 

Let us emphasise that in general, $\P^{\PP}$ is not equal to $\int_{\A(\F)\times \H(\F)} \P^{g,K}\; \d\PP(g,K)$, but rather to
\begin{equation}\label{ZP}
\P^{\PP} = \frac{1}{Z^{\PP}}\int_{\A(\F)\times \H(\F)} Z^{g,K}\P^{g,K}\; \d\PP(g,K). 
\end{equation}

If one is given a reference connection $h$ and a reference potential $H$, then the definition of the measure $\P^{\PP}$ can be rewritten as
\begin{align*}
&\int_{\Omega^{0}(\V\setminus \W,\P)}F(f)\; \d\P^{\PP}(f)=\\
&\hspace{2.5cm}\frac{1}{Z^{\PP}}\int_{\Omega^{0}(\V\setminus\W,\F)} F(f) \Bigg[\int_{\A(\F)\times \H(\F)} e^{-\frac{1}{2}(f,(\Delta_{g,K}-\Delta_{h,H})f)_{\Omega^{0}}}\; \d\PP(g,K) \Bigg]\; \d\P^{h,H}(f),
\end{align*}
so that this measure appears as a perturbation of the distribution $\P^{h,H}$ of the Gaussian free vector field. If one is interested in a particular perturbation, then one should look for a probability measure $\PP$ which makes the expression between the brackets equal to this perturbation. This perturbative point of view was one of Symanzik's original motivations to prove the identity that bears his name.

Let us give another expression of the measure $\P^{\PP}$ based on a computation of the partition function $Z^{\PP}$. 

\begin{lemma} Let $\L$ be a Poissonnian ensemble of loops with intensity $\frac{r\beta}{2}\mupl$. For every connection~$g$ and every potential $K$, the following equality holds:
\[\Ex\Big[\prod_{\gamma\in\L} \tr(\hol_{g,K}(\gamma^{-1})) \Big]=\left(\frac{\det\Delta^{r}}{\det \Delta_{g,K}}\right)^{\frac{\beta}{2}}=(\pi^{-|\V|}\det\Delta)^{\frac{r\beta}{2}}Z^{g,K}.\]
\end{lemma}

\begin{proof} By Campbell formula, the leftmost quantity is equal to 
\[\exp \frac{\beta}{2}\int_{\CL(\G)} (\Tr(\hol_{g,K}(\gamma^{-1}))-r)\; \d\mupl(\gamma).\]
In order to apply \eqref{masterLJStr}, we need to understand the scalar $r$ as the trace of the holonomy of a connection along $\gamma$. Such a connection can be constructed by choosing a basis of each fibre of $\F$ and letting, for every vertex $x$ and every edge $e$ issued from $x$, the holonomy $h_{e,x}$ be the identity in the chosen bases of $\F_{x}$ and $\F_{e}$. For such a connection, the holonomy along every loop based at a vertex $x$ is the identity of $\F_{x}$. Moreover, the Laplacian associated to such a connection is conjugated to the Laplacian of the trivial bundle of rank $r$ endowed with the trivial connection. This accounts for the first equality. The second follows from \eqref{defZ}.
\end{proof}

It follows immediately from this result that for every connection $g$ and every potential $K$,
\begin{equation}\label{ZZ}
\frac{Z^{g,K}}{Z^{\PP}}=\frac{\Ex\Big[\prod_{\gamma\in\L} \tr(\hol_{g,K}(\gamma^{-1})) \Big]}{\EE\otimes \Ex\Big[\prod_{\gamma\in\L} \tr(\hol_{g',K'}(\gamma^{-1})) \Big]}.
\end{equation}

\subsection{Symanzik's identity}
We can now state the covariant version of Symanzik's identity. Recall the definition of the measure $\nu$ (Definition \ref{defnu}).

\begin{theorem}\label{thm:Symanzik} Let $\G$ be a weighted graph with a well. Let $\F$ be a vector bundle over $\G$. Let $\PP$ be a Borel probability measure on $\A(\F)\times \H(\F)$. 

Assume that $\K=\R$. Let $f_{1},\ldots,f_{2k}$ be $2k$ sections of $\F$. Then
\begin{align}\label{symanzikreel}
\Ex^{\PP}\left[\prod_{i=1}^{2k}(f_{i},\Phi)_{\Omega^{0}}\right]=&\\
\nonumber
&\hspace{-2cm} \frac{\displaystyle\EE\otimes \Ex\otimes \mbox{$\int_{\nu^{\otimes k}}$}\bigg[\prod_{\gamma\in\L} \tr(\hol_{g,K}(\gamma^{-1}))\sum_{\pi}\prod_{l=1}^{k}\lambda_{\ul{\gamma_{l}}} \langle f_{i_{l}}(\ul{\gamma_{l}}),\hol_{g,K}(\gamma^{-1})f_{j_{l}}(\ol{\gamma_{l}})\rangle_{\ul{\gamma_{l}}}\bigg]}{\displaystyle\EE\otimes \Ex\bigg[\prod_{\gamma\in\L} \tr(\hol_{g,K}(\gamma^{-1}))\bigg]},
\end{align}
where the sum is taken over all partitions $\pi=\{\{i_{1},j_{1}\},\ldots,\{i_{k},j_{k}\}\}$ of $\{1,\ldots,2k\}$ by pairs.

Assume that $\K=\C$. Let $f_{1},\ldots,f_{k},f'_{1},\ldots,f'_{k}$ be $2k$ sections of $\F$. Then
\begin{align}\label{symanzikcomplexe}
\Ex^{\PP}\left[\prod_{i=1}^{k}(f_{i},\Phi)_{\Omega^{0}}\prod_{i=1}^{k}(\Phi,f'_{i})_{\Omega^{0}}\right]=&\\
\nonumber
&\hspace{-4cm} \frac{\displaystyle\EE\otimes \Ex\otimes \mbox{$\int_{\nu^{\otimes k}}$}\Bigg[\prod_{\gamma\in\L} \tr(\hol_{g,K}(\gamma^{-1}))\sum_{\sigma}\prod_{l=1}^{k}\lambda_{\ul{\gamma_{l}}} \langle f_{l}(\ul{\gamma_{l}}),\hol_{g,K}(\gamma^{-1})f'_{\sigma(l)}(\ol{\gamma_{l}})\rangle_{\ul{\gamma_{l}}}\Bigg]}{\displaystyle\EE\otimes \Ex\bigg[\prod_{\gamma\in\L} \tr(\hol_{g,K}(\gamma^{-1}))\bigg]},
\end{align}
where the sum is over all permutations of $\{1,\ldots,k\}$.
\end{theorem}

An illustration of Theorem~\ref{thm:Symanzik} is given in Figure~\ref{fig:symanzik}.

\begin{figure}[ht!]
\includegraphics[width=.6\textwidth]{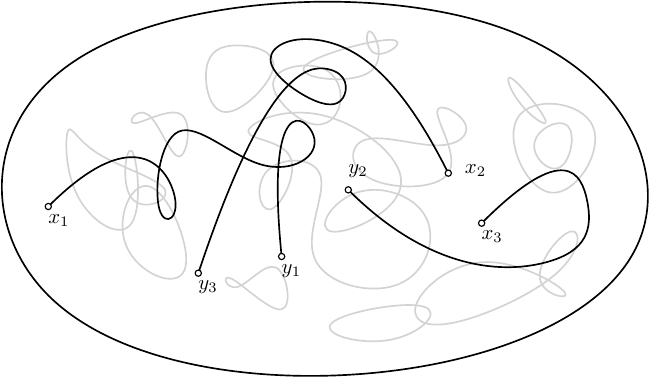}
\caption{\small According to Symanzik's identity, illustrated here in the complex case, the correlation of the random section of $\F$ at the points $x_1,x_2,x_3$ and its conjugate at the points $y_{1},y_{2},y_{3}$ can be computed in terms of the traces of the holonomies along loops of a Poissonnian ensemble (here in grey) and the holonomies along paths connecting the $x$ points to the $y$ points.
}\label{fig:symanzik}
\end{figure}

\begin{proof} The only difference between the real and complex cases is the form of Wick's theorem, which we stated in the present context as Proposition \ref{wick}. Let us treat the complex case. Writing the definition of $\P^{\PP}$ and using precisely Proposition \ref{wick}, we find that the left-hand side of \eqref{symanzikcomplexe} is equal to 
\[\frac{1}{Z^{\PP}}\sum_{\sigma} \int_{\A(\F)\times \H(\F)} Z^{g,K} \int_{\CP(\G)^{k}} \prod_{l=1}^{k}\lambda_{\ul{\gamma_{l}}} \langle f_{l}(\ul{\gamma_{l}}),\hol_{g,K}(\gamma^{-1})f'_{\sigma(l)}(\ol{\gamma_{l}}) \rangle_{\ul{\gamma_{l}}} \; \d\nu(\gamma_{1})\ldots \d\nu(\gamma_{k}) \; \d\PP(g,K).
\]
It suffices to replace $\frac{Z^{g,K}}{Z^{\PP}}$ by its value given by \eqref{ZZ} to find the right-hand side of \eqref{symanzikcomplexe}.
\end{proof}

\section*{Concluding remarks}

We view our paper as setting a framework for further study in random spatial processes on vector bundles over graphs. Some further aspects are studied in \cite{KL3,KL4}. As a conclusion, we mention without much detail a number of possible directions of future research.

We expect that our results should extend to the framework of Euclidean or Hermitian vector bundles over manifolds, at least for choices of smooth enough connections. In that case, random walk paths are replaced by Brownian paths on the manifold, and the continuum Gaussian free vector field is a random section of the vector bundle; moments of the twisted holonomies as well as moments of the Gaussian free vector field need to be renormalised by replacing them with appropriate Wick powers. In order to prove these statements, the existing covariant Feynman--Kac formulas \cite{Albeverio, Norris, Guneysu-continuum} will certainly be key.

Our setup should also allow to study the random interlacement model~\cite{Sznitman}, which can be investigated on an infinite lattice by using an exhausting sequence of finite graphs. We expect the analogue of Theorem~\ref{thm:LeJanSznitman} to be true in that case too.

Loop percolation is the study of connectivity properties of the loop soup. On hypercubic lattices, its phase transition with respect to an intensity parameter is now well understood. In our setup, there are many variants of loop percolation which could be defined and studied. In particular, since our parameter space is the space of connections, one can expect more complicated phase diagrams. 

Our definition of twisted holonomies provides a covariant analogue of local times. Cover times also have strong relations to Gaussian free fields. Recently, Ding and Li~\cite{Ding} and Zhai~\cite{Zhai} used the refinement of Le Jan's isomorphism theorem by Lupu~\cite{Lupu} (a refinement of which is given by Sznitman~\cite{SznitmanRefinedIsomorphism}) to prove results on cover times. Can one extend these results to our setup? What is the covariant analogue of the cover time and its relation to the Gaussian free vector field?

\def\@rst #1 #2other{#1}
\renewcommand\MR[1]{\relax\ifhmode\unskip\spacefactor3000 \space\fi
  \MRhref{\expandafter\@rst #1 other}{#1}}
\renewcommand{\MRhref}[2]{\href{http://www.ams.org/mathscinet-getitem?mr=#1}{MR#1}}

\bibliographystyle{hmralphaabbrv}
\bibliography{biblio-Symanzik}

\begin{thebibliography}{AHKHK89}

\bibitem[AHKHK89]{Albeverio}
S.~Albeverio, R.~Hoegh-Krohn, H.~Holden, and T.~Kolsrud.
\newblock {\em A covariant Feynman-Kac formula for unitary bundles over
  {E}uclidean space}, pages 1--12.
\newblock Springer Berlin Heidelberg, Berlin, Heidelberg, 1989.

\bibitem[AS18]{Abacherli-Sznitman}
A.~Ab\"{a}cherli and A.-S. Sznitman.
\newblock Level-set percolation for the {G}aussian free field on a transient
  tree.
\newblock {\em Ann. Inst. Henri Poincar\'{e} Probab. Stat.}, 54(1):173--201,
  2018.
\newblock \arXiv{1606.02411}. \MR{3765885}

\bibitem[BFS79]{BFS}
D.~Brydges, J.~Fr{\"o}hlich, and E.~Seiler.
\newblock On the construction of quantized gauge fields. {I}. {G}eneral
  results.
\newblock {\em Annals of Physics}, 121(1):227--284, 1979.

\bibitem[BFS82]{BFS1}
D.~Brydges, J.~Fr{\"o}hlich, and T.~Spencer.
\newblock The random walk representation of classical spin systems and
  correlation inequalities.
\newblock {\em Comm. Math. Phys.}, 83(1):123--150, 1982. \MR{648362
  (83i:82032)}

\bibitem[BFS83]{BFS2}
D.~C. Brydges, J.~Fr{\"o}hlich, and A.~D. Sokal.
\newblock The random-walk representation of classical spin systems and
  correlation inequalities. {II}. {T}he skeleton inequalities.
\newblock {\em Comm. Math. Phys.}, 91(1):117--139, 1983. \MR{719815
  (86i:81076)}

\bibitem[BHS19]{Helmuth}
R.~Bauerschmidt, T.~Helmuth, and A.~Swan.
\newblock The geometry of random walk isomorphism theorems.
\newblock 2019.
\newblock \arXiv{1904.01532}.

\bibitem[BP09]{Wick-q}
W.~Bryc and V.~Pierce.
\newblock Duality of real and quaternionic random matrices.
\newblock {\em Electronic Journal of Probability [electronic only]},
  14:452--476, 2009.

\bibitem[DL16]{Ding}
J.~Ding and L.~Li.
\newblock Chemical distances for level-set percolation of two-dimensional
  discrete {G}aussian free fields.
\newblock 2016.
\newblock \arXiv{1605.04449}.

\bibitem[Dyn80]{Dynkin-originel}
E.~B. Dynkin.
\newblock Markov processes and random fields.
\newblock {\em Bulletin of the American Mathematical Society}, 3(3):975--999,
  1980.

\bibitem[Dyn83]{Dynkin-Markov}
E.~B. Dynkin.
\newblock Markov processes as a tool in field theory.
\newblock {\em J. Funct. Anal.}, 50(2):167--187, 1983. \MR{693227 (85e:81074)}

\bibitem[Dyn84a]{Dynkin}
E.~B. Dynkin.
\newblock Gaussian and non-{G}aussian random fields associated with {M}arkov
  processes.
\newblock {\em J. Funct. Anal.}, 55(3):344--376, 1984. \MR{734803 (86h:60085a)}

\bibitem[Dyn84b]{Dynkin-localtime}
E.~B. Dynkin.
\newblock Local times and quantum fields.
\newblock In {\em Seminar on stochastic processes, 1983 ({G}ainesville, {F}la.,
  1983)}, volume~7 of {\em Progr. Probab. Statist.}, pages 69--83. Birkh\"auser
  Boston, Boston, MA, 1984. \MR{902412 (88i:60080)}

\bibitem[Eis95]{Eisenbaum}
N.~Eisenbaum.
\newblock Une version sans conditionnement du th\'eor\`eme d'isomorphisme de
  {D}ynkin.
\newblock In {\em S\'eminaire de {P}robabilit\'es, {XXIX}}, volume 1613 of {\em
  Lecture Notes in Math.}, pages 266--289. Springer, Berlin, 1995. \MR{1459468}

\bibitem[GMT14]{Guneysu}
B.~G{\"u}neysu, O.~Milatovic, and F.~Truc.
\newblock Generalized {S}chr{\"o}dinger semigroups on infinite graphs.
\newblock {\em Potential Analysis}, 41(2):517--541, 2014.

\bibitem[G{\"u}n10]{Guneysu-continuum}
B.~G{\"u}neysu.
\newblock The {F}eynman--{K}ac formula for {S}chr{\"o}dinger operators on
  vector bundles over complete manifolds.
\newblock {\em Journal of Geometry and Physics}, 60(12):1997--2010, 2010.

\bibitem[Jan97]{Janson}
S.~Janson.
\newblock {\em Gaussian {H}ilbert spaces}, volume 129 of {\em Cambridge Tracts
  in Mathematics}.
\newblock Cambridge University Press, Cambridge, 1997. \MR{1474726}

\bibitem[Ken11]{Kenyon}
R.~Kenyon.
\newblock Spanning forests and the vector bundle {L}aplacian.
\newblock {\em Ann. Probab.}, 39(5):1983--2017, 2011. \MR{2884879}

\bibitem[KK17]{Kassel-Kenyon}
A.~Kassel and R.~Kenyon.
\newblock Random curves on surfaces induced from the {L}aplacian determinant.
\newblock {\em Ann. Probab.}, 45(2):932--964, 2017. \MR{3630290}

\bibitem[KL19]{KL3}
A.~Kassel and T.~L\'evy.
\newblock Determinantal probability measures on {G}rassmannians.
\newblock 2019.
\newblock \arXiv{1910.06312}.

\bibitem[KL20a]{KL2}
A.~Kassel and T.~L\'evy.
\newblock A colourful path to matrix-tree theorems.
\newblock {\em Algebraic Combinatorics}, 3(2):471--482, 2020.
\newblock \arXiv{1903.02491}.

\bibitem[KL20b]{KL4}
A.~Kassel and T.~L\'evy.
\newblock {Quantum spanning forests}.
\newblock 2020.
\newblock In preparation.

\bibitem[KR16]{KR}
A.~Kassel and R.~Rosenthal.
\newblock {Isomorphism theorems and simplicial loop soups}.
\newblock 2016.
\newblock Unpublished notes.

\bibitem[Law18]{Lawler-survey}
G.~F. Lawler.
\newblock Topics in loop measures and the loop-erased walk.
\newblock {\em Probab. Surveys}, 15:28--101, 2018.

\bibitem[L{\'e}v19]{Levy-survey}
T.~L{\'e}vy.
\newblock Two-dimensional quantum {Y}ang--{M}ills theory and the
  {M}akeenko--{M}igdal equations.
\newblock 2019.
\newblock Proceedings of the 6th {S}trasbourg--{Z}urich meeting on frontiers in
  analysis and probability (to appear).

\bibitem[LJ08]{LeJan-1}
Y.~Le~Jan.
\newblock Dynkin's isomorphism without symmetry.
\newblock In {\em Stochastic analysis in mathematical physics}, pages 43--53.
  World Sci. Publ., Hackensack, NJ, 2008. \MR{2406020}

\bibitem[LJ10]{LeJan-2}
Y.~Le~Jan.
\newblock Markov loops and renormalization.
\newblock {\em Ann. Probab.}, 38(3):1280--1319, 2010. \MR{2675000}

\bibitem[LJ11]{LeJan-book}
Y.~Le~Jan.
\newblock {\em Markov paths, loops and fields}, volume 2026 of {\em Lecture
  Notes in Mathematics}.
\newblock Springer, Heidelberg, 2011.
\newblock Lectures from the 38th Probability Summer School held in Saint-Flour,
  2008, {\'E}cole d'{\'E}t{\'e} de Probabilit{\'e}s de Saint-Flour.
  [Saint-Flour Probability Summer School]. \MR{2815763}

\bibitem[LJ16]{LeJan16}
Y.~Le~Jan.
\newblock Markov loops, coverings and fields.
\newblock 2016.
\newblock \arXiv{1602.02708}.

\bibitem[LP19]{Lawler-Panov}
G.~F. Lawler and P.~Panov.
\newblock Weighted graphs and complex {G}aussian free fields.
\newblock {\em Electronic Communications in Probability}, 24, 2019.

\bibitem[Lup16]{Lupu}
T.~Lupu.
\newblock From loop clusters and random interlacements to the free field.
\newblock {\em Ann. Probab.}, 44(3):2117--2146, 2016. \MR{3502602}

\bibitem[Lup19]{Lupu-Dynkin}
T.~Lupu.
\newblock Topological expansion in isomorphisms with random walks for matrix
  valued fields.
\newblock 2019.
\newblock \arXiv{1908.06732}.

\bibitem[LW04]{Lawler-Werner}
G.~F. Lawler and W.~Werner.
\newblock The {B}rownian loop soup.
\newblock {\em Probab. Theory Related Fields}, 128(4):565--588, 2004.
  \MR{2045953}

\bibitem[LW16]{LupuWerner}
T.~Lupu and W.~Werner.
\newblock {A note on Ising random currents, Ising-FK, loop-soups and the
  Gaussian free field}.
\newblock {\em Electron. Commun. Probab.}, 21:7 pp., 2016.

\bibitem[MR06]{Marcus-Rosen}
M.~B. Marcus and J.~Rosen.
\newblock {\em Markov processes, {G}aussian processes, and local times}, volume
  100 of {\em Cambridge Studies in Advanced Mathematics}.
\newblock Cambridge University Press, Cambridge, 2006. \MR{2250510}

\bibitem[Nel73]{Nelson}
E.~Nelson.
\newblock The free {M}arkoff field.
\newblock {\em Journal of Functional Analysis}, 12(2):211--227, 1973.

\bibitem[Ner73]{Neruda}
P.~Neruda.
\newblock {\em El mar y las campanas}.
\newblock Editorial Losanda, Buenos Aires, 1973.

\bibitem[Nor92]{Norris}
J.~R. Norris.
\newblock A complete differential formalism for stochastic calculus in
  manifolds.
\newblock In {\em S\'eminaire de {P}robabilit\'es, {XXVI}}, volume 1526 of {\em
  Lecture Notes in Math.}, pages 189--209. Springer, Berlin, 1992. \MR{1231995}

\bibitem[QW15]{Qian-Werner}
W.~Qian and W.~Werner.
\newblock Decomposition of {B}rownian loop-soup clusters.
\newblock 2015.
\newblock \arXiv{1509.01180}.

\bibitem[ST16]{SabotTarres}
C.~Sabot and P.~Tarres.
\newblock Inverting {R}ay-{K}night identity.
\newblock {\em Probability Theory and Related Fields}, 165(3):559--580, 2016.

\bibitem[Sym69]{Symanzik}
K.~Symanzik.
\newblock Euclidean quantum field theory.
\newblock In {\em Scuola internazionale di Fisica ``Enrico Fermi'', XLV Corso},
  pages 152--223. Academic Press, 1969.

\bibitem[Szn12a]{Sznitman}
A.-S. Sznitman.
\newblock An isomorphism theorem for random interlacements.
\newblock {\em Electron. Commun. Probab.}, 17:no. 9, 9, 2012. \MR{2892408}

\bibitem[Szn12b]{Sznitman-book}
A.-S. Sznitman.
\newblock {\em Topics in occupation times and {G}aussian free fields}.
\newblock Zurich Lectures in Advanced Mathematics. European Mathematical
  Society (EMS), Z\"urich, 2012. \MR{2932978}

\bibitem[Szn13]{Sznitman-continuum}
A.-S. Sznitman.
\newblock On scaling limits and {B}rownian interlacements.
\newblock {\em Bull. Braz. Math. Soc. (N.S.)}, 44(4):555--592, 2013.
  \MR{3167123}

\bibitem[Szn16]{SznitmanRefinedIsomorphism}
A.-S. Sznitman.
\newblock Coupling and an application to level-set percolation of the
  {G}aussian free field.
\newblock {\em Electron. J. Probab.}, 21:26 pp., 2016.

\bibitem[Wer16]{Werner}
W.~Werner.
\newblock On the spatial {M}arkov property of soups of unoriented and oriented
  loops.
\newblock In {\em S{\'e}minaire de Probabilit{\'e}s XLVIII}, pages 481--503.
  Springer, 2016.
\newblock \arXiv{1508.03696}.

\bibitem[Zha18]{Zhai}
A.~Zhai.
\newblock Exponential concentration of cover times.
\newblock {\em Electronic Journal of Probability}, 23, 2018.
\newblock \arXiv{1407.7617}.

\end{thebibliography}

\end{document}